\documentclass[hidelinks]{siamart220329}

\usepackage{lipsum}
\usepackage{amsfonts}
\usepackage{graphicx}
\usepackage{epstopdf}
\usepackage{algorithmic}
\ifpdf
  \DeclareGraphicsExtensions{.eps,.pdf,.png,.jpg}
\else
  \DeclareGraphicsExtensions{.eps}
\fi

\newsiamremark{remark}{Remark}
\newsiamremark{hypothesis}{Hypothesis}
\crefname{hypothesis}{Hypothesis}{Hypotheses}
\newsiamthm{claim}{Claim}

\headers{Convergent Incremental Potential Contact}{Li et al.}

\title{Convergent Incremental Potential Contact\thanks{Submitted to arXiv \today.%
\funding{This work was funded by the Fog Research Institute under contract no.~FRI-454.}}}

\author{Minchen Li\thanks{Department of Mathematics, University of California, Los Angeles (\email{minchernl@gmail.com}, \email{cffjiang@math.ucla.edu}).}
\and Zachary Ferguson\thanks{Courant Institute of Mathematical Science, New York University (\email{zfergus@nyu.edu}, \email{dzorin@cs.nyu.edu}, \email{panozzo@nyu.edu}).}
\and Teseo Schneider\thanks{Department of Computer Science, University of Victoria (\email{teseo@uvic.ca}).}
\and Timothy Langlois\thanks{Adobe Research (\email{tlangloi@adobe.com}, \email{kaufman@adobe.com}).}
\and Denis Zorin\footnotemark[3]
\and Daniele Panozzo\footnotemark[3]
\and Chenfanfu Jiang\footnotemark[2]
\and Danny M. Kaufman\footnotemark[5]}

\usepackage{amsopn}

\usepackage{amssymb}
\usepackage{siunitx}
\usepackage{amsmath}
\usepackage{wrapfig}
\usepackage{lineno}

\newcommand{\citetdoyen}[0]{Doyen et al.~\cite{doyen2011time}}
\newcommand{\citetIPC}[0]{Li et al.~\cite{Li2020IPC}}

\PassOptionsToPackage{dvipsnames}{xcolor}

\newcommand{\R}[0]{\mathbb{R}}

\DeclareMathOperator*{\argmin}{\arg\!\min}

\newcommand{\bfi}[1]{\textbf{\textit{#1}}}

\newcommand{\ct}[0]{t}
\newcommand{\cs}[0]{s}

\newcommand{\excluderadius}[1]{\setminus_{#1}}

\usepackage{mathtools}
\newcommand{\defeq}{\vcentcolon=}

\newcommand{\dhat}{\ensuremath{\hat{d}}}
\newcommand{\Vint}{\ensuremath{V_\text{int}}}
\newcommand{\Eint}{\ensuremath{E_\text{int}}}

\usepackage{xspace}
\newcommand{\EIPC}{EIPC\xspace} %
\newcommand{\acceptable}{acceptable\xspace}

\definecolor{darkred}{rgb}{0.8, 0.0, 0.0}
\definecolor{pinkish}{rgb}{0.87, 0.36, 0.51}
\definecolor{darkgreen}{rgb}{0.0, 0.5, 0.0}
\definecolor{violet}{rgb}{0.341, 0.024, 0.549}
\definecolor{Orange}{rgb}{1, 0.6, 0}

\ifpdf
\hypersetup{
  pdftitle={Convergent IPC},
  pdfauthor={Li et al.}
}
\fi

\begin{document}

\maketitle

\begin{abstract}
Recent advances in the simulation of frictionally contacting elastodynamics with the Incremental Potential Contact (IPC) model have enabled inversion and intersection-free simulation via the application of mollified barriers, filtered line-search, and optimization-based solvers for time integration. In its current formulation the IPC model is constructed via a discrete constraint model, replacing non-interpenetration constraints with barrier potentials on an already spatially discretized domain. However, while effective, this purely discrete formulation prohibits convergence under refinement. To enable a convergent IPC model we reformulate IPC potentials in the continuous setting and provide a first, convergent discretization thereof. We demonstrate and analyze the convergence behavior of this new model and discretization on a range of elastostatic and dynamic contact problems, and evaluate its accuracy on both analytical benchmarks and application-driven examples.
\end{abstract}



\section{Introduction}

The robust and accurate modeling of large-deformation frictionally contacting elastodynamics remains a challenging problem in simulation. The recently proposed Incremental Potential Contact (IPC) model\ \cite{Li2020IPC} enables inversion and intersection-free simulation of contacting elastodynamics via the application of mollified barriers, filtered line-search, and optimization-based solvers for time integration. As originally formulated, IPC begins with a discrete model, replacing non-interpenetration constraints with locally-supported $C^2$ barrier potentials, on an already spatially discretized domain. Each barrier potential, in turn, evaluates unsigned distances between boundary mesh-primitive pairs to obtain intersection-free trajectories for complex multibody simulations where domains can have arbitrarily sharp geometries and undergo large deformation. Subsequent work has extended the IPC model to solve problems in rigid and multibody dynamics\ \cite{Ferguson2021Rigid,Lan2022ABD,Chen2022MultibodyIPC}, codimensional simulation of shells and rods\ \cite{li2021codimensional}, subspace modeling \cite{Lan2021Medial}, embedded interfaces\ \cite{choo2021barrier}, viscoelasticity and elastoplasticity\ \cite{Li2022ECI}, and coupled MPM-FEM modeling\ \cite{Li2022BFEMP}.

However, while effective, the original IPC model's purely discrete formulation prohibits convergence under refinement. To enable a convergent IPC model we reformulate IPC potentials in the continuous setting and provide a first, convergent discretization thereof. We focus on providing a consistent frictional contact model that converges under refinement of discretization, while retaining the original IPC model's non-interpenetration and global convergence properties for both highly refined and coarse models --- regardless of problem complexity. To do so we re-derive contact barrier and dissipative friction, beginning from a continuous formulation while addressing short-comings in the original handling of the max operator and so physical forces.

\section{Contact Model}

We begin construction of our contact model with the barrier energy,
\[
b(d, \hat{d})=
\Bigg\{\begin{array}{lr}
-\kappa \left(\frac{d}{\hat{d}} - 1\right)^2\ln \left(\frac{d}{\hat{d}}\right), & 0<d<\hat{d} \\
0 & d \geq \hat{d},
\end{array}
\]
evaluated on unsigned distances $d$, with a stiffness parameter $\kappa$ in units of \si{\pascal\cdot\meter}. Note that for larger distances the barrier is not active, it then activates when $d$ decreases below the small activation threshold $\hat{d}>0$. The barrier diverges as the distances tend to zero, preventing interpenetrations.

Assigning the barrier to distances $d(x_1,x_2) = \|x_1-x_2\|$ evaluated between material point pairs $x_1$ and $x_2$, we define the corresponding smoothed, contact energy density
$$b\big(d(x_1,x_2), \hat{d} \big).$$
For contact between a point $x \in \R^d$ and a curve $c(s):[0,1] \rightarrow \mathbb{R}^{d}$ the barrier is likewise defined by the unsigned distance between the point and curve  
$$d(x,c) = \min_{s\in [0,1]} d\big(x, c(s) \big).$$
This gives us a corresponding point-to-curve barrier energy 
$$b(d(x,c), \hat{d}),$$
or equivalently, as $b$ is monotonically decreasing, we can define it as
$$\max_{s\in [0,1]} b\big( d\big(x, c(s)\big),\hat{d} \big).$$

The contact energy for a curve $c_1$, with respect to a curve $c_2$, is then 
$$ P_{\to c_2}(c_1) = \int_{\ct \in [0,1]} \left(\max_{\cs \in [0,1]} b\left( d\big(c_1(\ct), c_2(s) \big),\hat{d} \right) \right) d\ct.$$
We correspondingly define the total contact potential between the two curves as  
\[P(c_1,c_2) = \tfrac{1}{2}\Big( P_{\to c_2}(c_1) + P_{\to c_1}(c_2) \Big).\]

Remark: in the limit, as $\hat{d}\rightarrow 0$, the energies $P_{\to c_2}(c_1)$ and $P_{\to c_1}(c_2)$ equivalently measure the total contact potential between the curves $c_1$ and $c_2$ while, for finite $\hat{d}$, we take their average.

With deformation we must consider self-contact. For a single curve $c$ the self-contact energy is then 
\[
\frac{1}{2} \int_{\ct \in [0,1]} \left(\max_{\cs \in [0,1] \excluderadius{r} \ct} b\left( d\big(c(\ct), c(\cs) \big), \hat{d} \right) \right) d\ct,
\]
where we define $\excluderadius{r}: \{x~|~a \leq x \leq b\} \times \R \times \R \mapsto \mathcal{P}(\R)$ as 
\[
[a,b] \excluderadius{r} \ct \defeq \{s \in [a, b]~|~|t - s| > r\},
\]
with $r\rightarrow0$.

With self-contact defined we are now able to directly generalize the contact potential to an arbitrary number of curves $\mathcal C = \{ c_i\}$ by treating all contacts in the domain as \emph{self-contacts}. We first (re)parameterize the domains across all curves in $\mathcal C$ with $s \in [0,1]$ so that $\mathcal C(s)$ traverses the material points across all curves $\{ c_i\}$ contiguously. The total contact potential is then simply
\[
P(\mathcal C) = \frac{1}{2} \int_{\ct\in [0,1]} \left(\max_{s\in [0,1] \excluderadius{r} \ct} b\left( d\big(\mathcal C(\ct), \mathcal C(\cs) \big),\hat{d} \right) \right) d\ct.
\]

For contact in $\R^3$, we extend our barrier potential to a set of surfaces $\mathcal S = \{S_i\}$.  We parameterize these surfaces by common (possibly discontinuous) coordinates $u \in \tilde M \subset \R^2$, so that $\mathcal S(u)$ traverses the material points across all surfaces $\{ S_i\}$ contiguously. The total contact potential is then 
\begin{align}
\label{eq:cont-3D-barrier=potential}
P(\mathcal S) = \frac{1}{2} \int_{u \in \tilde M} \left(\max_{v \in \tilde M \excluderadius{r} u} b\left( d\big(\mathcal S(u), \mathcal S(v) \big),\hat{d} \right) \right) du,
\end{align}
where we overload the operator $\excluderadius{r}: \mathcal{P}(\R^2) \times \R \times \R^2 \mapsto \mathcal{P}(\R^2)$ to be 
\[
\tilde{M} \excluderadius{r} u \defeq \{v \in \tilde{M}~|~\|u-v\|_2 > r\},
\]
with $r\rightarrow0$.

Remark: within an infinitesimal region around each point at coordinate $u$ (respectively $\ct$), we do not resolve contact; this region is empty in the limit. In turn this requires that $\hat{d}/r\rightarrow 0$ to ensure that the contact potential will not diverge for all configurations. In the discrete setting, with finite $\hat{d}$ and spatial mesh resolution, this requirement simplifies.

\section{Friction}

Frictional contact adds contact-driven dissipative forcing that opposes sliding. The magnitude and direction of these frictional forces, generated across contacting codomains are determined by choice of a friction model that is, in turn, parameterized by the sliding velocity field, the normal pressures exerted by contact, and the frictional coefficient between the contacting codomains. 

We model friction via the Maximal Dissipation Principle~\cite{moreau73unilateral} which posits frictional forces maximize the rate of dissipation in sliding up to a maximum magnitude imposed by a limit surface; e.g., Coulomb's constraint.

For contacts $k$ formed between any two surface (alternately curve) points $x_1$ and $x_2$, with corresponding velocities $\dot{x}_1$ and $\dot{x}_2$, we extract the sliding velocity as 
$$v_k = P(x_1,x_2)(\dot{x}_1 - \dot{x}_2),$$ 
where the sliding projection is $P = T(x_1,x_2) T(x_1,x_2)^T \in \R^{3 \times 3}$ with $T(x_1,x_2) \in \R^{3 \times 2}$ constructed from the unit column vectors orthogonal to $x_1-x_2$. %

Maximizing dissipation rate subject to the Coulomb constraint defines friction forces $f_k \in \R^3$ applied at $k$ as
\begin{align}
\begin{split}
f_k = \argmin_{f} \> f^T v_k \> \> \text{s.t.} \> \> \|f\| \leq \mu_k \lambda_k,
   \end{split}
 \label{eq:frictionForceDef}
\end{align}
with $\mu_k$ the local frictional coefficient and $\lambda_k$ the magnitude of the normal force exerted by the contact barrier between points $x_1$ and $x_2$.
Equivalently we have 
\begin{align}
\begin{split}
f_k \in -\partial F_k(v_k),
   \end{split}
 \label{eq:frictionPotential}
\end{align}
with a nonsmooth energy  
$$F_k(v_k) = \mu_k \lambda_k \| v_k \|$$
encoding the transitions between sticking and sliding behaviors corresponding to the varying active sets of \cref{eq:frictionForceDef}.

\begin{wrapfigure}{r}{0.3\textwidth}
    \centering
    \includegraphics[width=\linewidth]{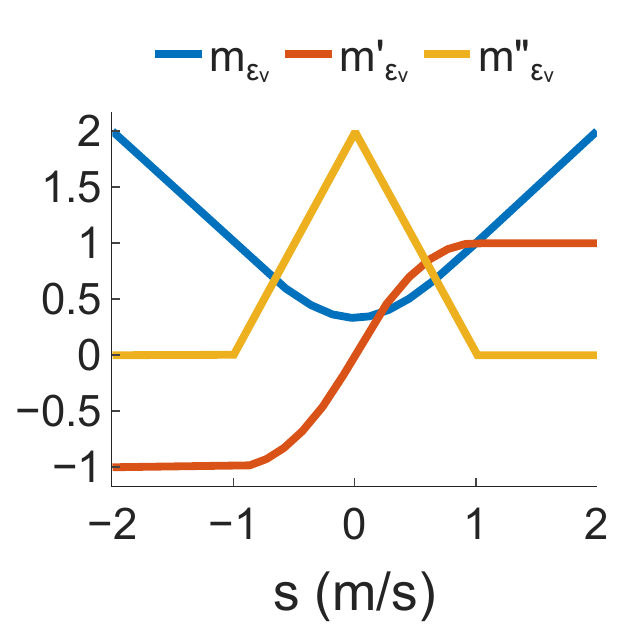} 
    \vspace*{-15pt}
    \caption{Plot of \cref{eq:friction_mollifier} and its derivatives ($\epsilon_v = 1$ for better visualization).}
    \label{fig:friciton_mollifier}
\end{wrapfigure}

\subsection{Smoothed friction}

We formulate friction with a smooth dissipative pseudo-potential. We start by mollifying $F_k$ with 
$$D_k(v_k) = \mu \lambda_k m_{\epsilon_v}(\|v_k\|), $$
where 
\begin{equation}
\label{eq:friction_mollifier}
m_{\epsilon_v}(s) = 
    \begin{cases} 
      -\frac{s^3}{3\epsilon_v^2} +\frac{s^2}{\epsilon_v} + \frac{\epsilon_v}{3}, & s < \epsilon_v \\
      s, & s \geq \epsilon_v 
   \end{cases}    
\end{equation}
approximates the jump conditions (see \cref{fig:friciton_mollifier}) and $\epsilon_v$ (in units of \si{\meter/\second}) defines the range below which small sliding velocities are resolved as static. Friction forces are then $$f_k = -\nabla D_k(v_k),$$ and better approximate the nonsmooth slip-stick transitions as $\epsilon_v \to 0$.

\subsection{Integration}

For a pair of \emph{contacting} points $x_1$ and $x_2$ the dissipative energy is then 
$$D(\dot x_1, \dot x_2, x_1, x_2) = \mu \ \lambda(x_1, x_2)  \ m_{\epsilon_v}\big( \| P(x_1,x_2)\ (\dot x_1 - \dot x_2)\| \big),$$ where 
$$\lambda(x_1, x_2) = -\frac {\partial b\big(d(x_1, x_2)\big)}{\partial d},$$ is the (positive) contact force magnitude between points $x_1$ and $x_2$, and the friction force applied between points $x_1$ and $x_2$ is correspondingly 
$$f(\dot x_1, \dot x_2, x_1, x_2) = - \nabla_{(\dot x_1, \dot x_2)} D(\dot x_1, \dot x_2, x_1, x_2).$$

We next parameterize collections of curves $\mathcal C = \{c_i\}$ and surfaces $\mathcal S = \{S_i\}$ in both space (respectively $s \in [0,1]$ and  $u \in \tilde M \subset \R^2$) and time $\ct \in \R$. In the following we continue to reserve overdots for \emph{time} derivatives and, unless needed, do not explicitly include time parameters, e.g. we have $c(s) = c(s,\ct)$ and $\dot c(s) = \dot c(s,\ct) = \partial c(s,\ct) / \partial \ct$. 
 Then, the total, dissipative friction potential for a system of curves is 
 \begin{align*}
 D(\dot{\mathcal C}, \mathcal C) = \frac{1}{2} &\int_{s\in [0,1]} D\Big(\dot{\mathcal  C}(s), \dot{\mathcal C}(\ell[s]), \mathcal C(s), \mathcal C(\ell[s]) \Big) \ ds \> \text{with}\\
 &\ell[s] = \argmin_{u\in [0,1] \excluderadius{r} s} d\Big(\mathcal C(s), \mathcal C(u) \Big),
 \end{align*}
 while the corresponding dissipative friction potential for surfaces is
  \begin{align*}
 D(\dot{\mathcal S}, \mathcal S) = \frac{1}{2} & \int_{u \in \tilde M} D\Big(\dot{\mathcal S}(u), \dot{\mathcal S}(\ell[u]), \mathcal S(u), \mathcal S(\ell[u]) \Big) \ du \> \text{with}\\
 &\ell[u] = \argmin_{v \in \tilde M \excluderadius{r} u} d\Big(\mathcal S(u), \mathcal S(v) \Big).
 \end{align*}

\section{Contact Spatial Discretization}
\label{sec:contact-discretization}

In the discrete setting we apply piecewise linear, compatible discretizations of curves with edges, areas (2D) and surfaces (3D) with triangles, and volumes (3D) with tetrahedra. Here we have two tasks. First, we integrate the contact potentials over boundary (edge and triangle) elements and second, we smoothly approximate the $\max$ operator in these integrals so that we can efficiently solve the resulting nonlinear problems (see \cref{sec:numerical-solution}) for statics and dynamic time-stepping with second-order, Newton-type methods. 

In the following we denote the computational mesh for contact and friction potentials over the set of \emph{boundary} vertices $V$, \emph{boundary} edges $E$, and \emph{boundary} triangles (3D) $T$. 

\subsection{Discretization and Numerical Integration}

We begin by defining our curve discretization with polyline geometry. As in the smooth case we can parameterize the domain across all polylines with $s \in [0,1]$ so that $p(s):[0,1] \rightarrow \mathbb{R}^{d}$ traverses all material points, across all edges $e \in E$ in the polylines contiguously. The corresponding curve contact potential is then 
\begin{align}
    \frac{1}{2}\int_{s\in [0,1]} \left(\max_{e\in E\setminus p(s)} b\big( d(p(s), e),\hat{d}\big) \right) \ ds,
    \label{eq:pw_linear_curve}
\end{align}
where $E\setminus p$ is the set of boundary edges that do not contain the point $p$.

Applying polyline vertices as nodes (and quadrature points), we numerically integrate the curve contact potential. For each nodal position $x \in V$ we then have a corresponding material space coordinate $\bar x \in \bar V$. Piecewise linear integration of the curve barrier is then 
\[ \frac{1}{2}\sum_{\bar x \in \bar V} w_{\bar x} \left(\max_{e\in E\setminus x(\bar x)} b\Big( d\big(x(\bar x), e\big),\hat{d} \Big) \right), \]
where $w_{\bar x}$ are the quadrature weights, each given by half the sum of the lengths (in material space) of the two boundary edges incident to ${\bar x}$.
Correspondingly, following the same steps, piecewise-linear integration of the surface barrier over a triangulated boundary mesh in 3D gives the surface contact potential
\begin{align}
    \frac{1}{2} \sum_{\bar x \in \bar V}  w_{\bar x} \left(\max_{\ct \in T\setminus x(\bar x)} b\Big( d\big(x(\bar x), \ct\big),\hat{d} \Big) \right),
    \label{eq:pw_linear_surface}
\end{align}
where $T\setminus x$ is the set of boundary faces that do not contain $x$, and $w_{\bar x}$ are the quadrature weights, each given by one third of the sum of the areas (in material space) of the boundary triangles incident to ${\bar x}$.

\subsection{Smoothly Approximating the Max Function}
Our next step is to smoothly approximate the $\max$ operator in the contact potentials.
A natural option to consider would be a softmax approximator. Using the discrete curve energy as a concrete example we could smooth the max operator with the $p$-norm as
\begin{align*}
    \frac{1}{2} \sum_{\bar x \in \bar V}  w_{\bar x} \left(\sum_{e \in E\setminus x(\bar x)} b\Big( d\big(x(\bar x), e\big),\hat{d} \Big)^p \right)^{1/p}.
\end{align*}
We note, however, that this would significantly decrease sparsity in subsequent numerical solves by increasing stencil-size per contact evaluation.  At the same time, accuracy would require large $p$ and so additional ill-scaling also contributing to significant increase in numerical solver costs. Similar issues arise for LogSumExp approximation, while sparsity increase could be addressed in the $p$-norm formulation, by dropping the outer normalizing $1/p$ term, but doing so would generate increasingly ill-scaled and ill-conditioned problems with the necessary further increases in $p$ required for improved accuracy in the approximation.

We leave investigations of these above approximators to future work. Here, to smoothly approximate the barrier energies with accuracy \emph{and} computational efficiency we directly consider the boundary geometry. We begin with an evaluation of the max barrier at a point $x$ where at least one edge in $E\setminus x$ (respectively triangle in $T\setminus x$) is closer than $\hat d$. If we consider a rough \emph{starting} approximation by summation over all barriers between $x$ and nonincident boundary elements, clearly one of the nonzero summands corresponds to the desired max. In the lucky case, when all other boundary elements are farther than $\hat d$, this approximation is sharp. However, in cases where more than one boundary element is within the $\hat d$ distance of $x$ this approximation will overestimate the max barrier with the incorrect addition of undesirable barriers from these close-by elements. We require $\hat d$ small with respect to edge length. Then, when this set of ``close-enough'' boundary elements form a convex curve (respectively surface) w.r.t. the evaluation point they each contribute an extra barrier contribution that is exactly evaluated by a distance to a vertex (respectively edges) incident to the closest boundary edge (respectively boundary triangle). 
Likewise, when they form a nonconvex curve (respectively surface) the extraneous barrier contributions are lower-bounded by barriers evaluated with vertices (respectively edges) incident to the closest edge (respectively triangle).

\begin{figure}[t]
    \centering
    \includegraphics[width=0.3\linewidth]{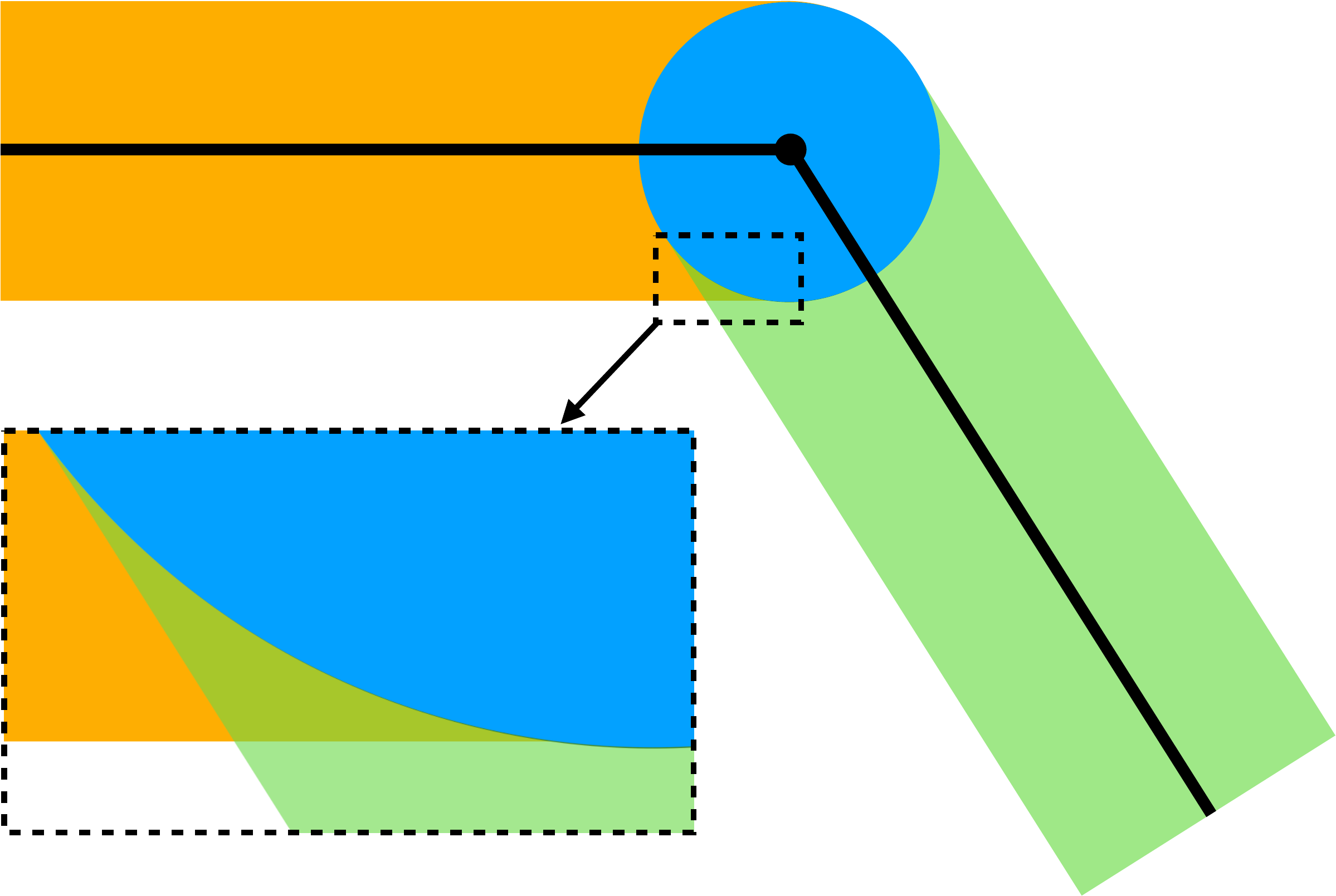}
    \caption{In this simple two-edge illustration, the yellow and green regions are only counted once by the summation, but the blue region and the yellow-green overlap are counted twice. If we subtract once the blue region, then for the right-top boundary (convex), it becomes perfect, but for the left-bottom boundary (concave), we can still see some overlap that are counted twice.}
    \label{fig:corner}
\end{figure}

Our resulting approximators for curves and surfaces are then respectively  
\begin{align*}
\Psi_c(x) &= \sum_{e\in E\setminus x} b( d(x, e),\hat{d}) \quad - \sum_{x_2\in V_{int}\setminus x}b( d(x, x_2),\hat{d}) \> \approx \max_{e\in E\setminus x} b( d(x, e),\hat{d}) ,
\end{align*} 
and
\begin{align*}
\Psi_s(x) &= \sum_{t\in T\setminus x} b( d(x, t),\hat{d}) 
\quad - \sum_{e\in E_{int}\setminus x} b( d(x, e),\hat{d})
\quad + \sum_{x_2\in V_{int}\setminus x}b( d(x, x_2),\hat{d}) \\
&\approx \max_{t\in T\setminus x} b( d(x, t),\hat{d}),
\end{align*}
where $\Vint \subseteq V$ is the subset of internal curve/surface nodes (e.g., vertices with valence two for curves) and $\Eint \subseteq E$ is the subset of internal surface edges (i.e., edges incident to two triangles). 
For locally convex regions this estimator is tight while remaining smooth. In turn, for nonconvex regions it improves over direct summation (see \cref{fig:corner,fig:dist_field_comp}). 

\begin{figure}[t]
    \centering
    \includegraphics[width=0.75\linewidth]{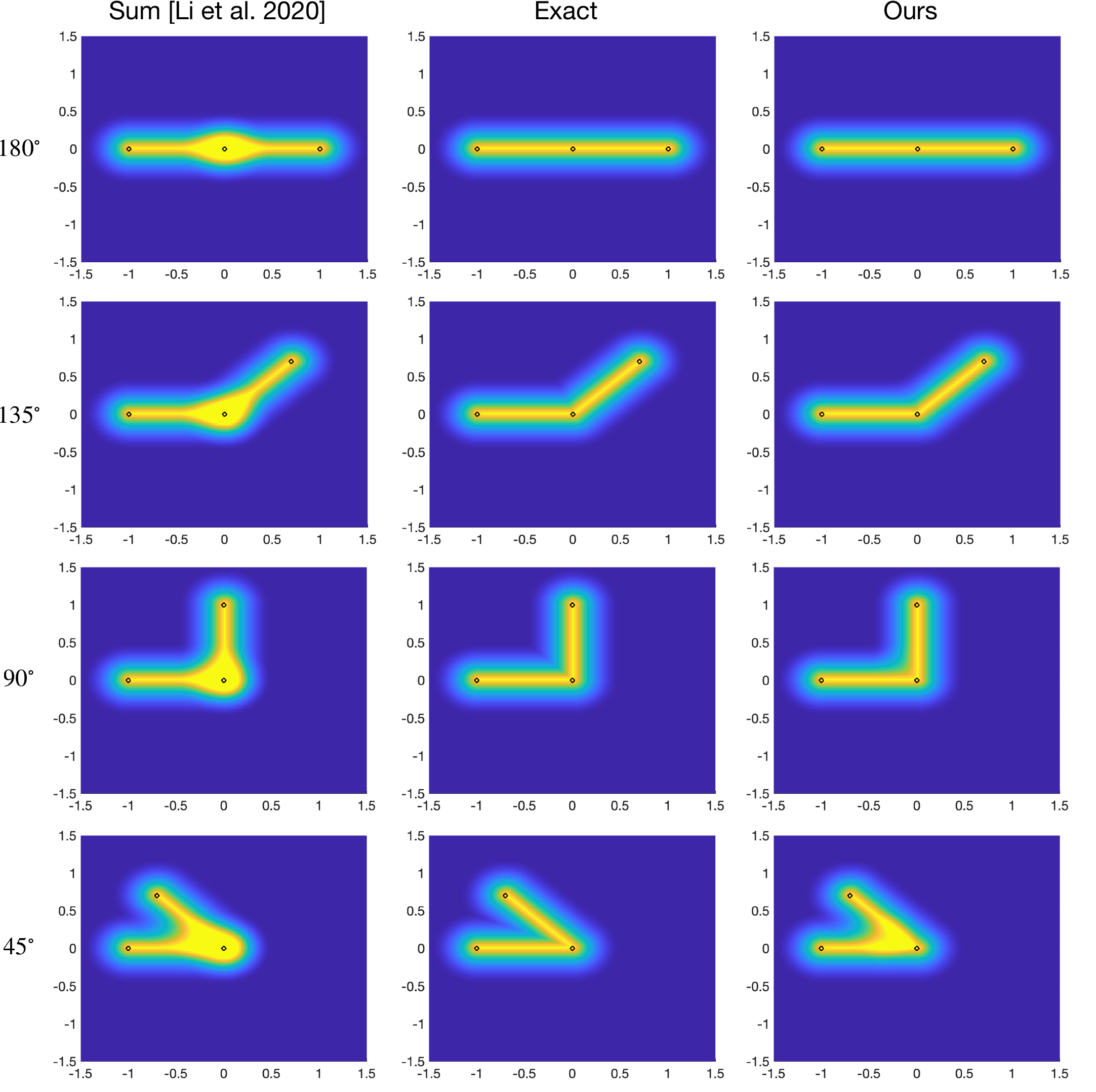}
    \caption{Here we visualize the contact energy field with different approximations to the $\max$ operator. The energy visualized is simplified to $(d-\hat{d})^2$ when $d < \hat d$ with $\hat{d}=0.5$ for better visibility.
    }
    \label{fig:dist_field_comp}
\end{figure}

The corresponding discrete barrier potentials are then simply
$$P_c(V) = \sum_{x \in V} w_x \Psi_c(x),$$ 
for curves, and 
$$P_s(V) = \sum_{x \in V} w_x \Psi_s(x),$$ for surfaces, where we simplify with $w_x = w_{\bar x}$ defined appropriately, per domain, as covered above.

Remark: For diminishing $\hat{d}$, the (naive) direct summation over all contact pairings, without correction (e.g., $\sum_{t\in T\setminus x} b( d(x, t),\hat{d})$), becomes an increasingly good approximation to the $\max$-based barrier energy density as the number of extraneous (not closest but still within $\hat d$ distance) surface elements becomes smaller and smaller.

\subsection{Edge Quadrature}

With the above discretization, we ensure that modeled geometries will remain free of point-edge and point-triangle interpenetrations independent of the resolution applied. This is because we have defined our barriers on point-edge and point-triangle pairings.   Likewise, as we show in \cref{sec:3D-convergence} this discretization converges under refinement so that edge-edge intersections also vanish. However, in many practical applications, it is desirable to prevent edge-edge intersections even at moderate resolution. To do so we additionally construct an alternate, edge-based quadrature of the surface barrier.

With an edge quadrature we discretize \cref{eq:cont-3D-barrier=potential} with
\[\frac{1}{2}\sum_{\bar{e}\in E} w_{e} (\max_{t\in T\setminus \bar{e}} b( d(e, t),\hat{d})).\]
Here $w_{e}$ are the quadrature weights, each given by one-third of the sum of the areas (in material space) of the boundary triangles incident to ${e}$, using edge-surface distance to approximate the average of the point-surface distances for surface points near each edge. 
For efficiency, we then further approximate edge-surface distances in our barriers with edge-edge distances,
\[\frac{1}{2}\sum_{e \in E} w_{e} (\max_{e_2\in E\setminus e} b( d(e, e_2),\hat{d})),\]
where $E \setminus e$ is the set of surface edges non-adjacent to $e$.
Approximating $\max$ with our summation we then obtain the edge-edge barriers 
\begin{align*}
    \Psi_w(e) &= \sum_{e_2 \in E\setminus e} b(d(e, e_2), \hat{d}) - \sum_{x_1\in V\setminus e} \> \big(\rho(x_1)-1 \big) \> b( d(x_1, e),\hat{d}) \\ 
    & \approx \max_{e_2\in E\setminus e} b( d(e, e_2),\hat{d}),\\
\end{align*}    
where $\rho(x_1)$ gives the number of surface edges incident to node $x_1$, and $V \setminus e$ is the set surface nodes that are not incident to edge $e$. 
The corresponding wireframe edge-barrier potential is then    
\begin{align*}    
    P_w(E) &= \frac{1}{2}\sum_{e\in E} w_{e} \Psi_w(e).
\end{align*}

\subsubsection{Combined Discretization}

When desired, to ensure complete non-in\-tersection of 3D surfaces, independent of discretization resolution, we then jointly employ the above edge-based quadrature in combination with our node-based quadrature. The total, combined contact potential for surfaces is then 
\[\alpha P_s(V) + (1-\alpha)P_w(E),\]
where $\alpha \in (0, 1]$ (averaging with $\alpha = 1/2$ is sufficient) so that under combined spatial ($x$) and distance ($\hat{d}$) refinement both energies converge. Alternately when small edge-edge intersections are acceptable for modeling errors we can apply solely the node-based potential $P_s(V)$.

\subsection[Positivity of Phi\_c, Phi\_s, and Phi\_w]{Positivity of $\Phi_c$, $\Phi_s$, and $\Phi_w$}
\label{sec:positivity}

Modeling contact via barrier representation requires a sufficiently small $\dhat$. Here, for the purposes of demonstrating positivity of our barrier energies, we  further define an acceptably small scale for $\dhat$:
\begin{definition}
    For any point $x$, we call $\dhat$ \acceptable{} if, every connected component of the intersection of ball $B_{\dhat}^x$ of radius $\dhat$ centered on $x$ with the boundary contains at most one vertex.
\end{definition}
We remark that, if the rest mesh is not in contact, then $\dhat$ is \acceptable. However, $\dhat$ must be \acceptable for every time step, which can be achieved by shrinking.

Before showing the positivity of the contact barriers, we settle on a few simple statements.
\begin{remark}\label{rem:edge-bound}
    Let $x_1$ and $x_2$ be the endpoints of an edge $e$; for any point $x$ in two and three dimension, $d(x,e)\le \min(d(x,x_1), d(x,x_2))$,
\end{remark}
\begin{remark}\label{rem:tri-bound}
    Let $e_1$, $e_2$ and $e_3$ be the edges of a triangle $t$; for any point $x$ in two and three dimension, $d(x,t)\le \min(d(x,e_1),d(x,e_2), d(x,x_3))$,
\end{remark}
and
\begin{remark}\label{rem:eedge-bound}
    Let $e_2$ be an edge and $x$ be one of its endpoints; for any edge $e_1$ in two and three dimensions, $d(e_1,e_2)\le d(x,e_1)$.
\end{remark}
We are now ready to show the positivity of the different barrier potentials. 

\begin{proposition}\label{prop:psic}
If $\dhat$ is \acceptable{}, then $\Psi_c \ge 0$.
\end{proposition}
\begin{proof}
We show that 
\begin{align*}
\Psi_c(x) &= \sum_{e\in E\setminus x} b( d(x, e),\dhat) - \sum_{x_2\in \Vint\setminus x}b( d(x, x_2),\dhat)
\end{align*} 
is positive for every connected component $C$ in the intersection  between $B_{\dhat}^x$ and the boundary. Note that if $C=\emptyset$, then $\Psi_c(x)=0$. We now count the number of possible edges and vertices in $C$:
\begin{enumerate}
    \item $C$ contains only one edge and no vertices.
    \item $C$ contains only two edges and one vertex.
\end{enumerate}
No other cases are possible as the boundary is manifold, $C$ contains only one connected component, and $B_{\dhat}^x$ can contain at most a vertex. For case 1, $\Psi_c(x)\ge 0$ is trivial as it contains only positive terms. For case 2, it follows from \cref{rem:edge-bound} and $b$ being a monotonically decreasing function, that for the vertex $\tilde x$ shared by the two edges $e_1$ and $e_2$
\[
b( d(x, e_1),\dhat) - b( d(x, \tilde x),\dhat)= \Delta b\ge 0,
\]
therefore 
\[
\Psi_c(x) =  b( d(x, e_1),\dhat)+ b( d(x, e_2),\dhat) - b( d(x, \tilde x),\dhat)=
\Delta b+b( d(x, e_2),\dhat)\ge0.
\]
\end{proof}
The proof for $\Psi_s$ follows a similar idea.
\begin{proposition}\label{prop:psis}
If $\dhat$ is \acceptable{}, then $\Psi_s \ge 0$.
\end{proposition}
\begin{proof}
We show that 
\begin{align*}
\Psi_s(x) &= \sum_{t\in T\setminus x} b( d(x, t),\dhat) 
 - \sum_{e\in \Eint\setminus x} b( d(x, e),\dhat)
 + \sum_{x_2\in \Vint\setminus x}b( d(x, x_2),\dhat)
\end{align*} 
is positive for every connected component $C$ in the intersection  between $B_{\dhat}^x$ and the boundary. Note that if $C=\emptyset$, then $\Psi_s(x)=0$. We now count the number of possible triangles, edges, and vertices in $C$:
\begin{enumerate}
    \item $C$ contains only one triangle and no edges or vertices.
    \item $C$ contains only two triangles, one edge, and no vertices.
    \item $C$ contains only $n$ triangles, $m$ edged, and one vertex $\tilde x$.
\end{enumerate}
No other cases are possible as the boundary is manifold, $C$ contains only one connected component, and $B_{\dhat}^x$ can contain at most a vertex. 
For case 1, $\Psi_s(x)\ge 0$ is trivial as it contains only positive terms. 
For case 2, it follows from \cref{rem:tri-bound} and $b$ being a monotonically decreasing function, that for the vertex $\tilde e$ shared by the two triangles $t_1$ and $t_2$
\[
b( d(x, t_1),\dhat) - b( d(x, \tilde e),\dhat)= \Delta b\ge 0,
\]
therefore 
\[
\Psi_s(x) =  b( d(x, t_1),\dhat)+ b( d(x, t_2),\dhat) - b( d(x, \tilde e),\dhat)=
\Delta b+b( d(x, t_2),\dhat)\ge0.
\]
For case 3, we first note that the number of triangles $n$ is always larger or equal to the number of edges $m$. This is the case since the edges need to be in the interior (no boundary edges), and if an edge is included in $C$, then the two adjacent triangles are. Following a similar argument as for case 2, we can bound every edge barrier $b( d(x, e),\dhat)$ with one of the adjacent triangles' barriers $b( d(x, t),\dhat)$. Since $n\ge m$,
\[
\sum_{t\in T\setminus x} b( d(x, t),\dhat) 
 - \sum_{e\in \Eint\setminus x} b( d(x, e),\dhat)= \Delta b'\ge 0,
\]
and
\[
\Psi_s(x) = 
 \Delta b'+b( d(x, \tilde x),\dhat)\ge 0.
\]
\end{proof}
Finally, we show the positivity of the edge-edge energy.
\begin{proposition}\label{prop:psiw}
If $\dhat$ is \acceptable{}, then $\Psi_w \ge 0$.
\end{proposition}
\begin{proof}
We show that 
\begin{align*}
\Psi_w(e) = \sum_{e_2 \in E\setminus e} b(d(e, e_2),\dhat) -
\sum_{x_1\in V \setminus e} (\rho(x_1)-1) b( d(x_1, e),\dhat)
\end{align*} 
is positive for every connected component $C$ in the intersection  between $B_{\dhat}^x$ and the boundary. We now count the number of possible edges and vertices in $C$ (excluding the trivial case $C=\emptyset$):
\begin{enumerate}
    \item $C$ contains only one edge and no vertices.
    \item $C$ contains $n_{e,\tilde x}$ edges and one vertex $\tilde x$.
\end{enumerate}
No other cases are possible as the boundary is manifold, $C$ contains only one connected component, and $B_{\dhat}^x$ can contain at most a vertex. For case 1, $\Psi_w(x)\ge 0$ is trivial as it contains only positive terms. For case 2, it follows from \cref{rem:eedge-bound} that for $\tilde x$ shared by the $n_{e,\tilde x}$  
\[
b( d(e, e_2),\dhat) - b( d(\tilde x, e),\dhat) \ge 0,
\]
for every edge $e_2\in C$. Therefore for an $e^\star \in E\setminus e$
\[
\Psi_w(e) = \sum_{e_2 \in E\setminus (e\cup
e^\star)} (b(d(e, e_2),\dhat) - b( d(\tilde x, e),\dhat))+b(d(e, e^\star),\dhat)\ge 0.
\]
\end{proof}

\subsection{Quality of the smooth approximation}
We just showed that our approximations to the actual maximum distance share the positivity property of the actual non-smooth max. By looking into the previous proofs, we can estimate how and where our approximations break down. In all cases, if the set $C$ contains more than one connected component, the approximation is poor. For instance, this can happen when $\dhat$ is \acceptable but larger than the high-frequency details of the mesh. In the following, we will focus on the case where $C$ contains only one component. If $C$ contains only one primitive (case 1 in the proofs), our approximation is trivially exact as the sum contains only one term, which coincides with the maximum.

\paragraph{Vertex-edge case} In the proof of \cref{prop:psic}, we arbitrary select $e_1$ to bound the point distance $b( d(x, \tilde x),\dhat)$. Let $e^\star=\argmin_{e_i=\{e_1, e_2\}}(d(x, e_i))$, $\bar e^\star$ the other edge and $\Delta b= b( d(x, \bar e^\star),\dhat) - b( d(x, \tilde x),\dhat)$. It is now clear to see that the error in the approximation is exactly $\Delta b$, which happens when $x$ is closer to the edge $\bar e^\star$ than to the vertex $\tilde x$ as shown in \cref{fig:corner}.

\paragraph{Vertex-triangle case} Case 2 in the proof of \cref{prop:psis} follows the same argument as the vertex-edge case. We call $t^\star$ the triangle closest to $x$, $\bar t^\star$ the other one, and $\Delta b= b( d(x, \bar t^\star),\dhat) - b( d(x, \tilde e),\dhat)$. In this case, the error is $\Delta b$, which happens when $x$ is closer to the triangle $\bar t^\star$ than to the edge $\tilde e$ (imagine an extruded version of \cref{fig:corner}). For the sake of simplicity, we exclude the case where the surface has boundaries; therefore, in case 3, we have the same number of edges and faces. Let us call $t^\star$ the triangle closest to $x$ and $\bar T$ the triangles in $C$ excluding $t^\star$. If there exists one edge $e^\star$ such that $d(x, e^\star)=d(x, \tilde x)$ (i.e., $x$ is sufficiently far from concavities), then the error is the sum of the differences between the barriers on the triangles and their adjacent edges. This is the case in concave parts of the mesh.

\paragraph{Edge-edge case} In the proof of \cref{prop:psiw}, we clearly see that our approximation is related to the errors introduced if the distances between the edges and the vertex are non-zero. This happens when more than one edge is closer to the vertex $\tilde x$ which is the case when $e$ is large or next to $\tilde x$.

\section{Time Discretization}
\label{sec:time-stepper}

After spatial discretization we start with $n$ nodal positions $x \in \R^d$ in the computational mesh, concatenated as degrees of freedom in the vector, $\bfi{x} \in \R^{d n}$. Correspondingly, we have a finite element mass matrix, $M$, and total deformation energy, $U(\bfi{x})$, defined on the material domain. 

Discretizing in time, we solve time steps variationally (see \cref{sec:numerical-solution} below). To do so we minimize discrete energies whose stationary points give each applied numerical time integration method's positional update \cite{oritz1999variational}. At time $t$, we have prior nodal positions $\bfi x^{t}, \bfi x^{t-1},...$  and velocities $\bfi v^{t}, \bfi v^{t-1},...$ . Applying a time step size of $h$, we then compute the time step update for the next nodal \emph{positions} $\bfi x^{t+1}$ as the minimization of an appropriately constructed Incremental Potential (IP)  \cite{kane2000variational}, $E\left(\bfi x, h \right)$, over valid $\bfi x \in \mathbb{R}^{d n}$ so that 
$$\bfi x^{t+1} = \argmin_{\bfi x} E\left(\bfi x, h \right).$$
For each such IP we accompany it with a velocity update function $v(\cdot)$, that correspondingly defines the time step method's velocity update, $$\bfi v^{t+1} = \bfi v(\bfi x^{t+1}),$$ or, applied per node as $v^{t+1} = v(x^{t+1})$.

As a particular example consider Newmark with $\beta = \tfrac{1}{4}$ and $\gamma = \tfrac{1}{2}$. With  applied external forces, $f_{\text{ext}}$, we have the Newmark IP,  
\begin{align}
E\left(\bfi x, h \right) & = \tfrac{1}{2} \|\bfi x - \tilde {\bfi x}\|_M^2 + \tfrac{h^2}{4} W(\bfi x), \> \text{with}\\
 \tilde {\bfi x} &= \bfi x^t + h \bfi v^t +  \tfrac{h^2}{2} M^{-1}  f{_\text{ext}} - \tfrac{h^2}{4} M^{-1}  \nabla W(\bfi x^t),
\end{align}
and its corresponding velocity update
$$
{\bfi v}(\bfi x) = \tfrac{2}{h} (\bfi x - \bfi x^t) - \bfi v^t, 
$$
with $W(\bfi x)$ the sum of internal ($U$), barrier (see \cref{sec:contact-discretization}) and (when applicable) friction (see \cref{sec:friction-discretization}) energies.

\section{Friction Discretization}
\label{sec:friction-discretization}

Following our contact barrier discretization we next smoothly approximate friction forces so the we can continue to employ Newton-type methods to solve contact with friction and numerically integrate the corresponding \emph{contact-coupled} energy over the curve and surface domains to form a dissipative potential.

We first construct discrete friction potentials corresponding with piecewise-linear discretizations corresponding to our barrier discretizations. On the polyline and triangular surface meshes they are respectively 
\begin{align*} 
\frac{1}{2}&\sum_{\bar x \in \bar V} w_{\bar x} D\Big( v\big[x(\bar x)\big], v\big[\ell_c[x(\bar x)]\big], x(\bar x), \ell_c[x(\bar x)] \Big) \>\> \text{with} \>\>
\ell_c[x_1] = \argmin_{x_2 \in e, \ e\in E\setminus x_1} d(x_1,x_2), 
\end{align*}
and 
\begin{align*} 
\frac{1}{2}&\sum_{\bar x \in \bar V} w_{\bar x} D\Big( v\big[x(\bar x)\big], v\big[\ell_s[x(\bar x)]\big], x(\bar x), \ell_s[x(\bar x)] \Big) \>\> \text{with} \>\>
\ell_s[x_1] = \argmin_{x_2 \in t, \ t \in T\setminus x_1} d(x_1,x_2). 
\end{align*}

Following the above smooth approximation of the max operator in the contact barrier via differences, our corresponding curve and surface friction energies are then respectively
\begin{align*}
\Phi_c(x) & =\sum_{e \in E \setminus x} D(v(x), v(e, x), x, p(e, x)) \quad-\sum_{x_2 \in V_{i n t} \setminus x} D\left(v(x), v\left(x_2\right), x_1, x_2\right) \\
& \approx D\left(v(x), v\left(\ell_e[x(\bar{x})]\right), x, \ell_e[x(\bar{x})]\right),
\end{align*}
and
\begin{align*}
\Phi_s(x) & =\sum_{t \in T \setminus x} D(v(x), v(t, x), x, p(t, x))-\sum_{e \in E_{i n t} \setminus x} D(v(x), v(e, x), x, p(e, x)) \\
& +\sum_{x_2 \in V_{i n t} \setminus x} D\left(v(x), v\left(x_2\right), x, x_2\right) \\
& \approx D\left(v(x), v\left(\ell_s[x]\right), x, \ell_s[x]\right),
\end{align*}
where $v(s, x)$ and $p(s, x)$ respectively return the discrete velocity and position of the closest point in a simplex $s$ (edge or face) to a point $x$. In turn, the final, discrete, dissipative friction potential is
$$D_c(V) = \sum_{x \in V} w_x \Phi_c(x),$$ 
for curves, and 
$$D_s(V) = \sum_{x \in V} w_x \Phi_s(x),$$ for surfaces, where we simplify with $w_x = w_{\bar x}$ defined appropriately, per domain, as covered above. Corresponding total friction forces are then
$$f_c(V) = -\sum_{x \in V} w_x \frac{\partial \Phi_c(x)}{\partial v},$$ 
and
$$f_s(V) = - \sum_{x \in V} w_x \frac{\partial \Phi_s(x)}{\partial v}.$$

\section{Numerical Solution}
\label{sec:numerical-solution}

As covered in \cref{sec:time-stepper}, after discretization each simulation step solution (dynamic, quasistatic, or static) is generated by \emph{locally} minimizing the applied incremental potential (IP), $E(\bfi x, h)$. To minimize $E$ we apply a Projected Newton solver customized for handling barrier potentials. Projected Newton (PN) methods are second-order unconstrained optimization strategies for minimizing nonlinear, nonconvex functions where the Hessian may be indefinite. At each Newton iteration, we project all \emph{local} energy stencils' (including barrier and friction) Hessians to the cone of symmetric positive semi-definite (PSD) matrices prior to assembly. 

While our barrier energies diverge at contact, this alone does not guarantee that a Newton iteration process will not violate the distance constraints ($d>0$) for all possible contact pairs. Standard line search~\cite{nocedal2006numerical}, e.g., back-tracking with Wolfe conditions, can find an energy decrease in configurations that have passed through intersection, resulting in a step that takes the geometry out of the admissible set. To ensure feasibility for all position updates internal to the solver we apply a continuous, intersection-aware line search filter for 3D meshes. In each line search we first apply a continuous collision detection (CCD)~\cite{li2021codimensional} to conservatively compute a large, but always feasible, step size along the descent direction. We then apply back-tracking line search from this step size upper bound to obtain energy decrease. CCD then certifies that each step taken is always valid. 
When we apply friction we follow Li et al.'s\ \cite{Li2020IPC} lagged-iteration method and supplement the incremental potential over successive iterated Newton solves, per time step, with a pre-scaled pseudo-potential energy that holds contact-force magnitudes and sliding projections fixed from the prior Newton solve, until convergence with current contact forces and projections. When we apply barrier-based energy densities for our elasticity potential, $U$, e.g., neo-Hookean, we combine an inversion-aware line search filter~\cite{smith2015bijective} that additionally pre-filters the search direction for a large but always inversion-free step size. In combination this guarantees that every step of every position change in the Newton iteration process (and so simulation) applies an intersection- and (when desired)  inversion-free update.

\paragraph{Termination} 
For termination of the solver we check convergence with the infinity norm of the Newton search direction (Newton decrement) scaled by time step (but \emph{unscaled} by line-search step size). Specifically we solve each time step's barrier IP to an accuracy satisfying $\frac{1}{h} \| H^{-1} \nabla E(x) \|_\infty < \epsilon_d$. This provides affine invariance and a characteristic measure using the Hessian's natural scaling as metric. Accuracy is then directly defined by $\epsilon_d$ in physical units of velocity (and so is independent of time-step size applied) and consistently measures quadratically approximated distance to local optima across examples with varying scales and conditions.

\paragraph{Solution accuracy} 
Each such numerically converged time-step solution satisfies accuracy criteria for choice of applied numerical time integration method.
Discrete \emph{momentum balance} is directly satisfied as standard after convergence. For example, in a simple illustrative case with implicit Euler we have
\begin{align*}
\label{eq:ie_em}
\nabla_x  E(\bfi x,h) = 0 \implies M\left(\frac{x-\hat{x}}{h^2}\right) = - \nabla W(x),
\end{align*}
Comparable discrete momentum balance follows when we apply alternate time integration methods, e.g. implicit Newmark.
Here contact forces, per surface vertex (and, when applied, edge) stencils $k$ are then  
\begin{align*}
- w_k \frac{\partial \Psi_k(x_k)}{\partial x_k}. 
\end{align*}
In turn \emph{positivity}, of these forces is covered in detail in Section\ \ref{sec:positivity} above.
Line-search filtering then guarantees \emph{admissibility} (non-intersection) and, when applicable, for barrier-type elasticity energy densities, \emph{global injectivity}. Finally, our barrier definition ensures that a \emph{discrete complimentarity} is always satisfied as contact forces can not be applied at distance more than $\hat{d}$ away.

\section{Benchmark Evaluation}

We begin our evaluation with two benchmark problems treating the transient impact of linearly elastic bars in one dimension. Importantly, both problems are equipped with analytic solutions. This allows us to compare results with prior methods analyzed by Doyen and colleagues~\cite{doyen2011time} and to also demonstrate convergence of our contact model to known elastodynamic impact solutions. 

Following Doyen et al.~\cite{doyen2011time}, both problems resolve the dynamics of a one-dimen\-sional linearly elastic bar of length $L=\SI{10}{\meter}$, Young's modulus $E=\SI{900}{\newton}$, and density $\rho=\SI{1}{\kilo\gram/\meter}$, initialized (undeformed) at a height of $h_0=\SI{5}{\meter}$ above a rigid ground. Each bar is spatially discretized with a uniform mesh size of $\Delta x$ via linear finite elements. With this common framework there are then two benchmark problems. 

The first, an \emph{impact} problem, resolves a single impact of an elastic bar by initializing the bar's velocity to $v_0 = \SI{-10}{\meter/\second}$ and eliminating gravitational acceleration ($g_0 = 0$). This benchmark has been widely applied in prior analyses and enables comparison of the numerical oscillation artifacts generated by differing contact models~\cite{doyen2011time}.

The second, a \emph{bouncing} problem, resolves a periodic sequence of elastic bar impacts and free-flights. To do so, with the above chosen material parameters, we initialize the bar at rest ($v_0 =0$) with a gravitational acceleration of $g_0 = \SI{-10}{\meter/\second^2}$. This obtains a periodic trajectory of alternating contacts and free-flight for the bar described by an analytic solution. This benchmark, introduced by Doyen et al.~\cite{doyen2011time}, further enables us to analyze the energy evolution and longer-term trajectories generated by contact models over sequential impacts. 

Below we cover the results of our benchmark testing in detail. Here, we first quickly summarize our key takeaways. In brief we note that the following tests demonstrate that \EIPC \emph{qualitatively} follows the the displacement and contact pressure behavior of a penalty-based contact model. However, unlike penalty-based methods, we confirm that \EIPC additionally provides interpenetration-free trajectories independent of choice of discretization and contact-stiffness parameters. We then show convergence of our model under refinement to the benchmarks' analytical solutions. To our knowledge, these are the first results to demonstrate this convergence.

\subsection{Comparison with Doyen et al.'s benchmark}

\label{sec:1d}
\label{sec:1d_impact}
\label{sec:1d_bounces}

For a direct, side-by-side comparison with Doyen et al.'s evaluation we begin by applying implicit Newmark ($\beta=\tfrac{1}{4}$, $\gamma=\tfrac{1}{2}$) time integration, spatial discretization with $\Delta x = \SI{0.1}{\meter}$, and calculate time steps $h = \nu_C \> \Delta x / \sqrt{E/\rho}$ with a Courant number of $\nu_C = 1.5$. We comparably set our contact model with $\hat{d}=\Delta x$ and $\kappa = 0.1 E$, and so effectively treat the contact barrier as an additional hyperelastic potential. 

We summarize simulation results for the \emph{impact} and \emph{bouncing} problems in \cref{fig:1d_impact,fig:1d_bounces} respectively. Here we observe that \EIPC generates displacement (bottom node), contact pressure, and energy trajectories with closely comparable profiles, and so qualitatively similar error behaviors, to penalty methods in both magnitude and over time. 

However, in these benchmarks we also see key differences between penalty methods and \EIPC. As demonstrated in the top-left plot of \cref{fig:1d_impact} we observe \EIPC preserves a penetration-free (more generally interpenetration-free) trajectory independent of choice of contact stiffness. This is in contrast to penalty methods where penetration errors are uncontrollable for a fixed contact stiffness. Here \EIPC's displacement error is controllable with a curve remaining above the analytic solution during contact by no more than $\hat{d}$. Following this observation, we next evaluate \EIPC's behavior over variations in the contact model's threshold and stiffness, as well as for alternate choices of numerical time integration.

\begin{figure}[ht]
    \centering
    \parbox{0.48\linewidth}{
        \centering
        {\small Penalty method~\cite{doyen2011time}}\\\vspace{5pt}
        \includegraphics[width=0.49\linewidth]{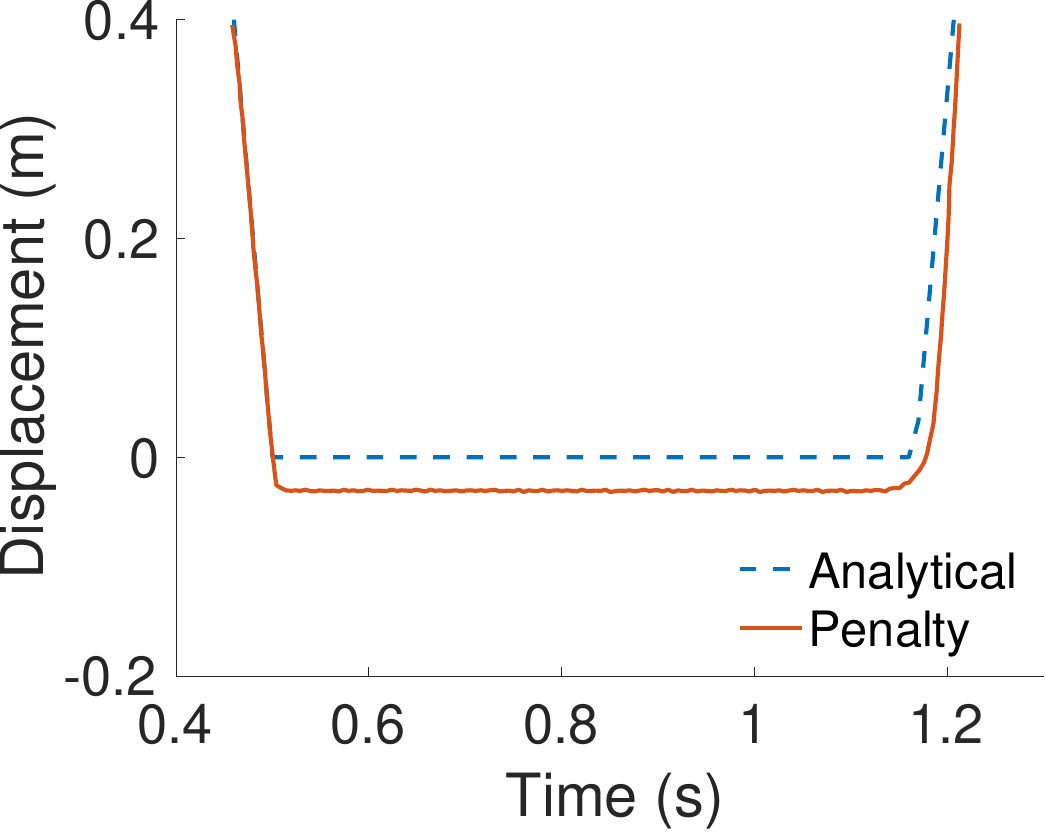}
        \includegraphics[width=0.49\linewidth]{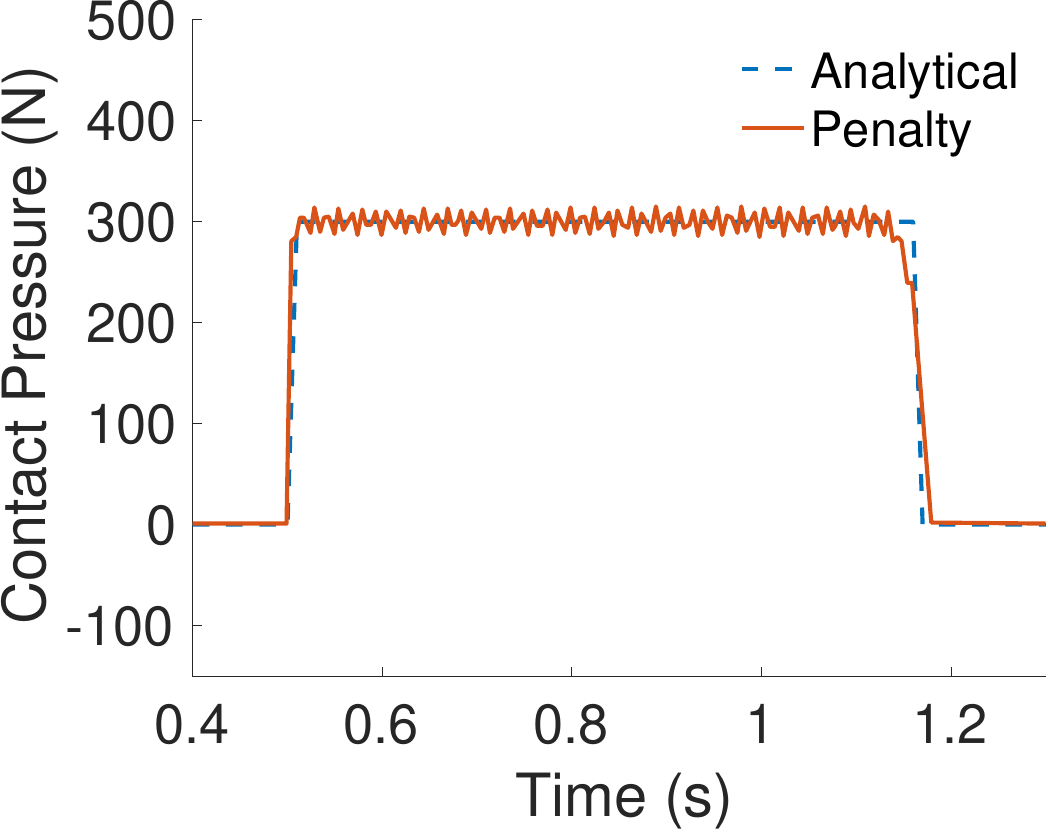}
    }
    \hspace{0.02\linewidth}
    \parbox{0.48\linewidth}{
        \centering
        {\small Ours with $\hat{d}=\Delta x$ and $\kappa = 0.1Y$}\\\vspace{5pt}
        \includegraphics[width=0.49\linewidth]{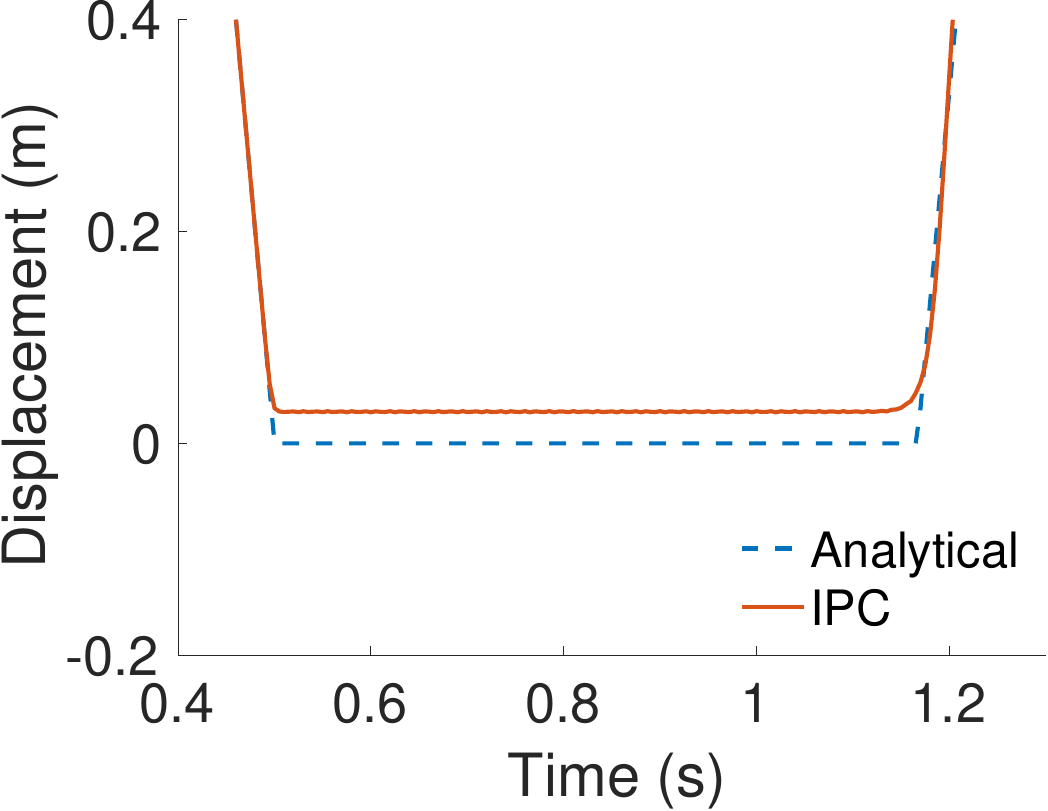}
        \includegraphics[width=0.49\linewidth]{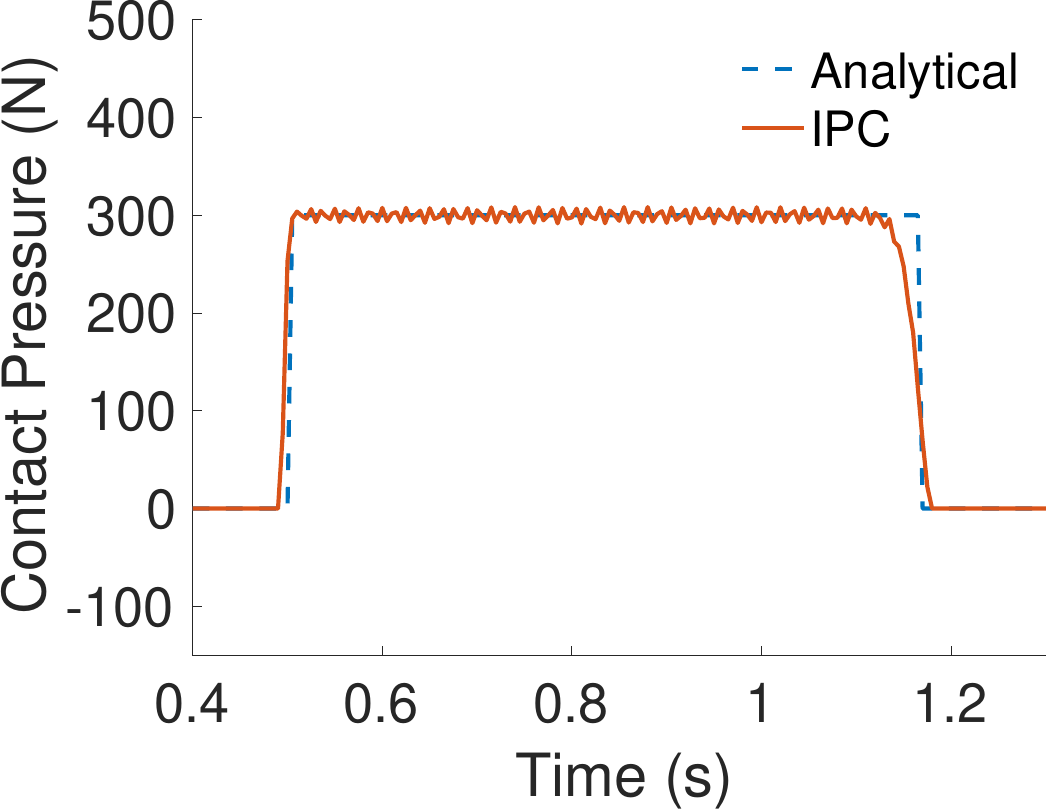}
    }
    \caption{
    \textbf{Impact of an elastic bar.}
    Comparison between the penalty-based contact method in \citetdoyen{} (left) our method (right), both with Newmark time integration and same other settings.}
    \label{fig:1d_impact}
\end{figure}

\subsection{Varying barrier stiffness, threshold, and time-integration}

While there is no need to change \EIPC's contact model parameters in order to avoid interpenetration, reducing the threshold parameter $\hat{d}$ (and/or the contact stiffness $\kappa$) will improve the complementarity accuracy (decreased gap at contact) in simulation results. At the same time varying these parameters has direct implications for the contact-pressure oscillations produced, and so on the stability of the solutions obtained. In turn it is then also important to consider choice of the time integration method applied. 

In \cref{fig:1d_impact_diffStiff}, top and middle, we see that varying \EIPC's contact stiffness by $0.1\times$ or $10\times$ introduces significantly smaller variations in contact pressure when compared to the large jumps obtained by varying contact stiffness with the  penalty method (as observed in \citetdoyen{}). However, varying $\hat{d}$ in the \EIPC model similarly reproduces comparably large contact pressure oscillations to varying stiffness in the penalty model. This is because the \EIPC barrier has local support in the distance range of $(0, \hat{d})$, while the sharpness of this potential is more sensitive to changes in $\hat{d}$ than $\kappa$.

\begin{figure}[ht]
    \centering
    \includegraphics[width=0.33\linewidth]{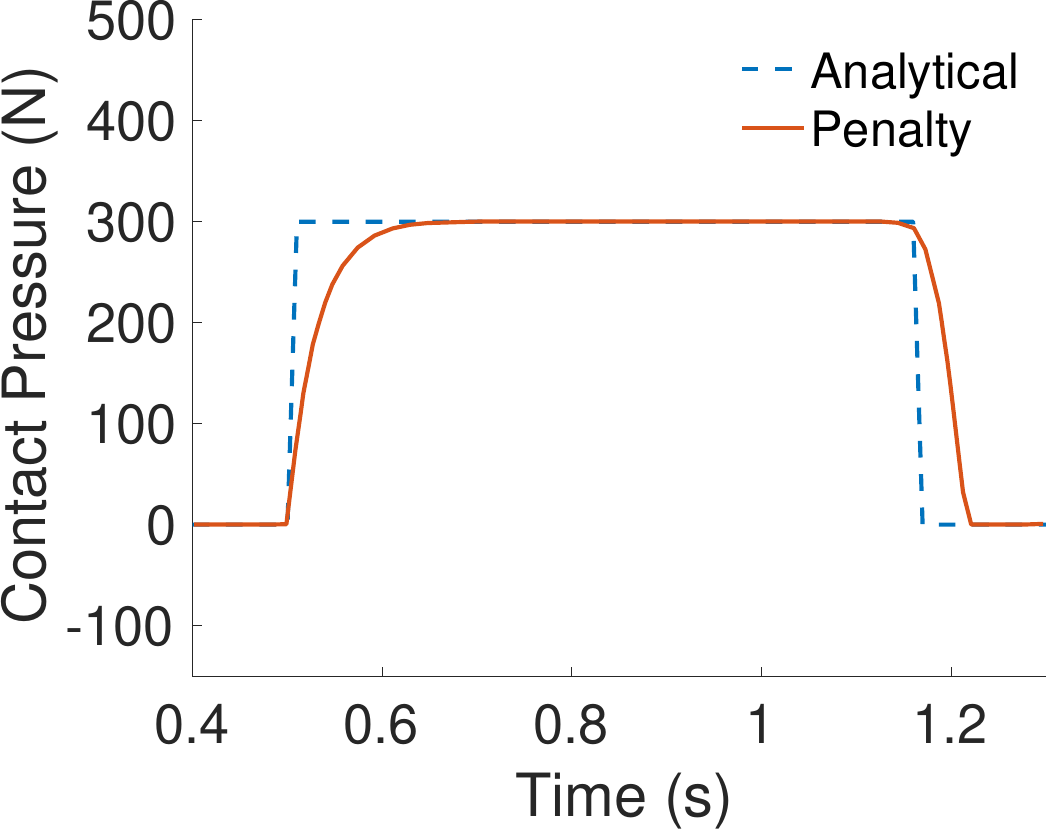} \hspace{0.2cm}
    \includegraphics[width=0.33\linewidth]{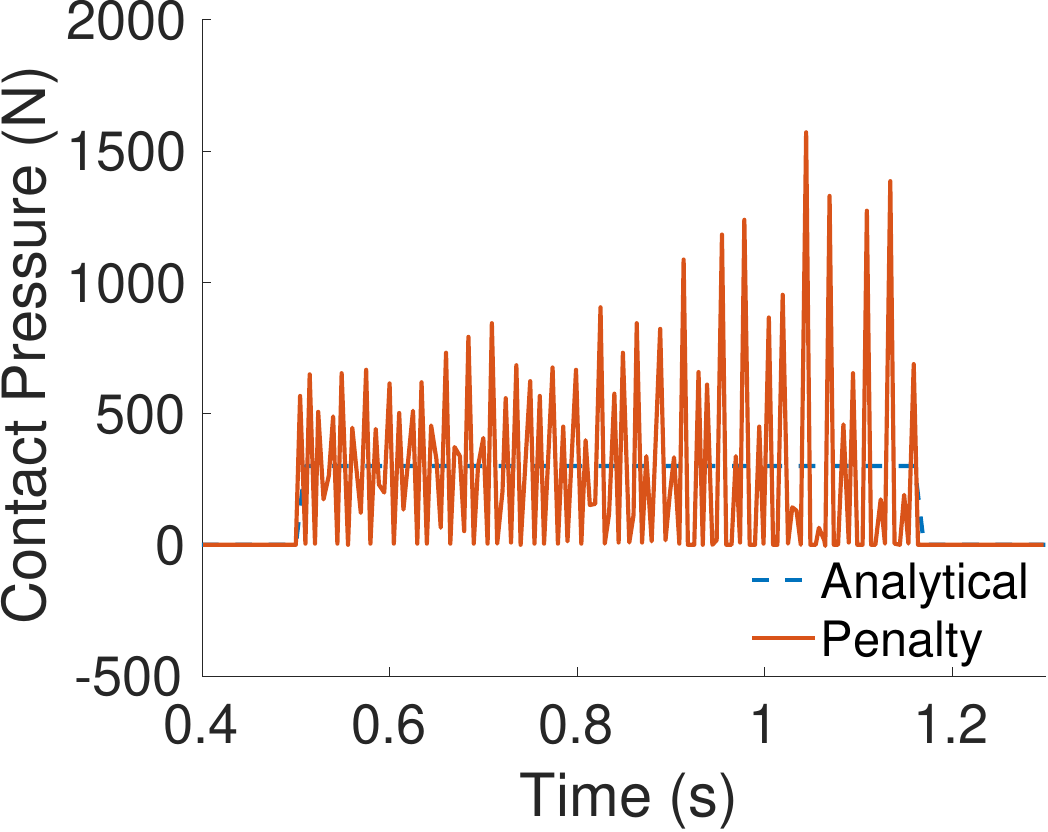}
    \\{\small Penalty method~\cite{doyen2011time} with $0.1\times$ (left) and $10\times$ (right) stiffness}\\
    \vspace{0.3cm}
    \includegraphics[width=0.33\linewidth]{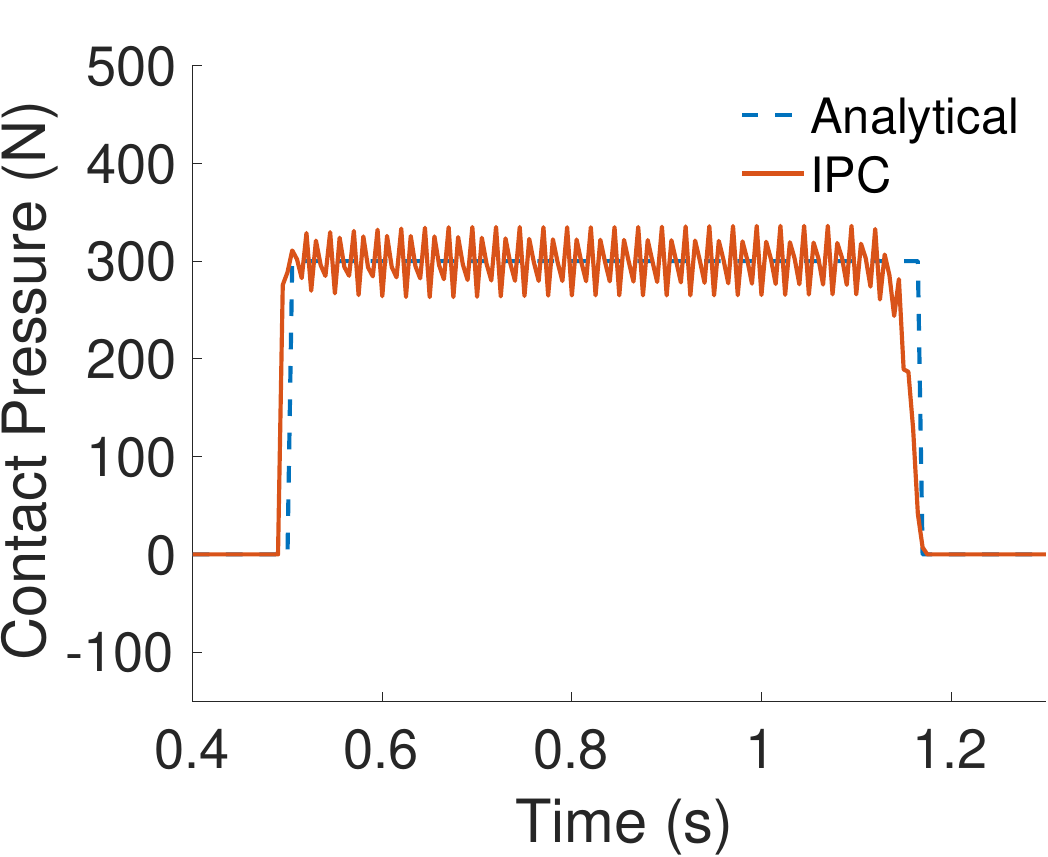} \hspace{0.2cm}
    \includegraphics[width=0.33\linewidth]{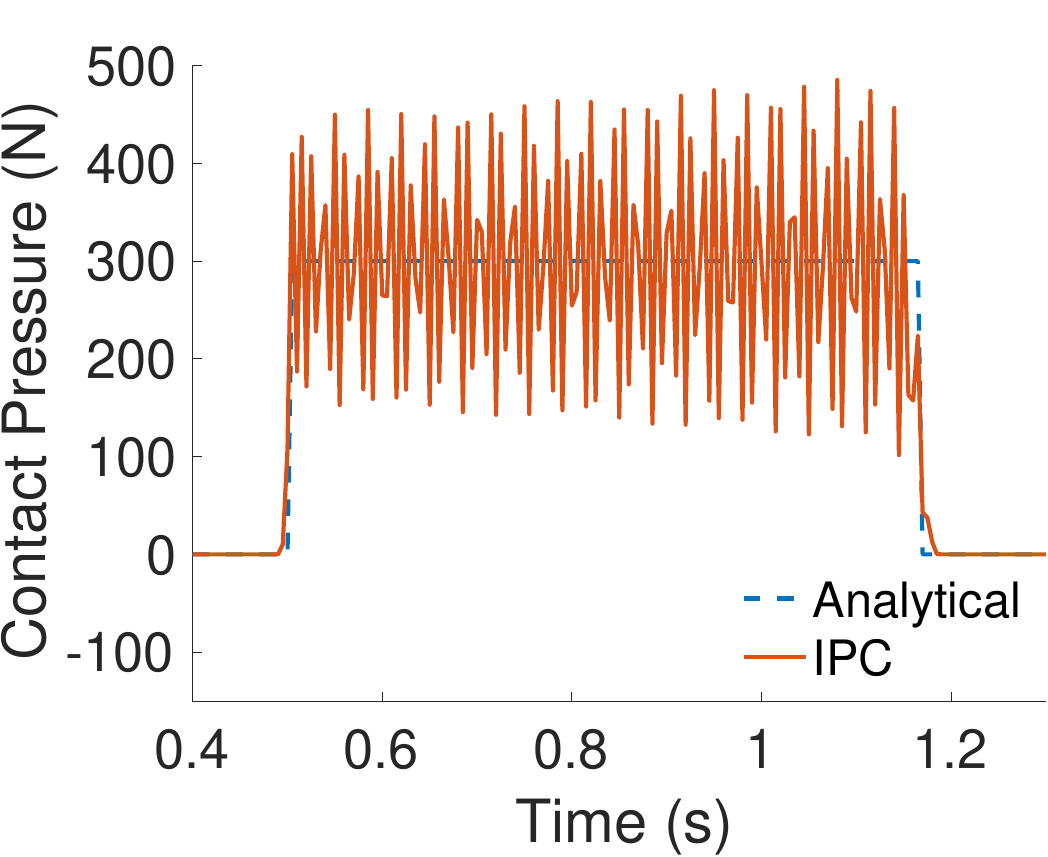}
    \\{\small Ours with $\kappa=Y$ (left) and $\kappa=0.01Y$ (right)}\\
    \vspace{0.3cm}
    \includegraphics[width=0.33\linewidth]{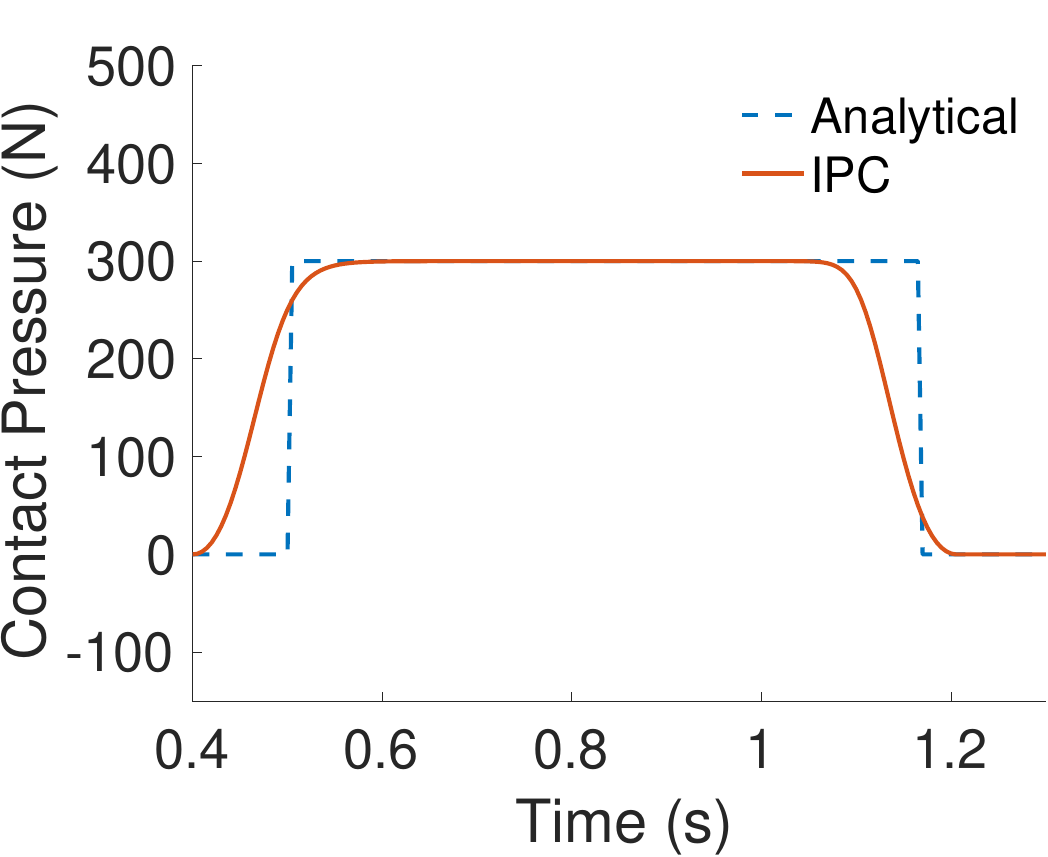} \hspace{0.2cm}
    \includegraphics[width=0.33\linewidth]{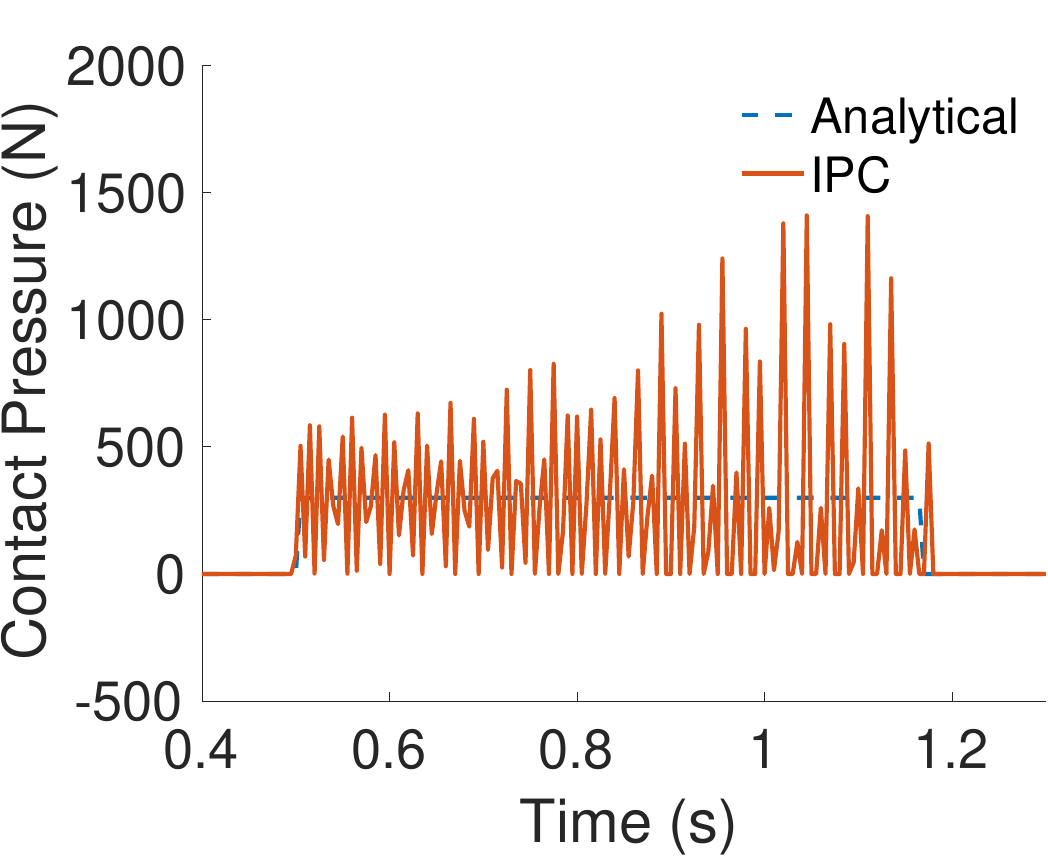}
    \\{\small Ours with $\hat{d}=10\Delta x$ (left) and $\hat{d}=0.1\Delta x$ (right)}\\
    \caption{
    \textbf{Impact of an elastic bar -- different contact stiffness.}
    From top to bottom: penalty method with different stiffness~\cite{doyen2011time}, our method with different $\kappa$, and our method with different $\hat{d}$, all with Newmark time integration and same other settings. Results in the right column has sharper energy than the left column, and so less smooth contact pressure profile.}
    \label{fig:1d_impact_diffStiff}
\end{figure}

This increase in generated contact pressure oscillation, as we decrease $\hat{d}$, implies an important tradeoff. We obtain improved gap accuracy (given by smaller $\hat{d}$) at the cost of a sharper contact potential. In turn, generated pressure oscillations are artifacts from time-integration with these increasingly sharp potentials. Simulations are then significantly improved if we step away from employing marginally stable time integrators like implicit Newmark with $\beta=\tfrac{1}{4}$, $\gamma=\tfrac{1}{2}$. For example, switching to A-stable integrators like BDF-2 and implicit Euler (IE) provides smooth contact pressures with reduced $\hat{d}=0.1\Delta x$ (see \cref{fig:1d_impact_ie_bdf2}) \emph{without} decreasing timestep size. If we then additionally lower the time step size of the IE solution by $0.1\times$ for less numerical dissipation, we see the contact pressure profile (\cref{fig:1d_impact_ie_bdf2} middle) then closely follows many discretizations proposed for improved stability (i.e., see methods 4.1, 4.3, 4.6, 4.7, 6.2, 7.1 in~\cite{doyen2011time}). Similarly, in terms of accuracy, we see BDF-2 generates a solution more than $2\times$ closer to the analytical solution than IE, with significantly less numerical dissipation of the total system's energy.

\begin{figure}[ht]
    \centering
    \parbox{0.3\linewidth}{
        \centering
        {\small \textbf{Implicit Euler}\\$\hat{d}=0.1\Delta x$}\\
        \includegraphics[width=\linewidth]{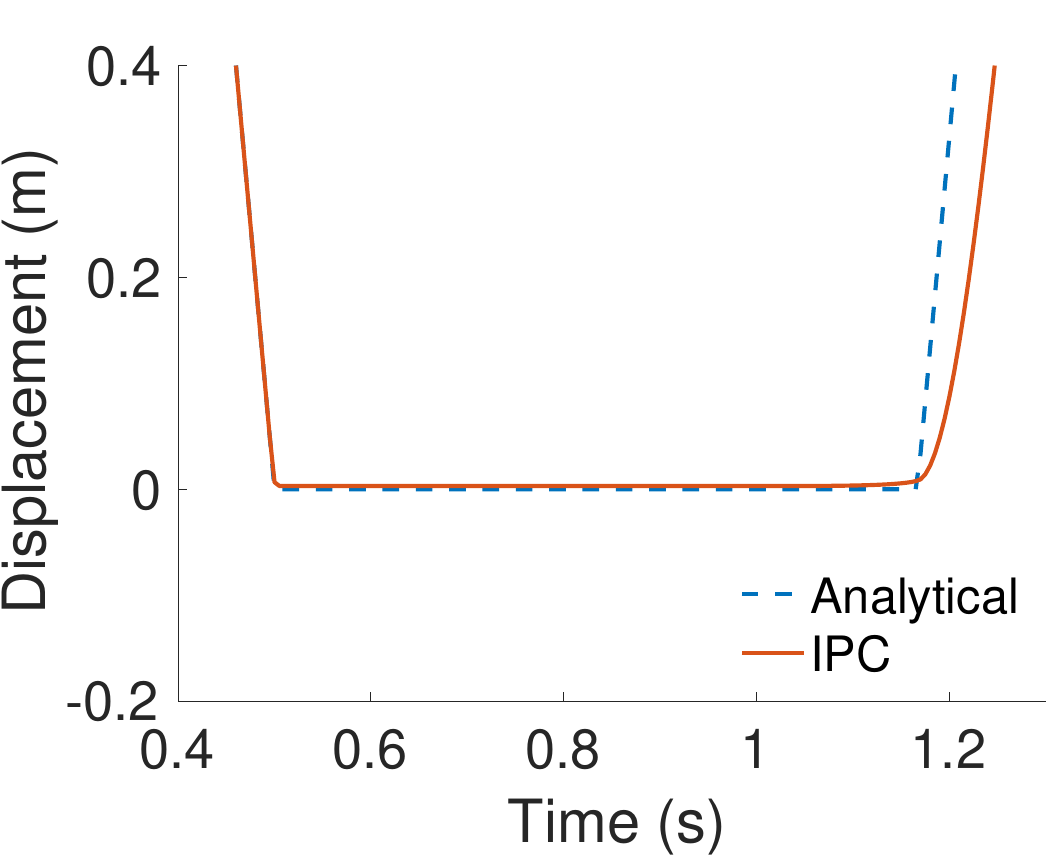}\\
        \includegraphics[width=\linewidth]{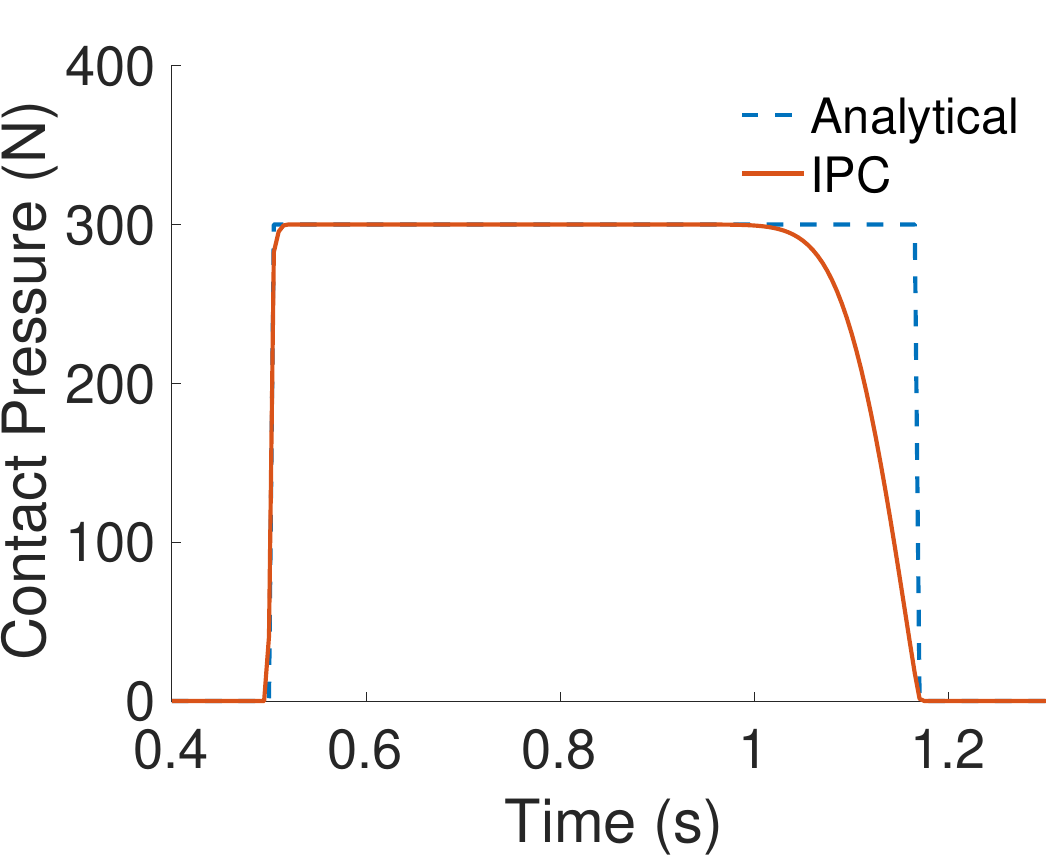}
    }
    \hspace{0.02\linewidth}
    \parbox{0.3\linewidth}{
        \centering
        {\small \textbf{Implicit Euler}\\$\hat{d}=0.1\Delta x$ and $0.1h$}\\
        \includegraphics[width=\linewidth]{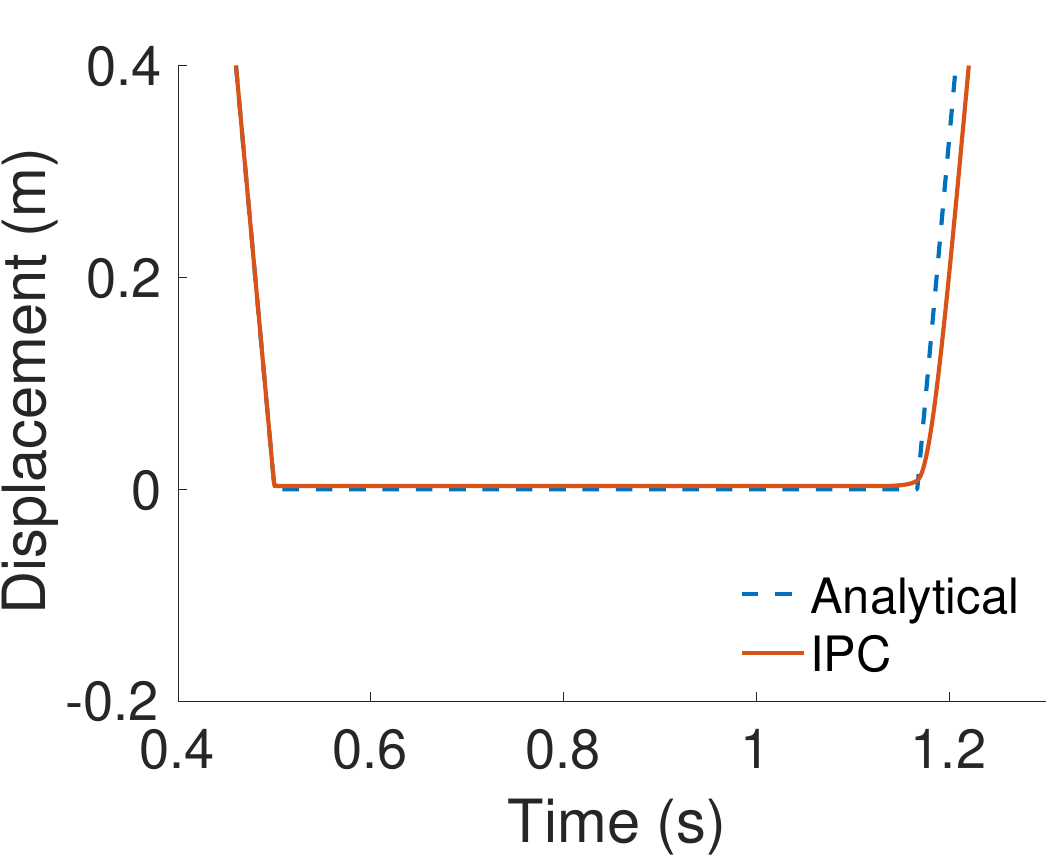}\\
        \includegraphics[width=\linewidth]{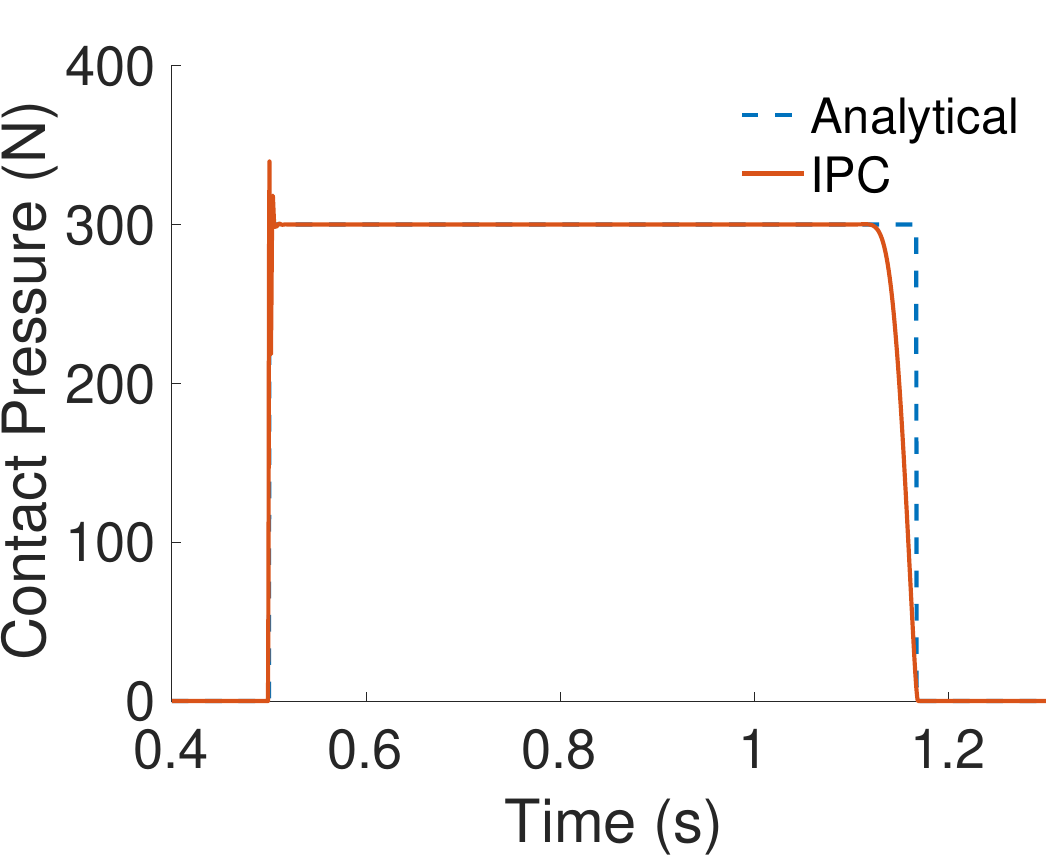}
    }
    \hspace{0.02\linewidth}
    \parbox{0.3\linewidth}{
        \centering
        {\small \textbf{BDF-2}\\$\hat{d}=0.1\Delta x$}\\
        \includegraphics[width=\linewidth]{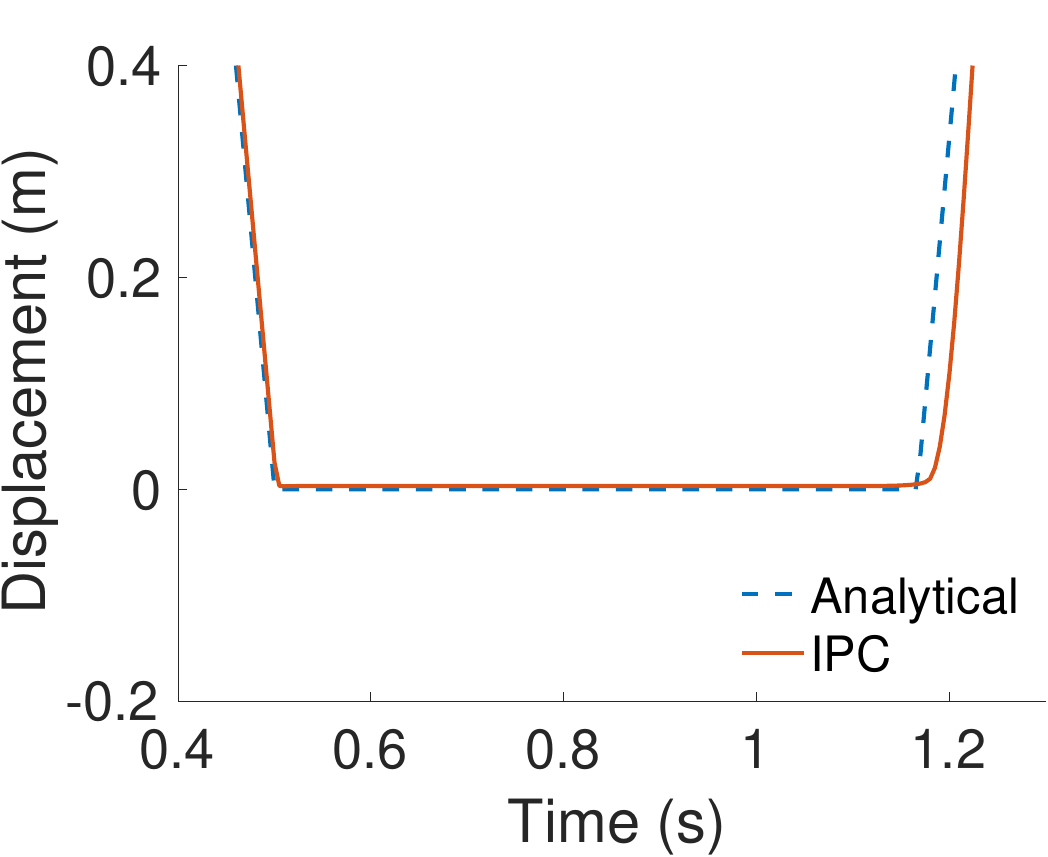}\\
        \includegraphics[width=\linewidth]{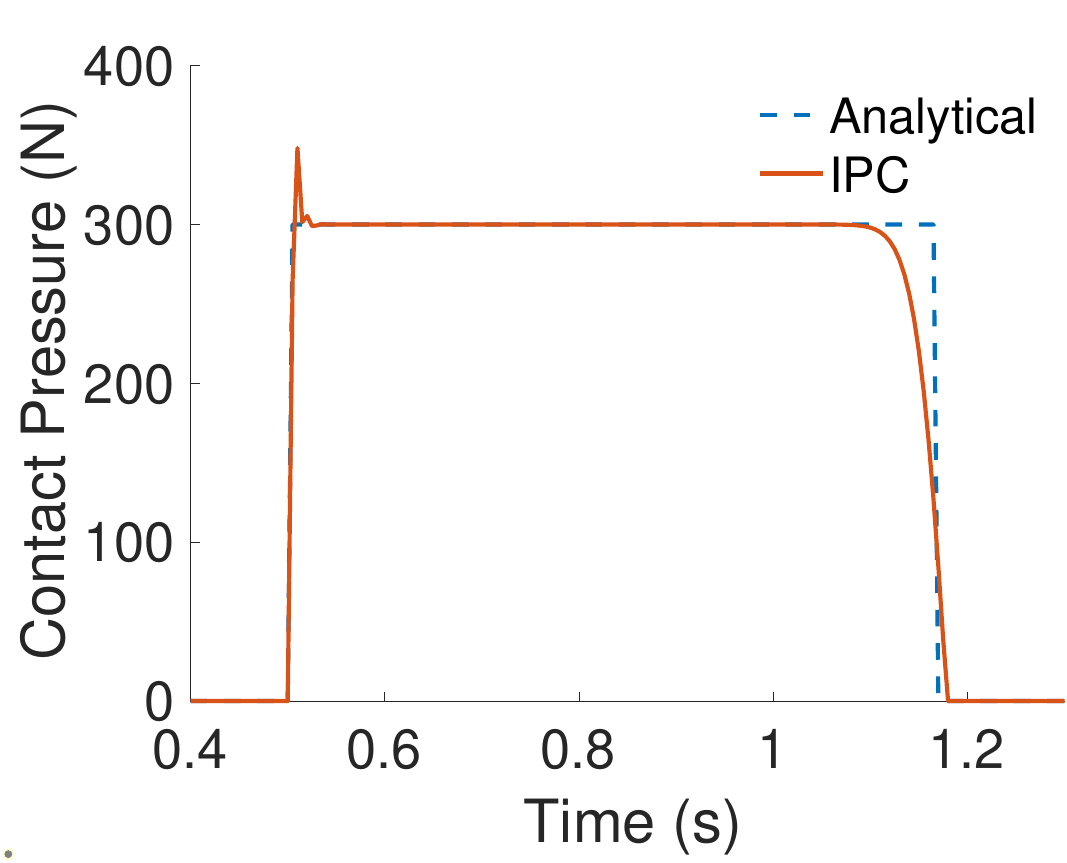}
    }
    \caption{
    \textbf{Impact of an elastic bar -- different time discretizations.}
    Our results with implicit Euler (left two columns) and BDF-2 (right column) time integration. Left: $0.1\times$ smaller $\hat{d}$, middle: $0.1\times$ smaller $\hat{d}$ with $0.1\times$ smaller $h$, right:  $0.1\times$ smaller $\hat{d}$.}
    \label{fig:1d_impact_ie_bdf2}
\end{figure}

\begin{figure}[ht]
    \centering
    \parbox{0.48\linewidth}{
        \centering
        {\small Penalty method~\cite{doyen2011time}}\\\vspace{5pt}
        \includegraphics[width=0.49\linewidth]{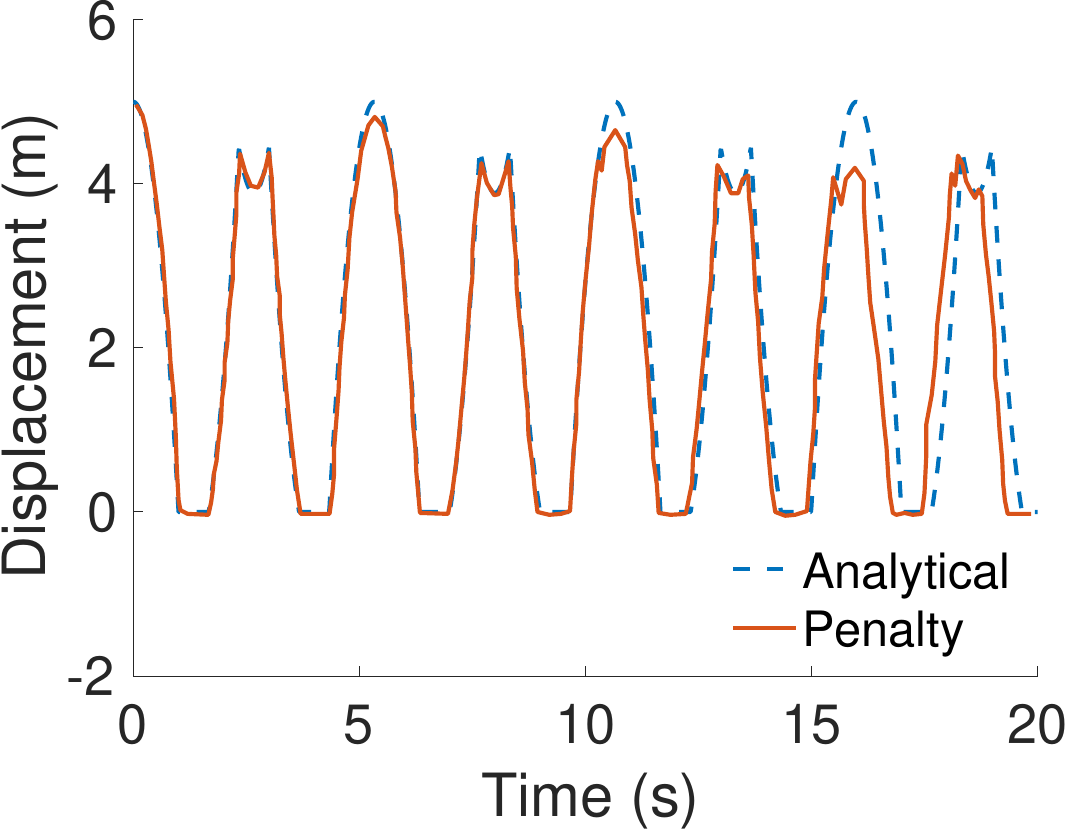}
        \includegraphics[width=0.49\linewidth]{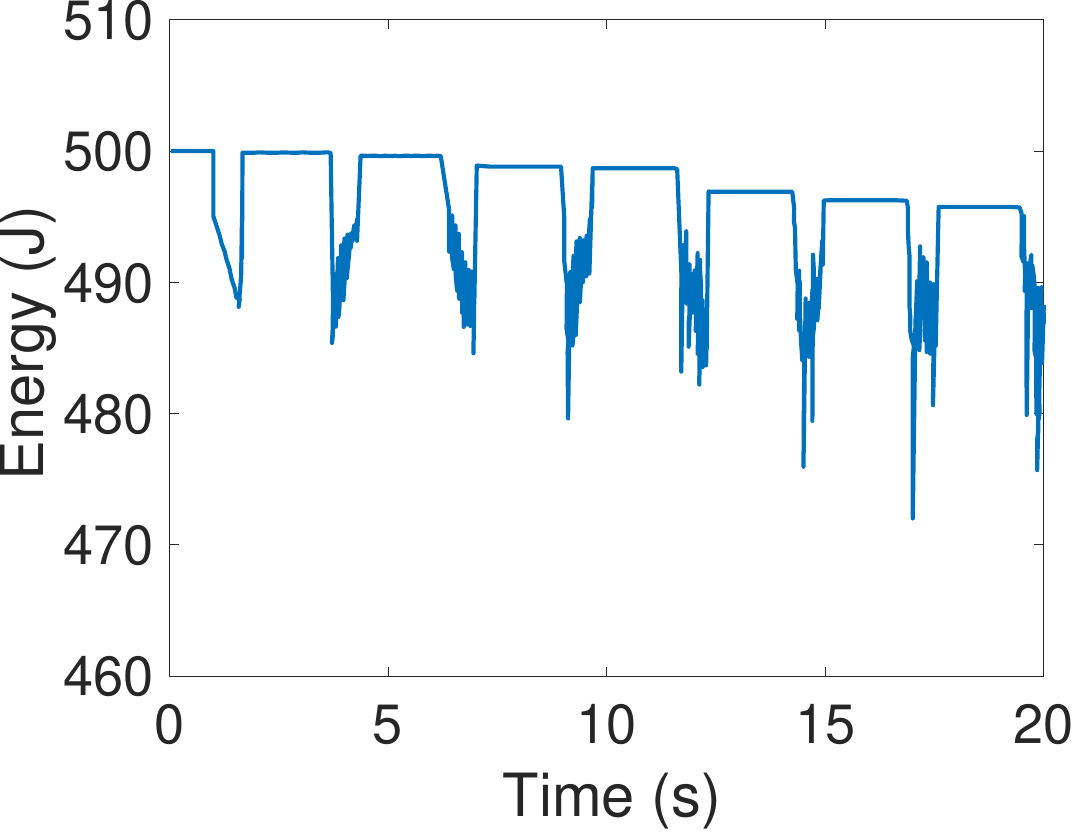}
    }
    \hspace{0.02\linewidth}
    \parbox{0.48\linewidth}{
        \centering
        {\small Ours with $\hat{d}=\Delta x$ and $\kappa = 0.1Y$}\\\vspace{5pt}
        \includegraphics[width=0.49\linewidth]{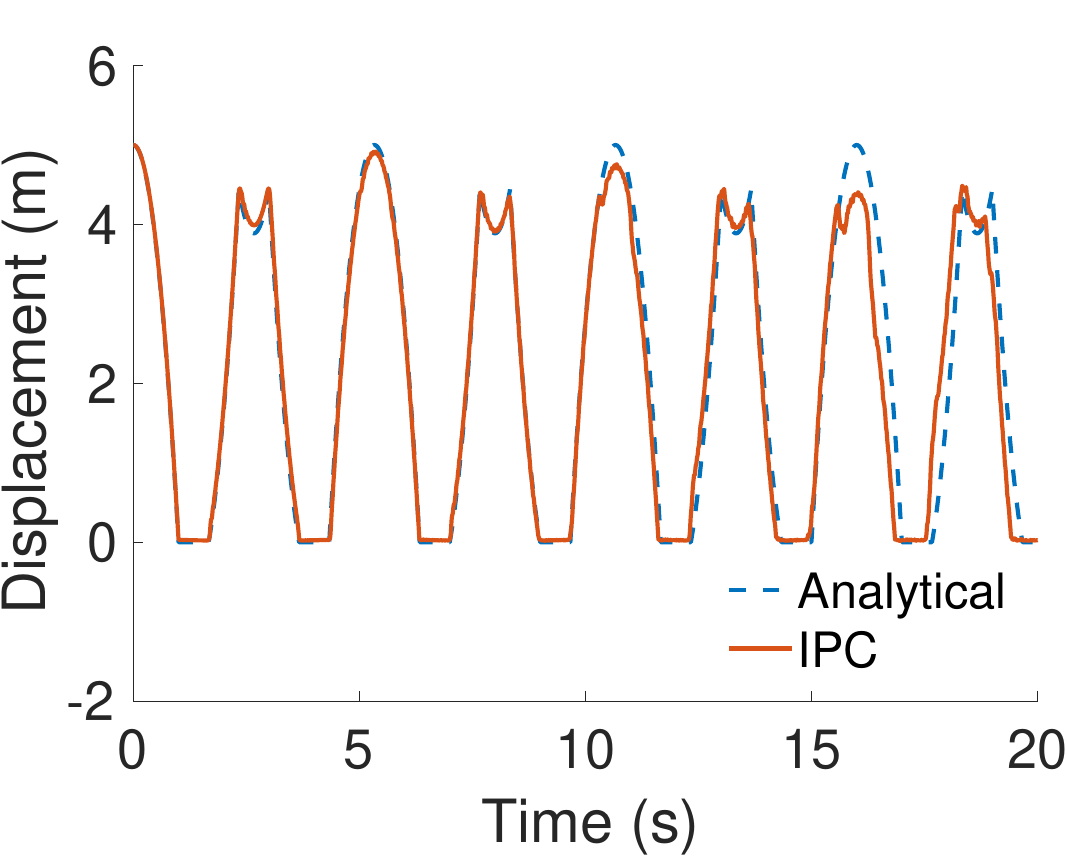}
        \includegraphics[width=0.49\linewidth]{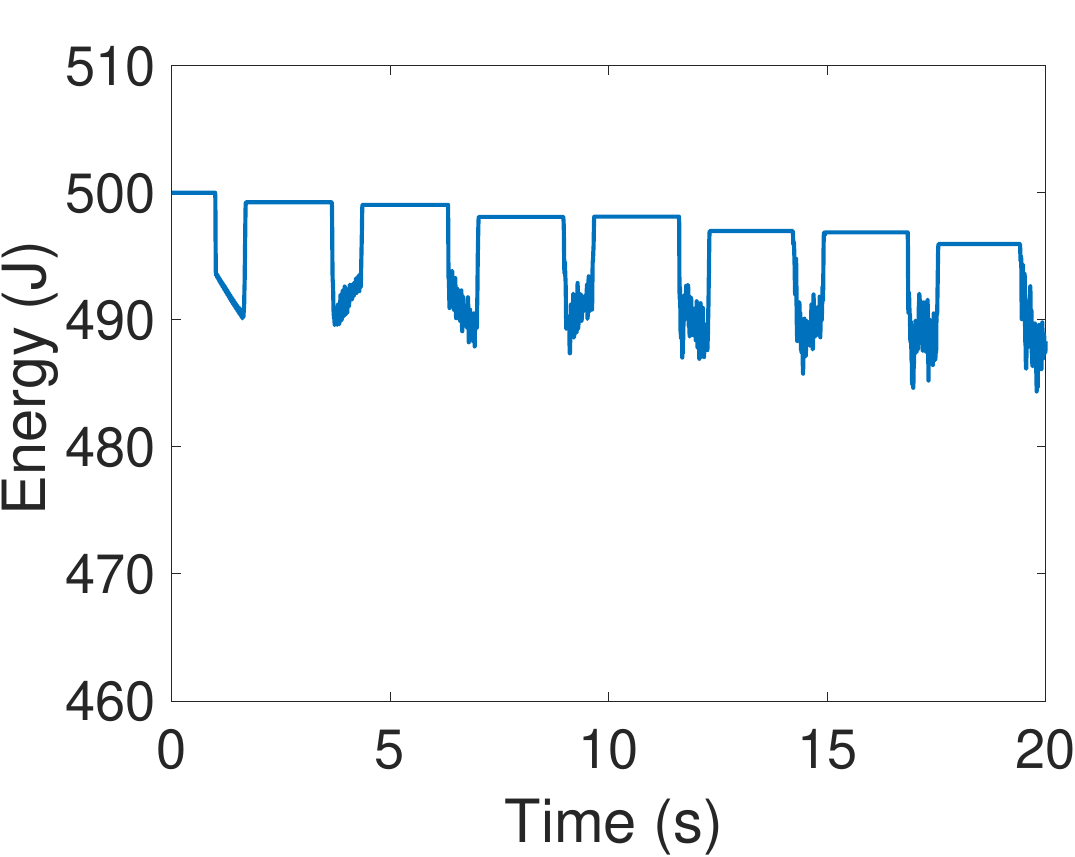}
    }
    \caption{
    \textbf{Bounces of an elastic bar.}
    Comparison between the penalty based contact method in \citetdoyen{} (left) our method (right), both with Newmark time integration and same other settings.}
    \label{fig:1d_bounces}
\end{figure}

\subsection{Refinement Analysis} 

As covered above, reduction in $h$ smooths the barrier for decreasing $\hat{d}$, and so improves stability. This is unsurprising as the contact pressure oscillations we observe are artifacts generated by refining the spatial discretization $\hat{d}$ without accompanying temporal refinement. Correspondingly, to improve accuracy, both of \EIPC's spatial parameters, $\Delta x$ and $\hat{d}$, must jointly be refined with $h$. Here we next analyze convergence under refinement for both the \emph{impact} problem and \emph{bouncing} problems, refining by successively halving $\Delta x$, setting the relationship to the threshold as $\hat{d} =  c_B \Delta x$ (with parameter $c_B$) and time step (as above) with, $h = \nu_C\tfrac{\Delta x}{\sqrt{E/\rho}}$.

\paragraph{Impact Problem}

In \cref{fig:1d_impact_conv}, we consider implicit Newmark time integration for the impact problem with $c_B = 4$, and observe that both displacement and contact pressure converge to the analytic solution. Displacement converges linearly while contact pressure converges sublinearly (rate of $\sim 0.5$). Both rates follow reasonable expectations with contact gap error decreasing linearly w.r.t. $\hat{d}$ and contact pressure given by the barrier energy derivative.
Next we consider BDF-2 time integration with $c_B = 4$. To support BDF-2 (wider time stencil) we provide consistent initialization of displacement and velocity history with the analytic solution at time $t=-h$. In \cref{fig:1d_impact_conv_BDF2}, we see BDF-2 provides comparable convergence to IE for both the displacement and contact pressure. We also note that if we decrease our stability criterion for $\hat{d}$ to $c_B = 1$, convergence rates significantly degrade for both Newmark and BDF-2 due to lack of smoothness in the barrier.

\paragraph{Bouncing Problem} 

Next in \cref{fig:1d_bounces_conv}, to look at longer, time-varying behavior with multiple impacts we consider implicit Newmark time integration for the bouncing problem with $c_B = 8$. Here we observe both displacement (bottom node) and total energy converge linearly to analytical solution. To achieve comparable (linear) convergence with BDF-2 (\cref{fig:1d_bounces_conv_BDF2}) time-integration in this problem requires setting $\nu_C = 0.75$ so that respective time step sizes are halved suggesting that Newmark's improved energy conservation helps in capturing the longer-term behavior of repeated elastic bouncing. 

 \paragraph{Comparison to Constraint-based IPC}
 
 In contrast to constraint-based contact model of the original IPC~\cite{Li2020IPC} formulation, \EIPC provides a consistent discretization of the contact potential in the smooth setting.
Here we consider the resulting, improved convergence behavior for {\EIPC} by considering the original IPC's behavior on the \emph{impact} problem benchmark (same settings as \EIPC above). We begin with \citetIPC{} original model which augments the unconstrained incremental potential with an uncalibrated barrier energy. Instead the barrier stiffness is adaptively and automatically updated to gain improved numerical conditioning of the Hessian. For this original formulation we observe no convergence for both displacement (order=$0.2334$) and contact pressure (order=$0.0823$) in the impact problem. Alternately, if we update the original IPC model to keep the contact barrier stiffness fixed (stiffness selected to match \EIPC simulation at coarsest resolution), convergence improves (displacement order=$0.6997$, contact pressure order=$0.2745$) but is still far from satisfactory.

\begin{figure}[ht]
    \centering
    \includegraphics[width=0.24\linewidth]{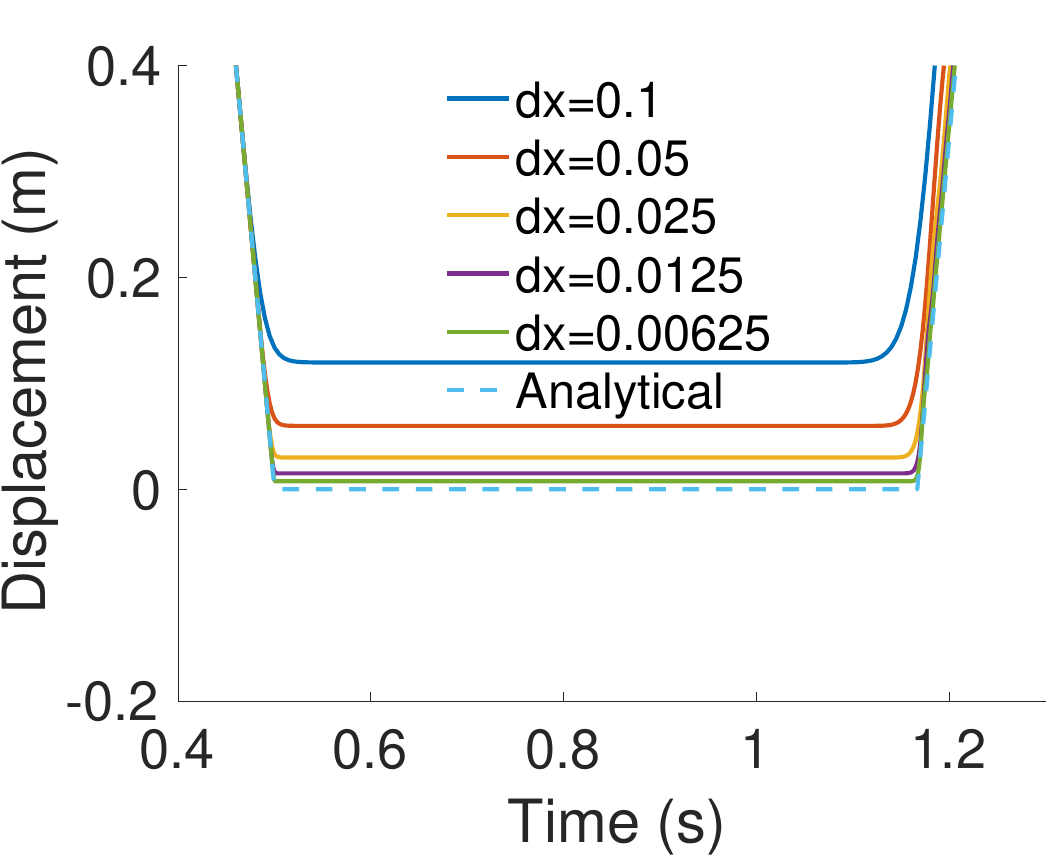} 
    \includegraphics[width=0.24\linewidth]{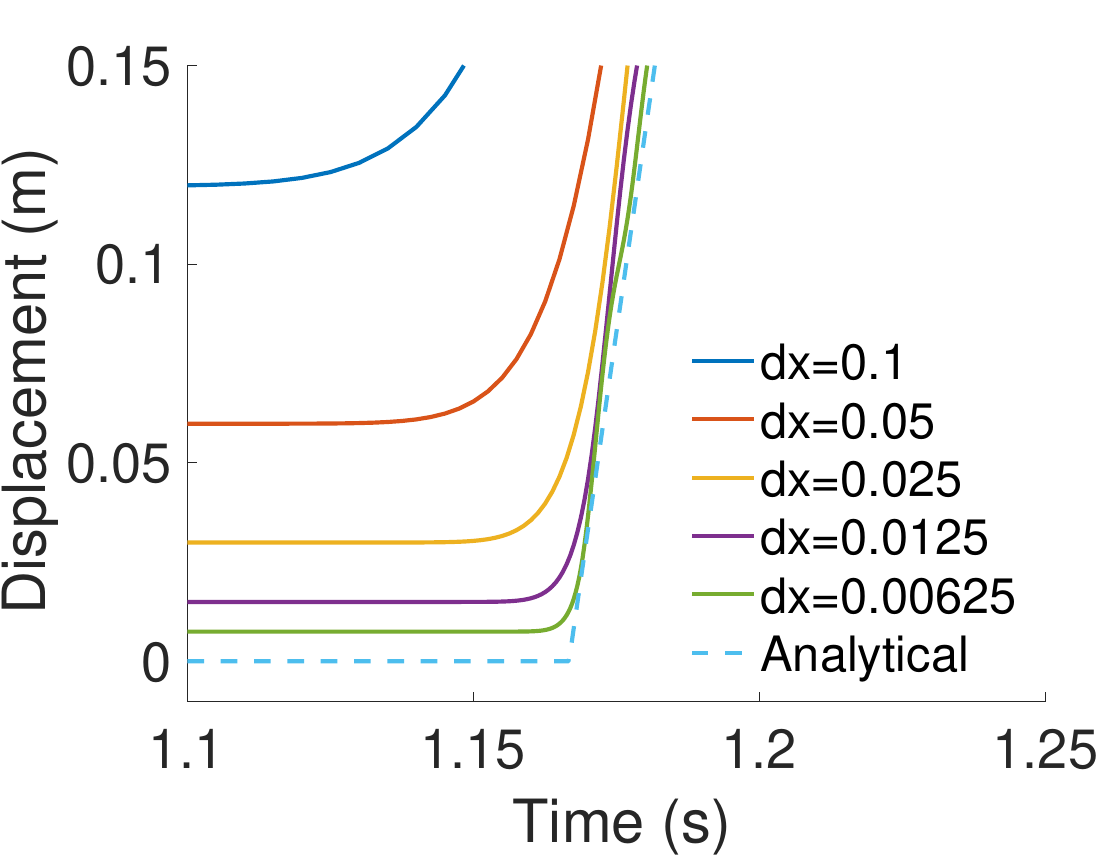}
    \includegraphics[width=0.24\linewidth]{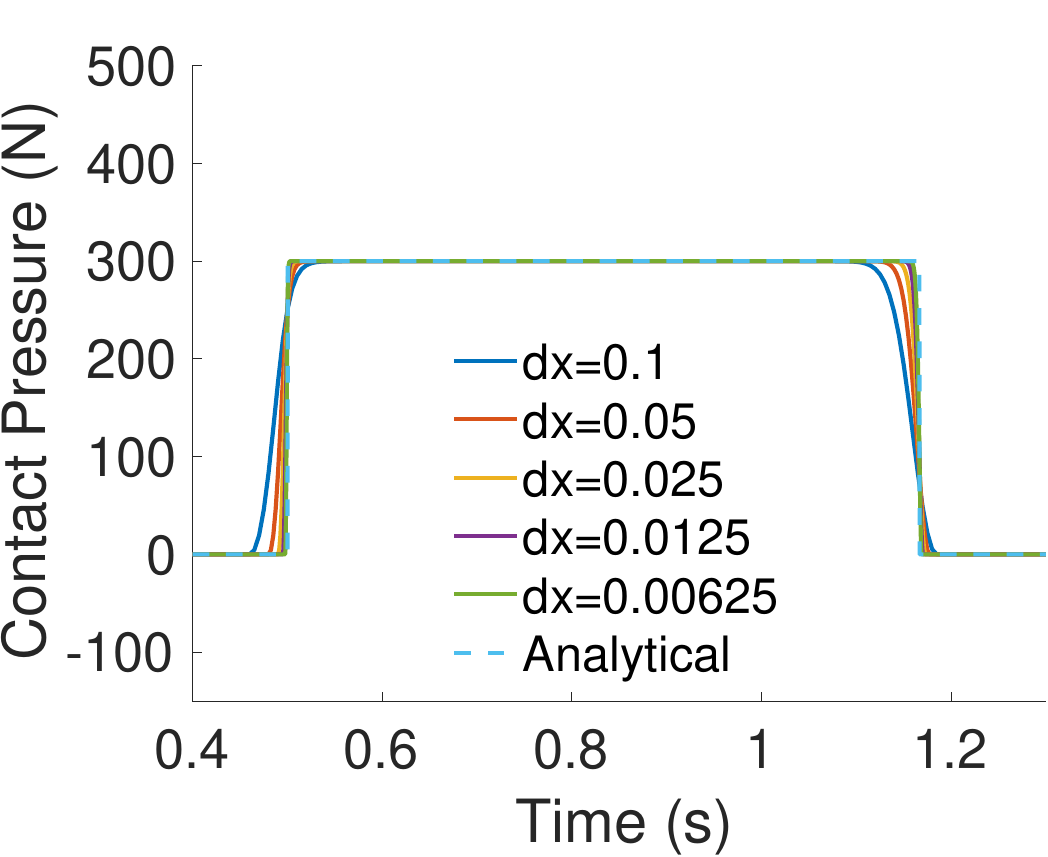} 
    \includegraphics[width=0.24\linewidth]{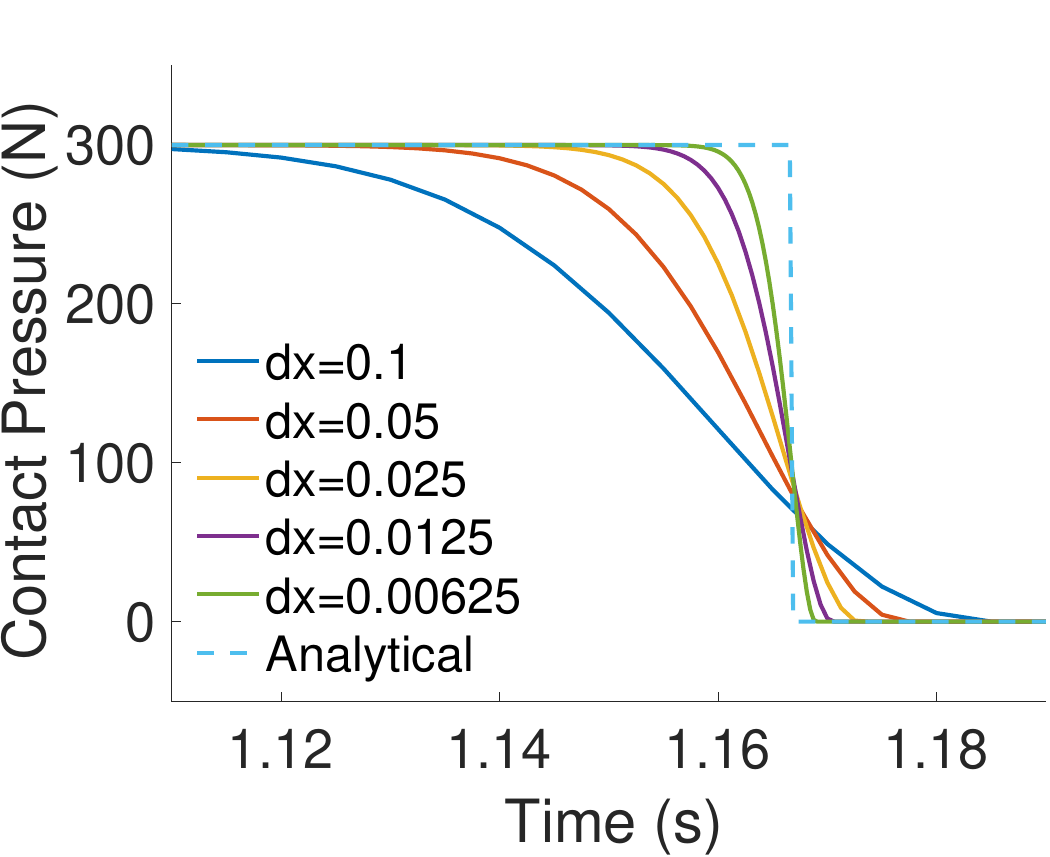}
    \caption{
    \textbf{Impact of an elastic bar -- convergence under refinement (Newmark).}
    Convergence of displacement (left, rate $=1.0389$) and contact pressure (right, rate $=0.4709$) under $\hat{d}$, $\Delta x$, and $h$ refinement with Newmark time integration. Second and fourth column contains zoom-in views of the first and third column plots respectively.}
    \label{fig:1d_impact_conv}
\end{figure}

\begin{figure}[ht]
    \centering
    \includegraphics[width=0.24\linewidth]{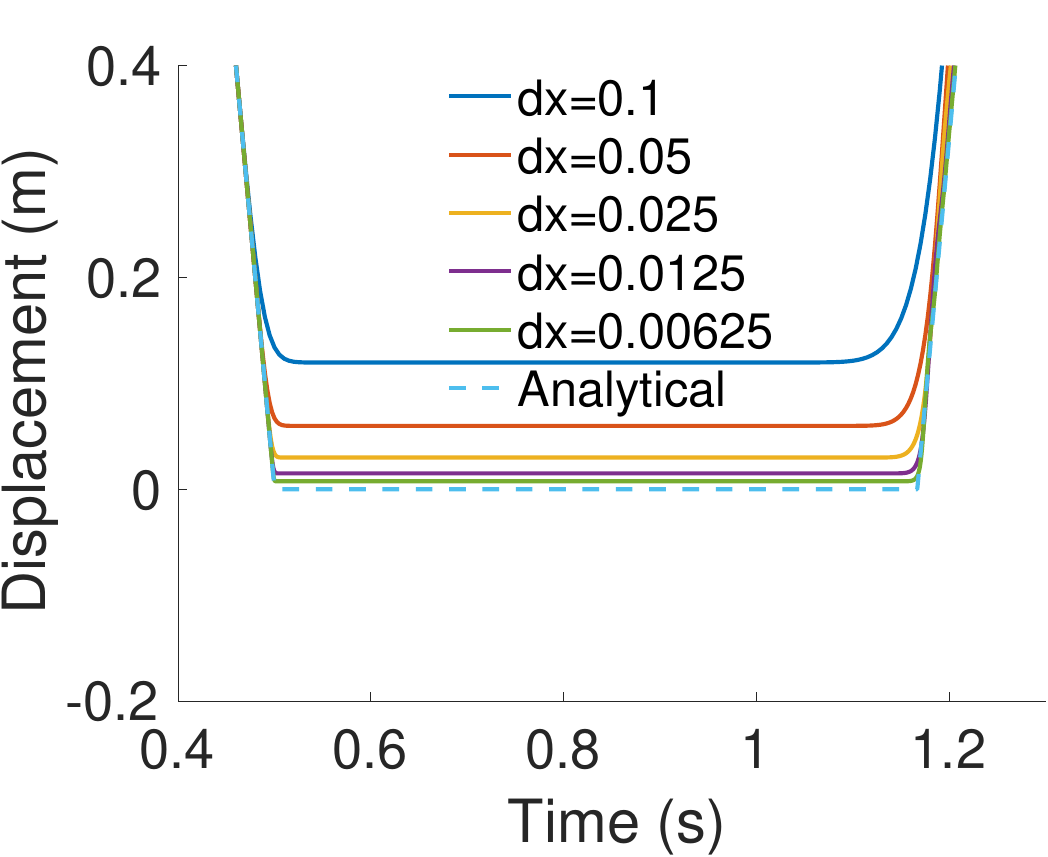} 
    \includegraphics[width=0.24\linewidth]{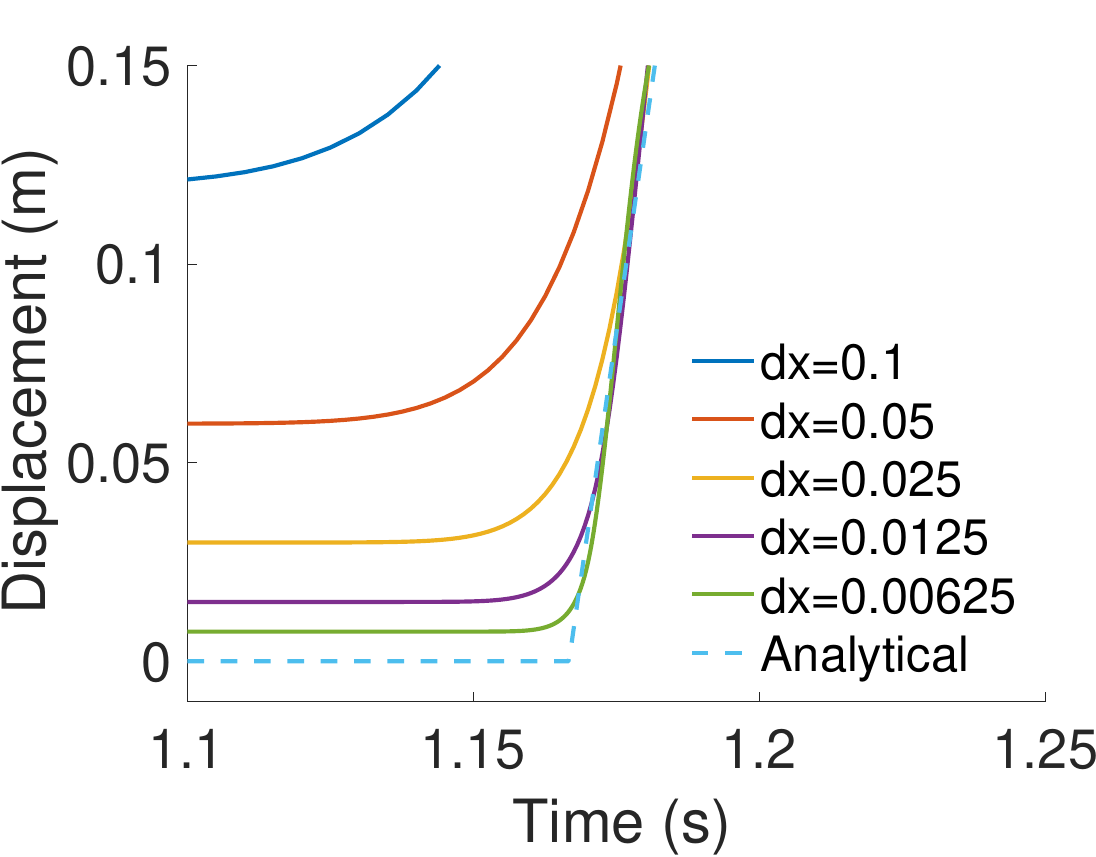}
    \includegraphics[width=0.24\linewidth]{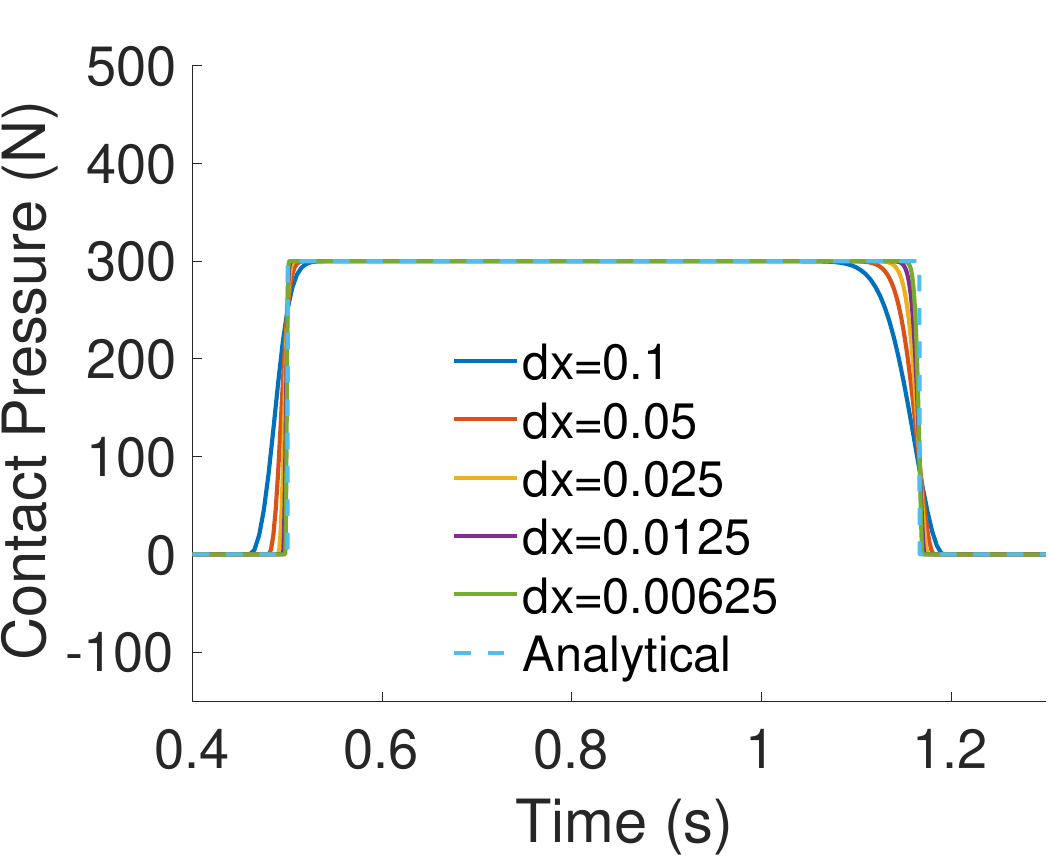} 
    \includegraphics[width=0.24\linewidth]{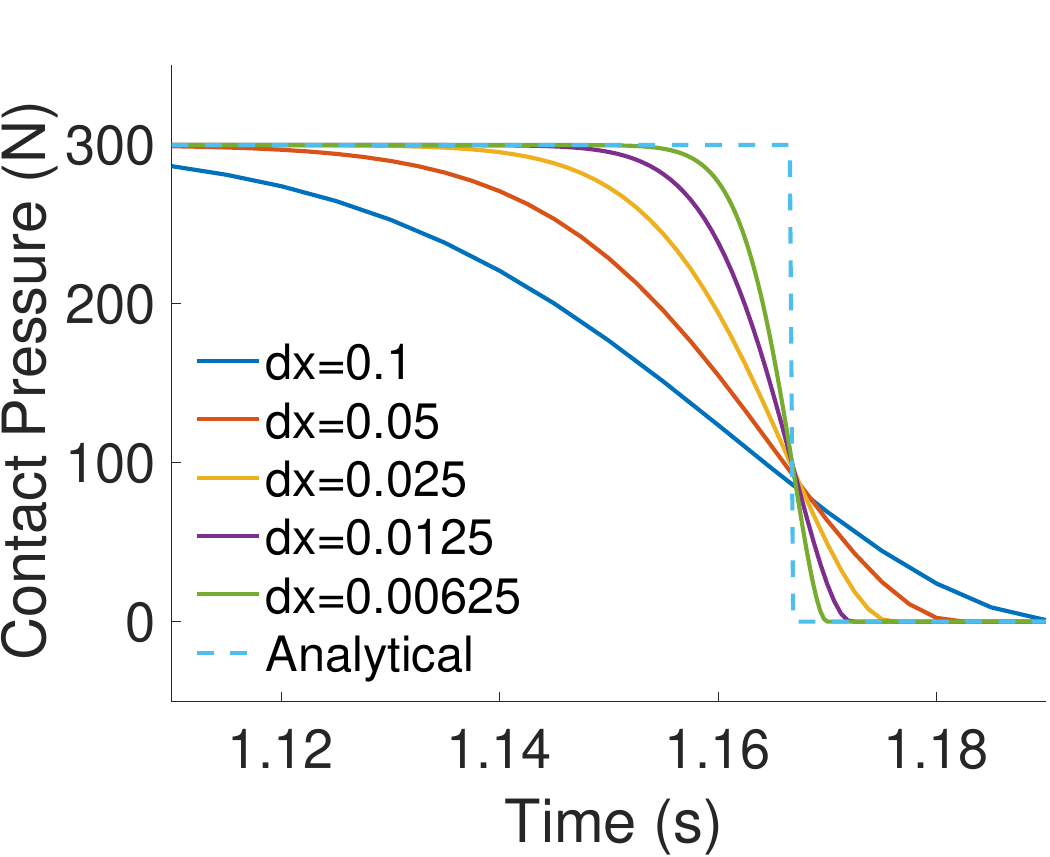}
    \caption{
    \textbf{Impact of an elastic bar -- convergence under refinement (BDF-2).}
    Convergence of displacement (left, rate $=1.0669$) and contact pressure (right, rate $=0.4402$) under $\hat{d}$, $\Delta x$, and $h$ refinement with BDF-2 time integration. Second and fourth column contains zoom-in views of the first and third column plots respectively.}
    \label{fig:1d_impact_conv_BDF2}
\end{figure}

\begin{figure}[ht]
    \centering
    \includegraphics[width=0.24\linewidth]{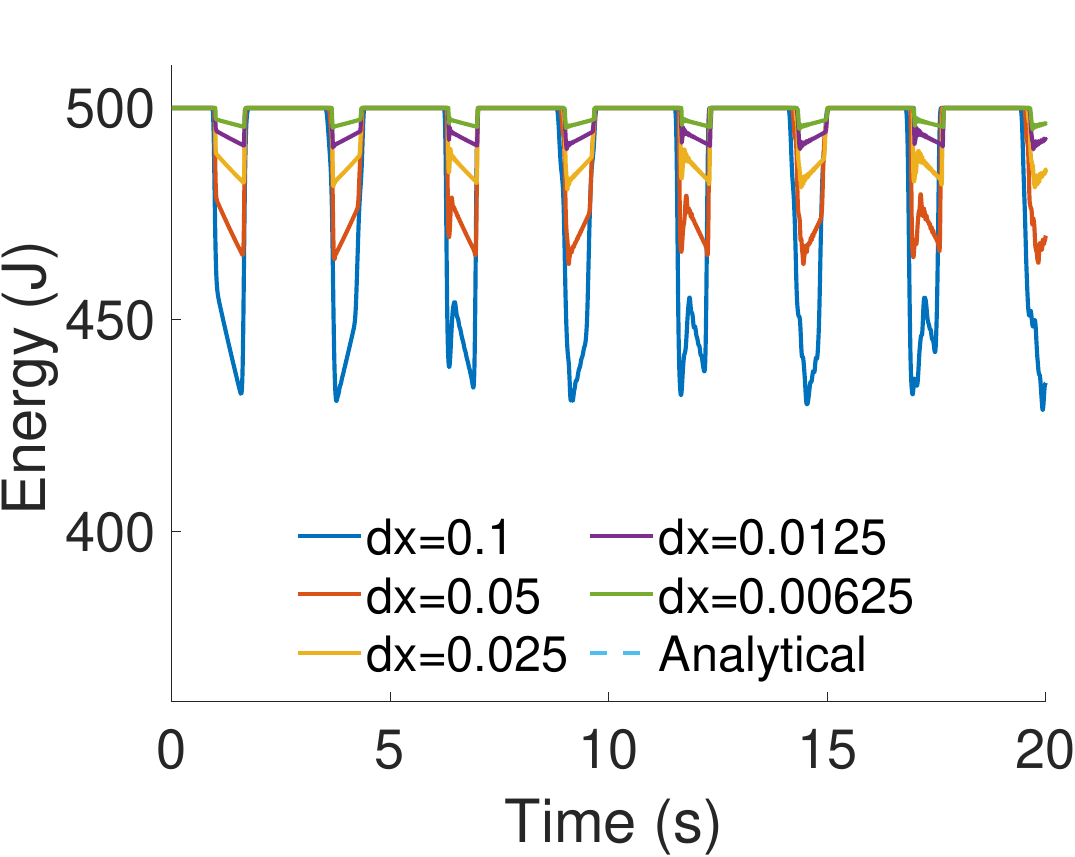}
    \includegraphics[width=0.24\linewidth]{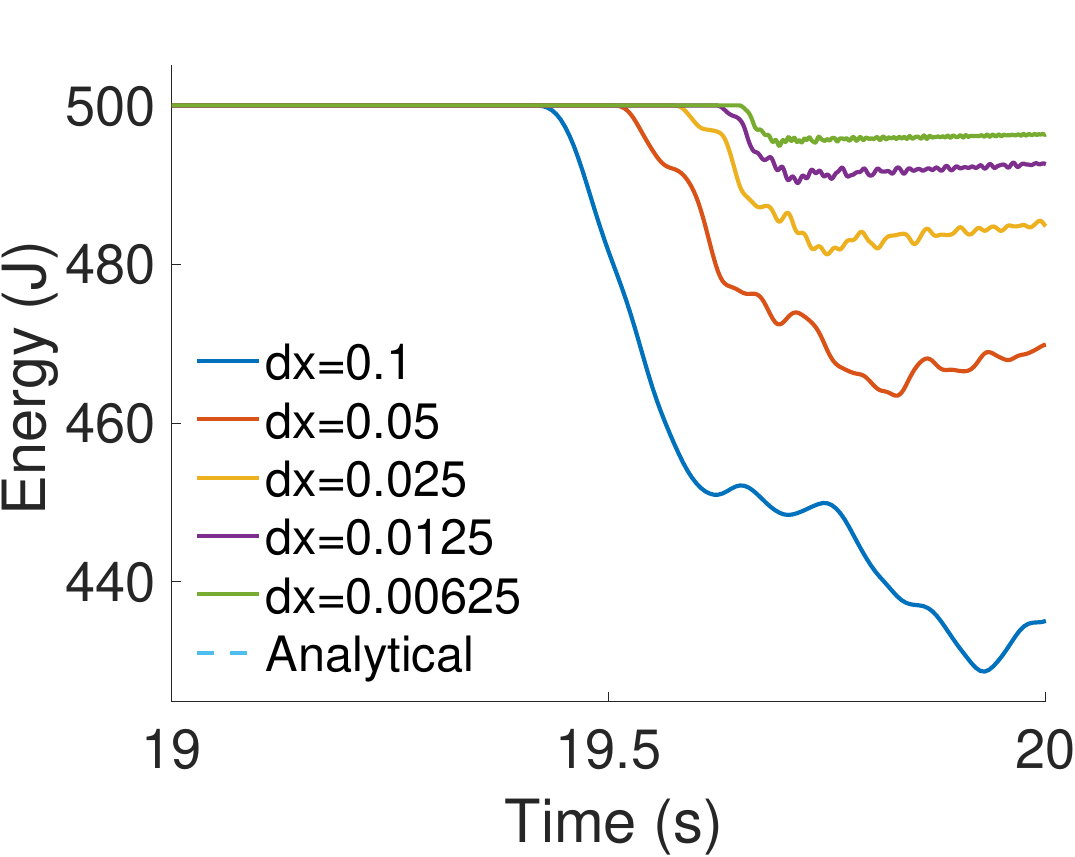}
    \includegraphics[width=0.24\linewidth]{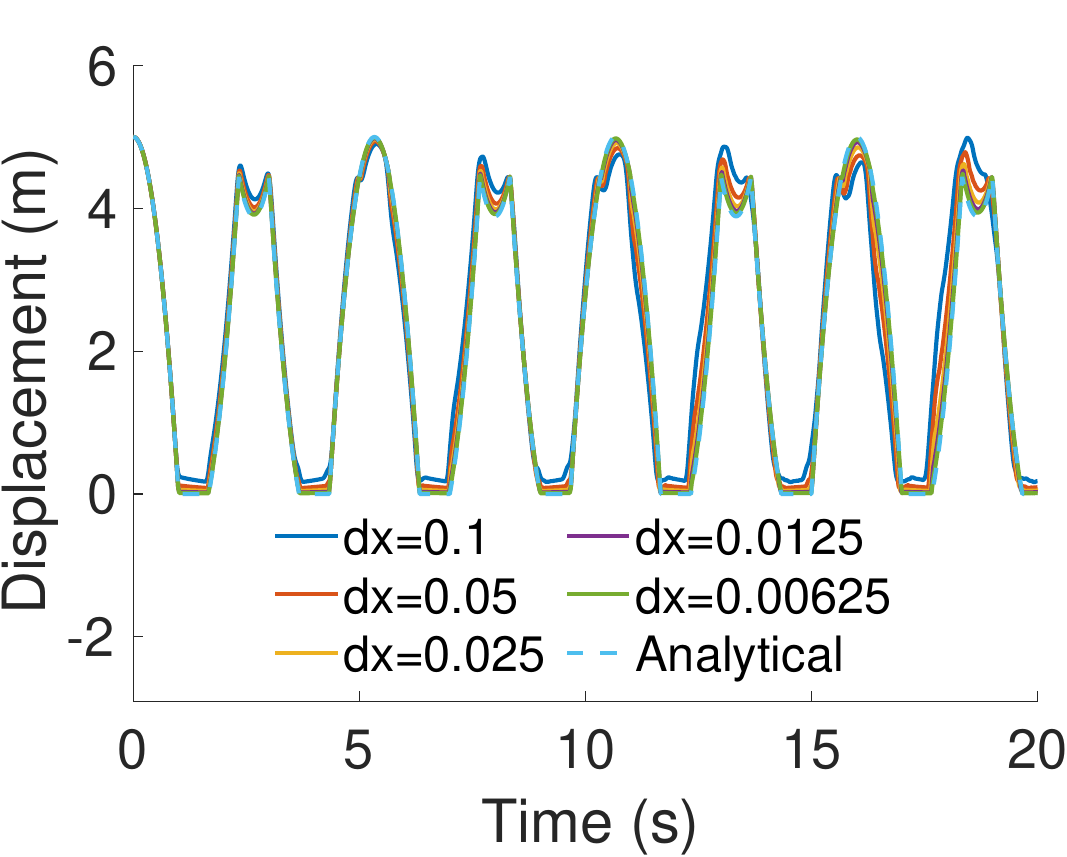}
    \includegraphics[width=0.24\linewidth]{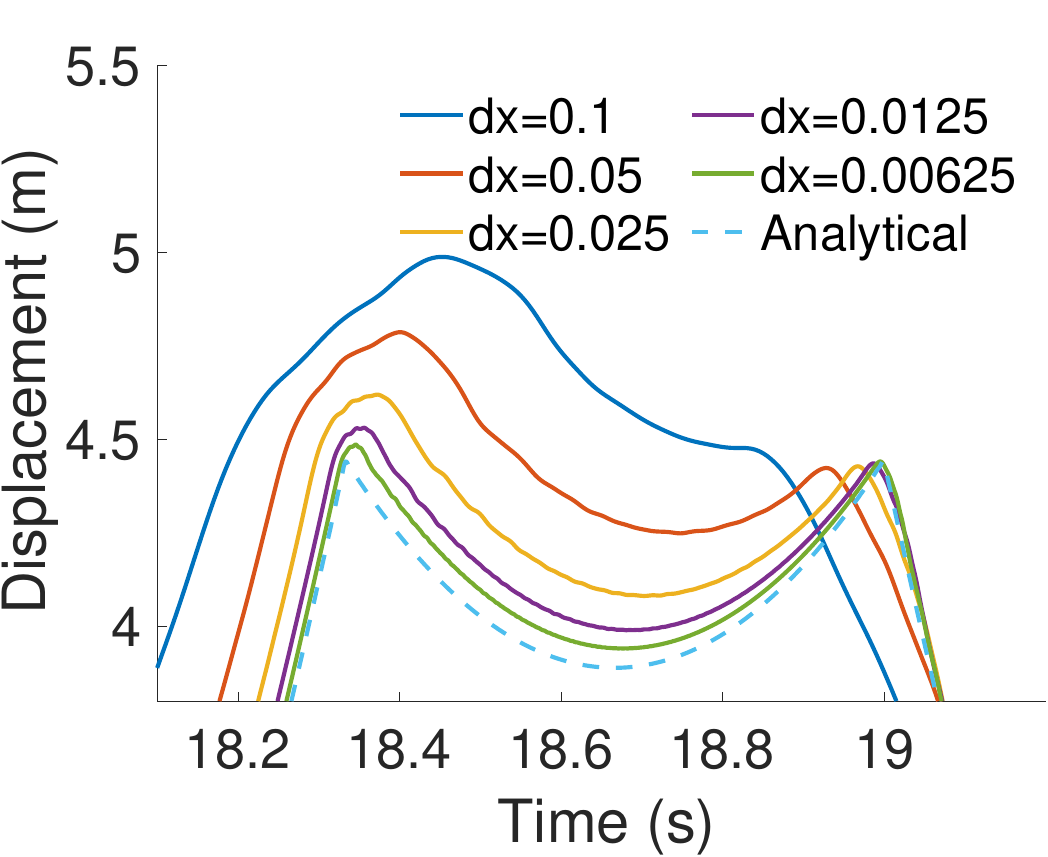}
    \caption{
    \textbf{Bounces of an elastic bar -- convergence under refinement (Newmark).}
    Convergence of energy (top, rate $=0.9877$) and displacement (bottom, rate $=1.0092$) under $\hat{d}$, $\Delta x$, and $h$ refinement with Newmark time integration. Second and fourth column contains zoom-in views of the first and third column plots respectively.}
    \label{fig:1d_bounces_conv}
\end{figure}

\begin{figure}[ht]
    \centering
    \includegraphics[width=0.24\linewidth]{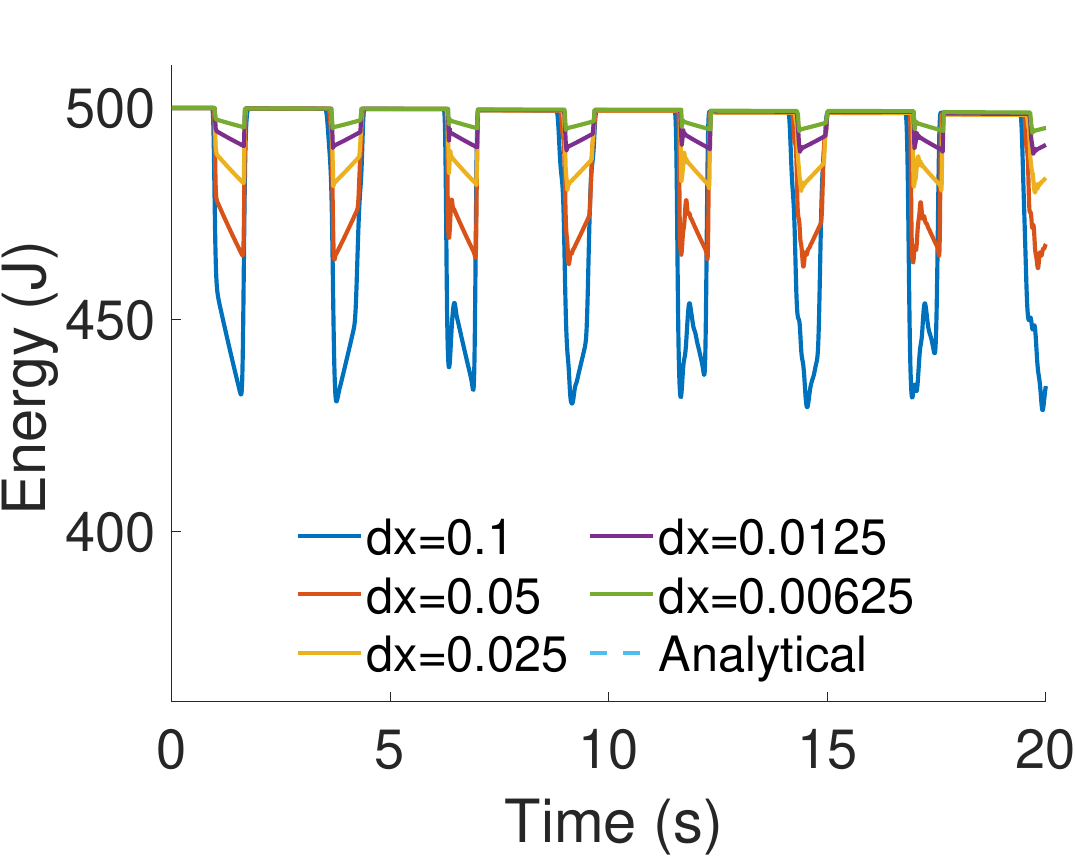}
    \includegraphics[width=0.24\linewidth]{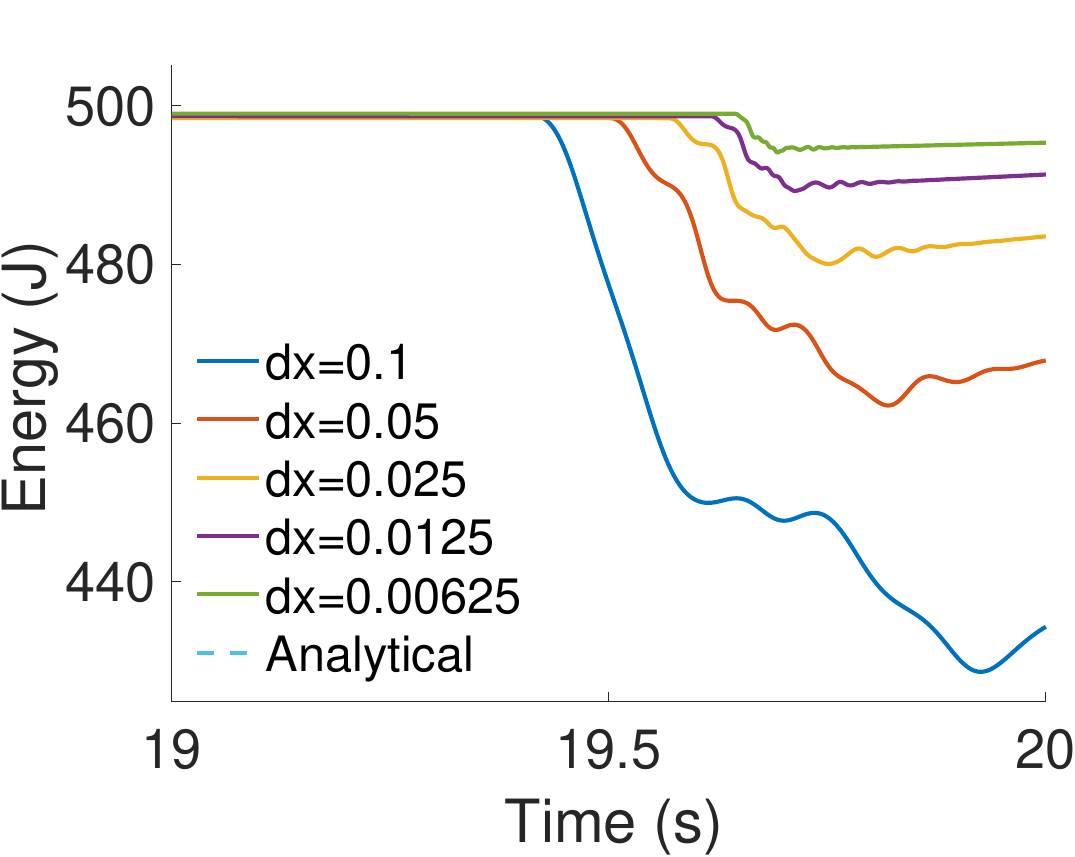}
    \includegraphics[width=0.24\linewidth]{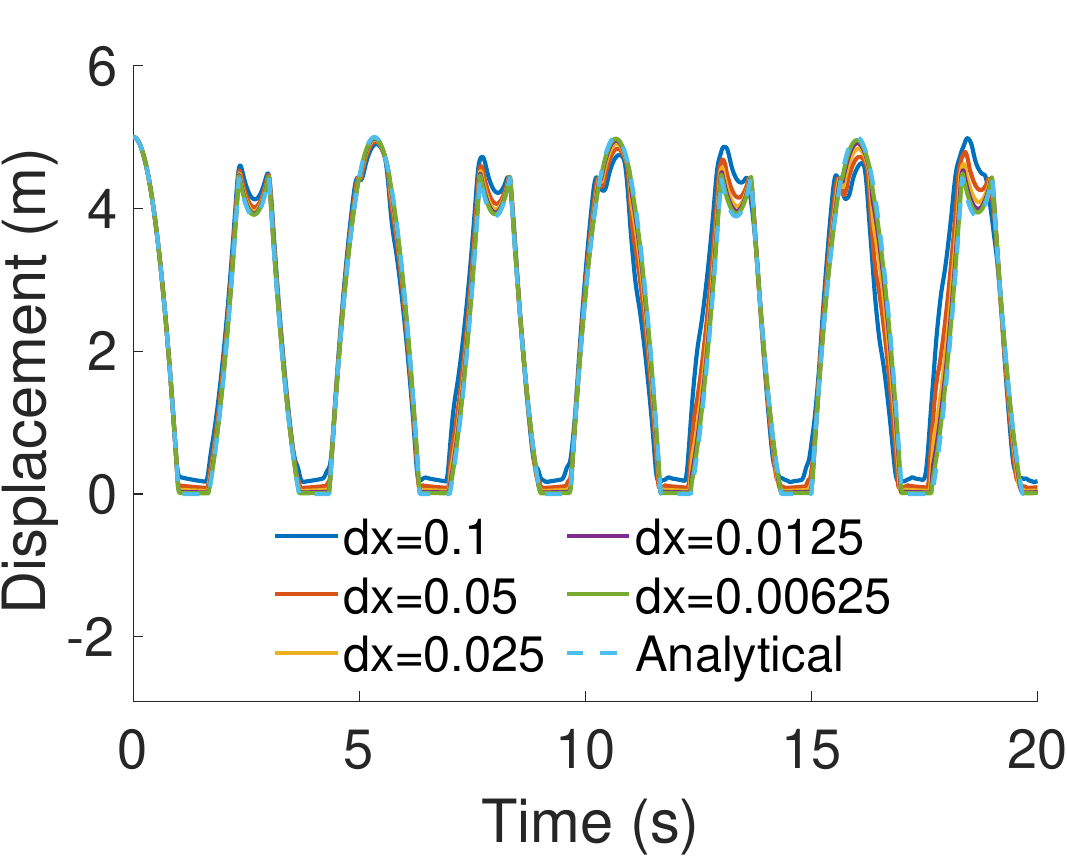}
    \includegraphics[width=0.24\linewidth]{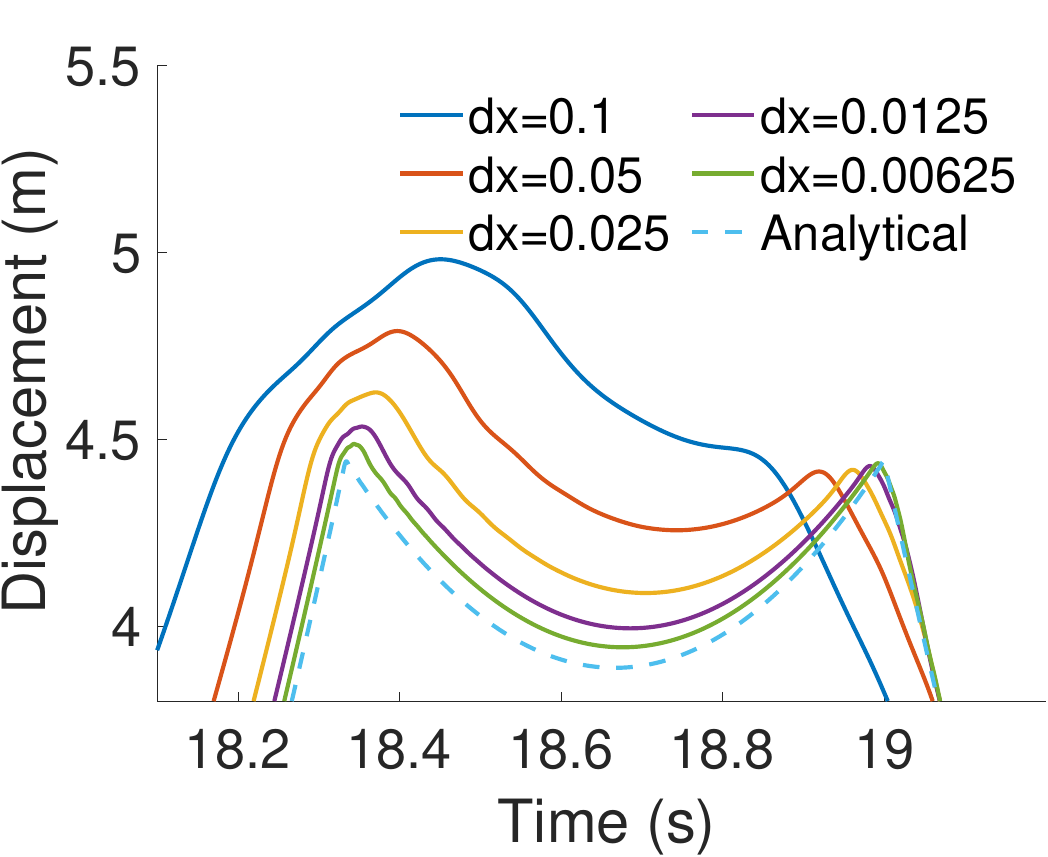}
    \caption{
    \textbf{Bounces of an elastic bar -- convergence under refinement (BDF-2).}
    Convergence of energy (left, rate $=0.9384$) and displacement (right, rate $=0.9661$) under $\hat{d}$, $\Delta x$, and $h$ refinement with BDF-2 time integration. Second and fourth column contains zoom-in views of the first and third column plots respectively.}
    \label{fig:1d_bounces_conv_BDF2}
\end{figure}

\section{Evaluation in 2D and 3D}

In two and three-dimensions, frictional contact now becomes possible and we must consider the contact-interaction of meshed interfaces. Here we first examine the sliding and bouncing behavior of an elastic square on a fixed analytical ground and show that the maximal energy dissipation and displacement and contact pressure curves all converge under refinement just as in our 1D evaluation above. We then demonstrate the accurate capture of stick and slip behaviors under varying friction coefficients by \EIPC with an analytical slope test and show that with \EIPC's consistent smooth approximation to the max operator, the vertical displacement of a square slipping on a fixed meshed ground can converge to a straight line with only spatial refinement of the mesh boundary. We then further consider frictional benchmark tests and close with challenging geometric collision ``stress-tests'', a large-deformation high-speed dynamic collision problem, and an application to the analysis of compressed microstructure testing.

\subsection{Refinement in 2D}

\label{sec:2d_slides}

\paragraph{Block on ground}
We first consider refinement of a slower-speed contact problem in 2D with a $\SI{2}{\meter}$-wide square, initialized to a height immediately ($\hat{d}$) above a fixed analytical ground without friction. We use a nonlinear (neo-Hookean) material with Young's: $E=\SI{4000}{\newton/\meter}$, Poisson: $\nu=0.2$, and density $\rho=\SI{100}{\kilo\gram/\meter^2}$; gravity is $g=\SI{-5}{\meter/\second^2}$, and time step is set by $\nu_C=1.5$. The square is uniformly and symmetrically tessellated with $\Delta x$. Under gravity, this soft square will compress while its bottom interface slides periodically back and forth along the ground.
Fixing the relations $\hat{d} = 0.5\Delta x$ and $h = \nu_C\frac{\Delta x}{\sqrt{E/\rho}}$, we perform  refinements by half down to $\Delta x = 0.0125m$. We measure system energy, (central top node's) vertical displacement, and (at center bottom node) contact pressure over time. Applying BDF-2 time integration, all above measures converge linearly (to finest solution -- no analytic model is available) as resolution increases despite the nonlinear elasticity applied (\cref{fig:2d_slides}).

\begin{figure}[ht]
    \centering
    \includegraphics[width=0.175\linewidth]{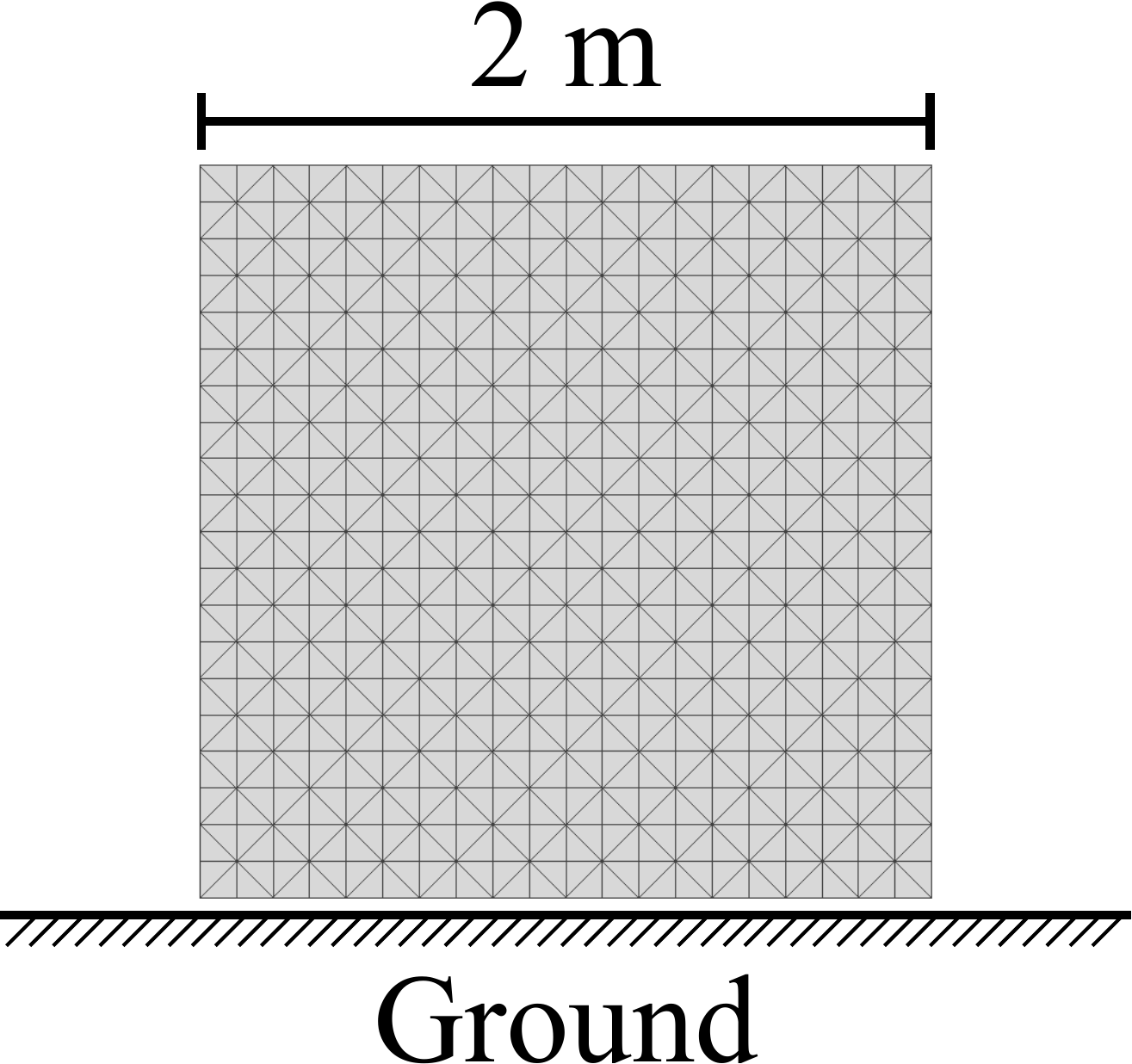} %
    \includegraphics[width=0.265\linewidth]{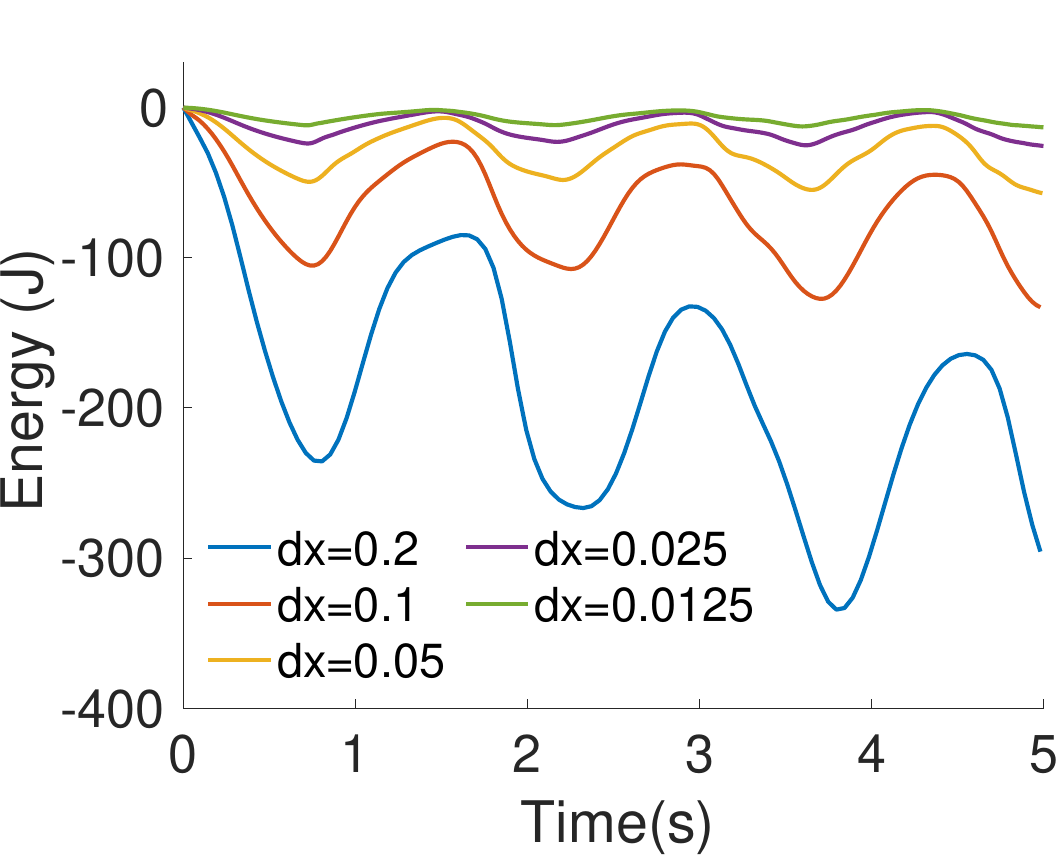}
    \includegraphics[width=0.265\linewidth]{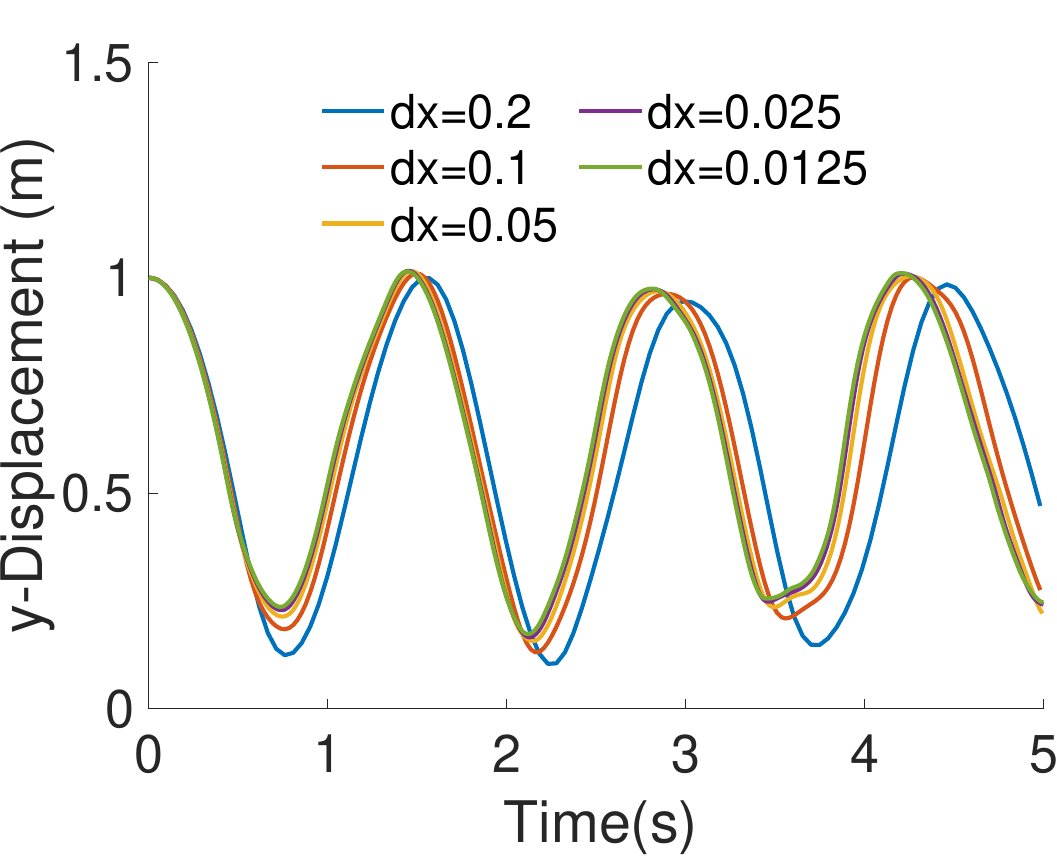} %
    \includegraphics[width=0.265\linewidth]{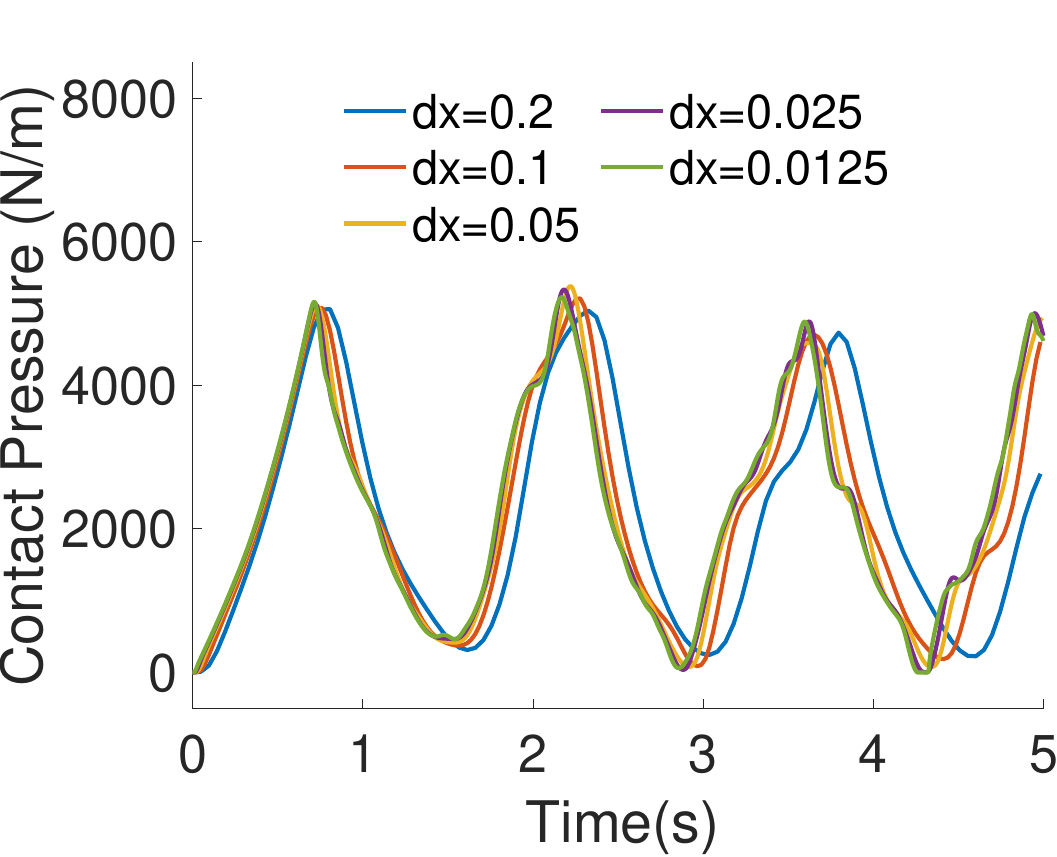}
    \caption{
    \textbf{Static block on ground.}
    From left to right: experiment setup ($\Delta x = 0.1m$) and convergence of energy (rate $=1.5688$), displacement (rate $=1.3850$), and contact pressure (rate $=1.1952$) under $\hat{d}$, $\Delta x$, and $h$ refinement with BDF-2 time integration.}
    \label{fig:2d_slides}
\end{figure}

\paragraph{Impact and bouncing on ground}
We next consider refinement with higher-speed impacts and repeated bouncing in 2D. We extrude the \emph{bouncing} benchmark problem from 1D (\cref{sec:1d_bounces}), setting the stiffness $2\times$ as large, $\nu_C = 0.75$, and apply BDF-2 time integration. This gives a $\SI{10}{\meter}$-wide square initialized $\SI{5}{\meter}$ above a fixed analytical ground (again no friction). The square is uniformly and symmetrically tessellated by $\Delta x$, with $\rho = \SI{1}{\kilo\gram/\meter^2}$, $E = \SI{1800}{\newton/\meter}$, $\nu=0.2$, applying neo-Hookean elasticity (ensuring no element inversion). Gravity is set to $g=\SI{-10}{\meter/\second^2}$ and $\kappa = 0.1E$.
During simulation, over repeated bounces, gravitational energy progressively transfers to elasticity energy as the highest bouncing point decreases and high-frequency elastic waves become more pronounced.
Repeating the same refinement as for the ``block-on-ground'' above, we now observe that all measurements converge with resolution increase (\cref{fig:2d_bounces}). 
However, as the simulation continues,  high-frequency elastic waves magnify, so that the simulation becomes less stable and the curves at changing resolution diverge increasingly from accumulated errors.

\begin{figure}[ht]
    \centering
    \includegraphics[width=0.18\linewidth]{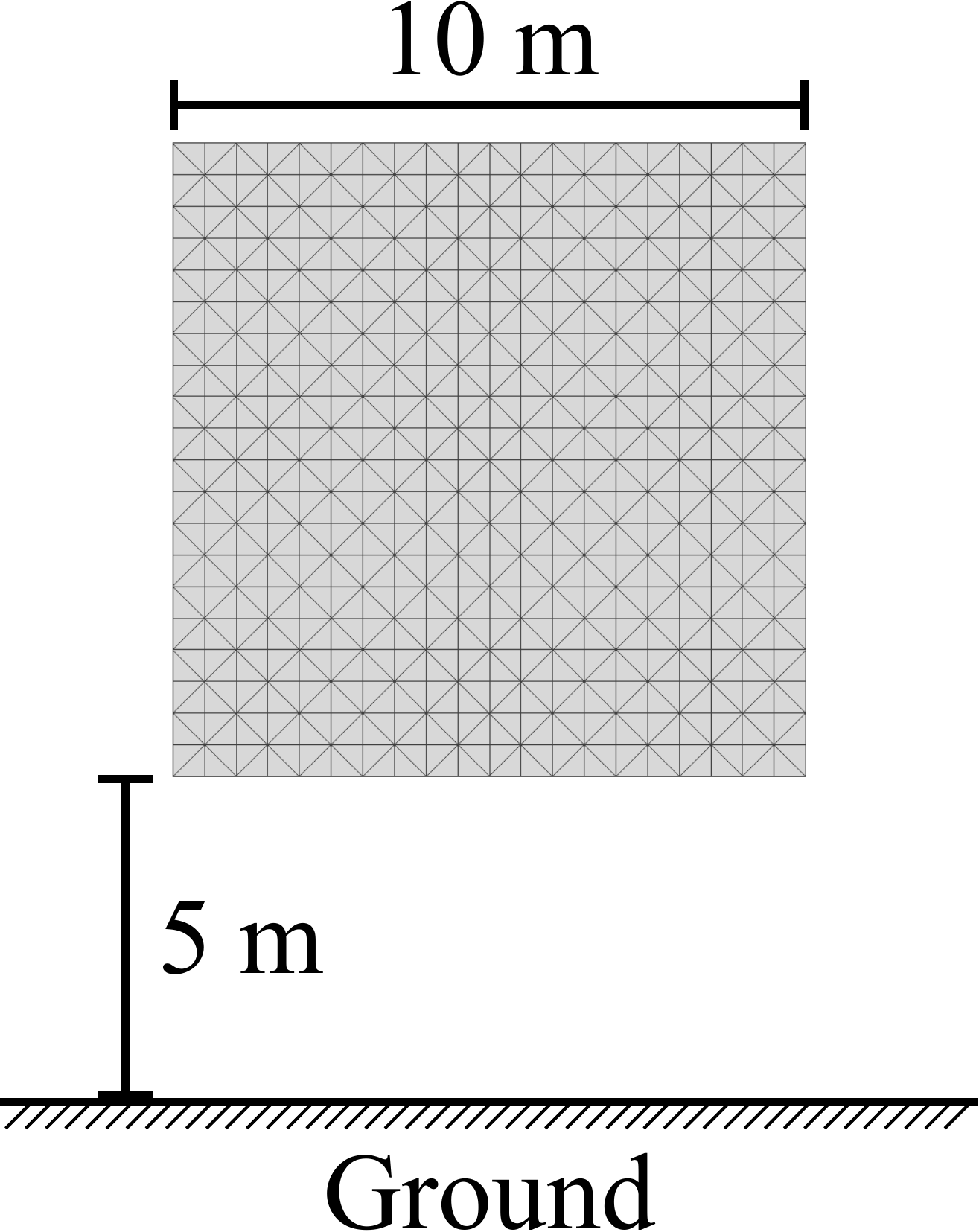}%
    \includegraphics[width=0.27\linewidth]{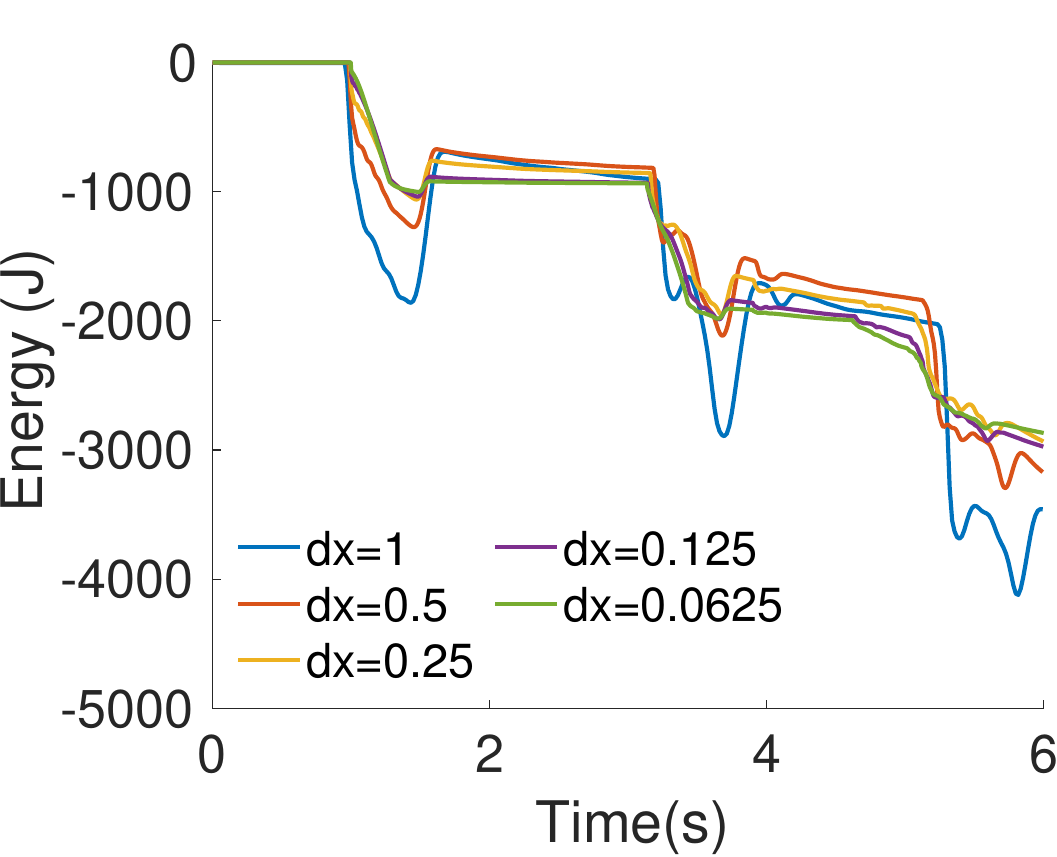}
    \includegraphics[width=0.27\linewidth]{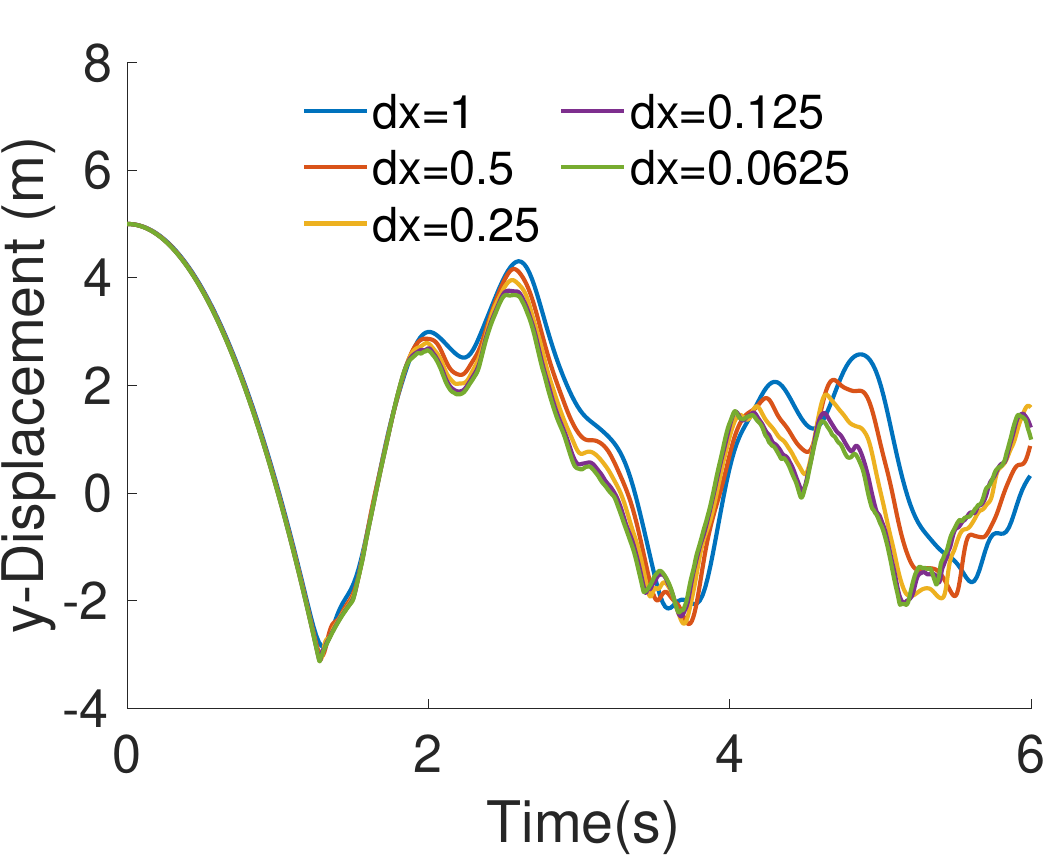}%
    \includegraphics[width=0.27\linewidth]{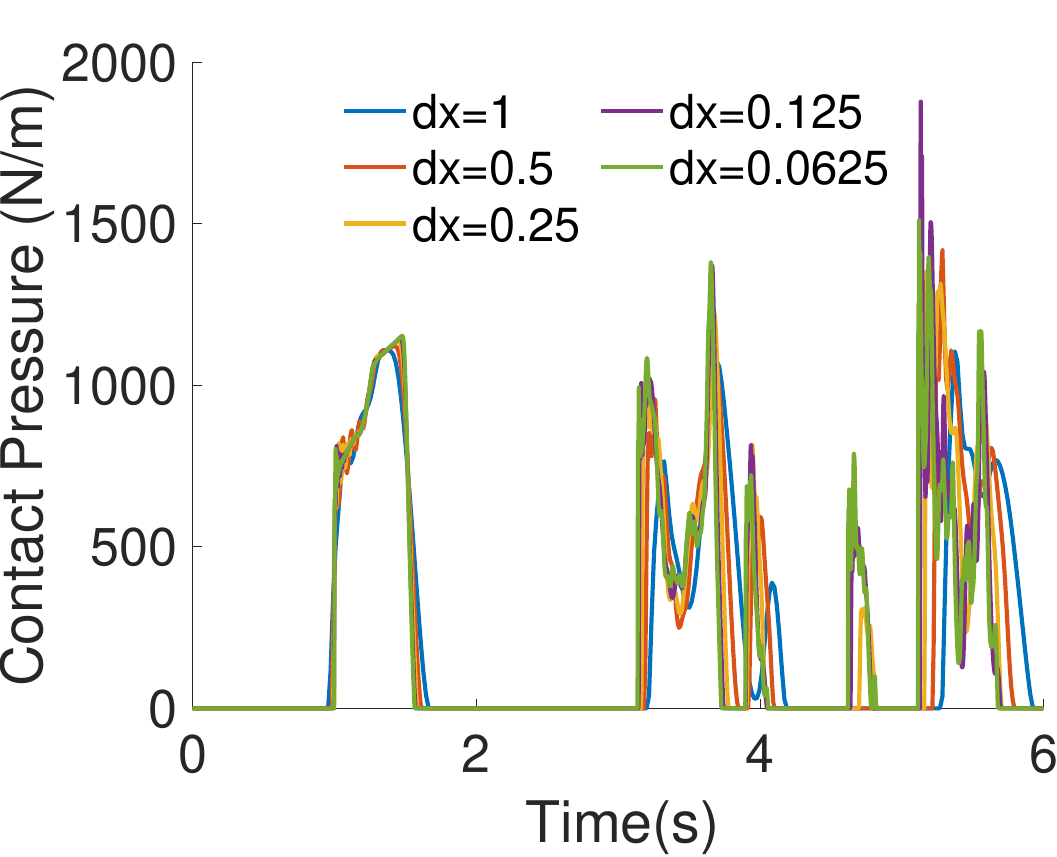}
    \caption{
    \textbf{Dynamic block on ground.}
    From left to right: experiment setup ($\Delta x = 0.5m$) and
    convergence of energy (rate $=1.0718$), displacement (rate $=1.1527$), and contact pressure (rate $=0.5065$) under $\hat{d}$, $\Delta x$, and $h$ refinement with BDF-2 time integration.}
    \label{fig:2d_bounces}
\end{figure}

\subsection{Tessellation Error: A Sliding Block on Meshed Boundary}

\subsubsection{2D}

To compare and verify the direct summation approximation~\cite{Li2020IPC} and our consistent approximation to the max operator, we test an example with a $\SI{2}{\meter}$-wide square sliding on a fixed $\SI{16}{\meter}$-wide meshed ground ($\mu=0$).
The square is placed right $\hat{d}=\SI{0.1}{\meter}$ above the ground in the middle with an initial velocity $v_0 = \SI{1}{\meter/\second}$. It has Young's modulus $E = \SI{2e11}{\newton/\meter}$, Poisson's ratio $\nu = 0.3$, and density $\rho = \SI{8000}{\kilo\gram/\meter^2}$, nearly rigid. The gravity is $g=\SI{-5}{\meter/\second^2}$, and we set $\kappa = \SI{1e6}{\newton/\meter}$ and fix the time step size at $h=\SI{0.01}{\second}$. Both the square and the single layer ground are uniformly tessellated with $\Delta x$.

With $\Delta x = \SI{2}{\meter}$, after the initial drop for acquiring contact forces in IPC framework, we clearly see the jumps on the horizontal displacement curve given by direct summation everytime when the square corner is crossing a node on the ground (\cref{fig:2d_slip} top). The arc between the jumps are due to the ground point to square edge contact pair, which hold the square at different location at bottom, forming unbalanced force distributions during slipping. Note that these all only happen within the scale of $\SI{e-3}{\meter}$, nearly 2 orders-of-magnitude smaller than $\hat{d}$.
Our consistent approximation still have jumps but the magnitude is much smaller (\cref{fig:2d_slip} bottom), and the arcs have very similar profile.

\begin{figure}[ht]
    \centering
    \parbox{0.48\linewidth}{
        \centering
        {\small Direction summation~\cite{Li2020IPC}}\\\vspace{5pt}
        \includegraphics[width=0.49\linewidth]{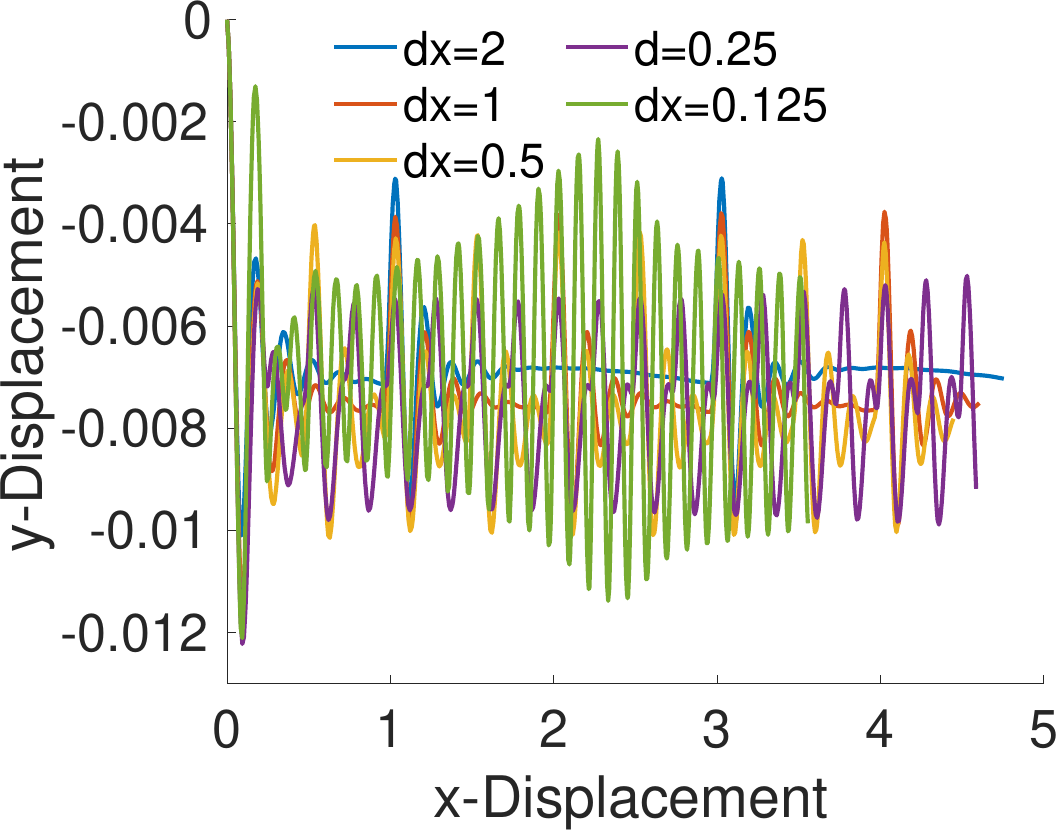}
        \includegraphics[width=0.49\linewidth]{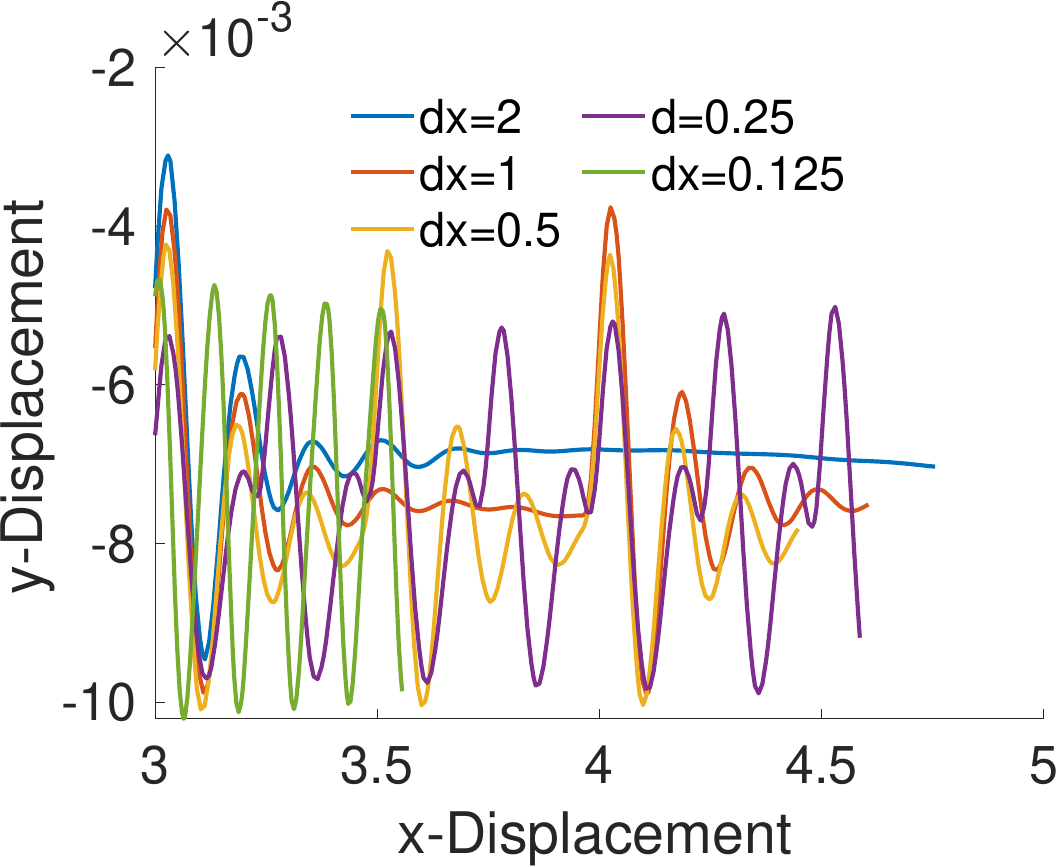}
    }
    \hspace{0.02\linewidth}
    \parbox{0.48\linewidth}{
        \centering
        {\small Our subtraction-based method}\\\vspace{5pt}
        \includegraphics[width=0.49\linewidth]{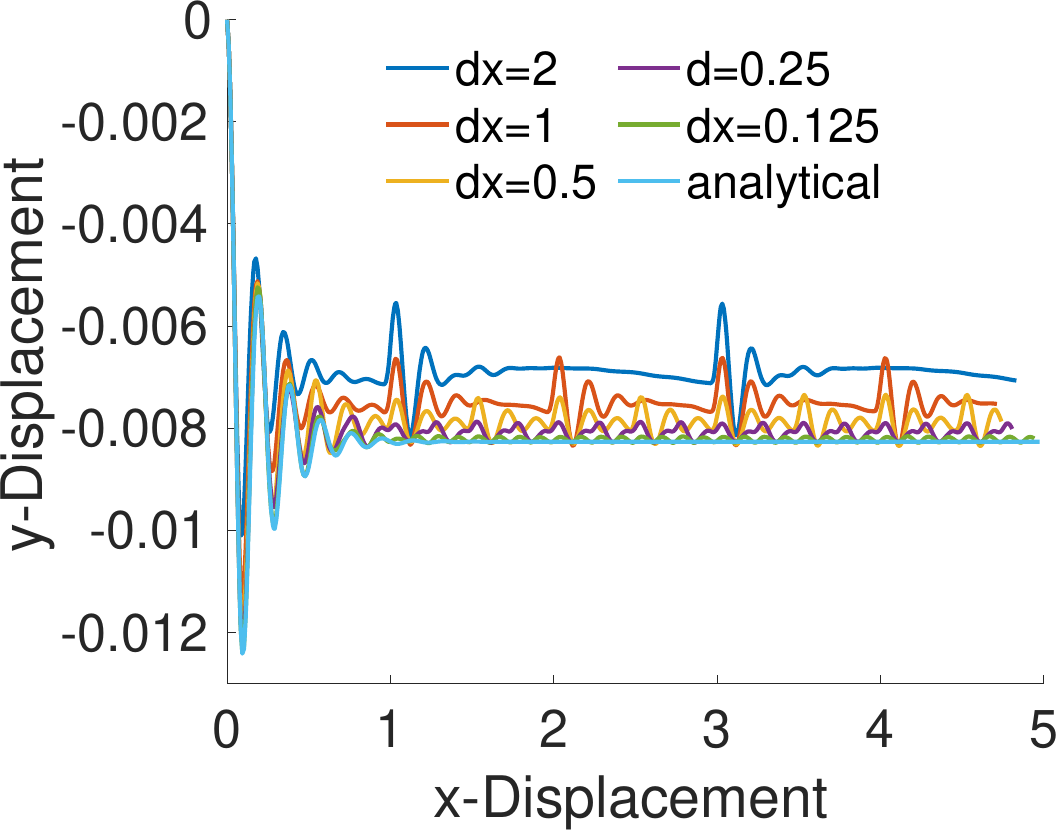}
        \includegraphics[width=0.49\linewidth]{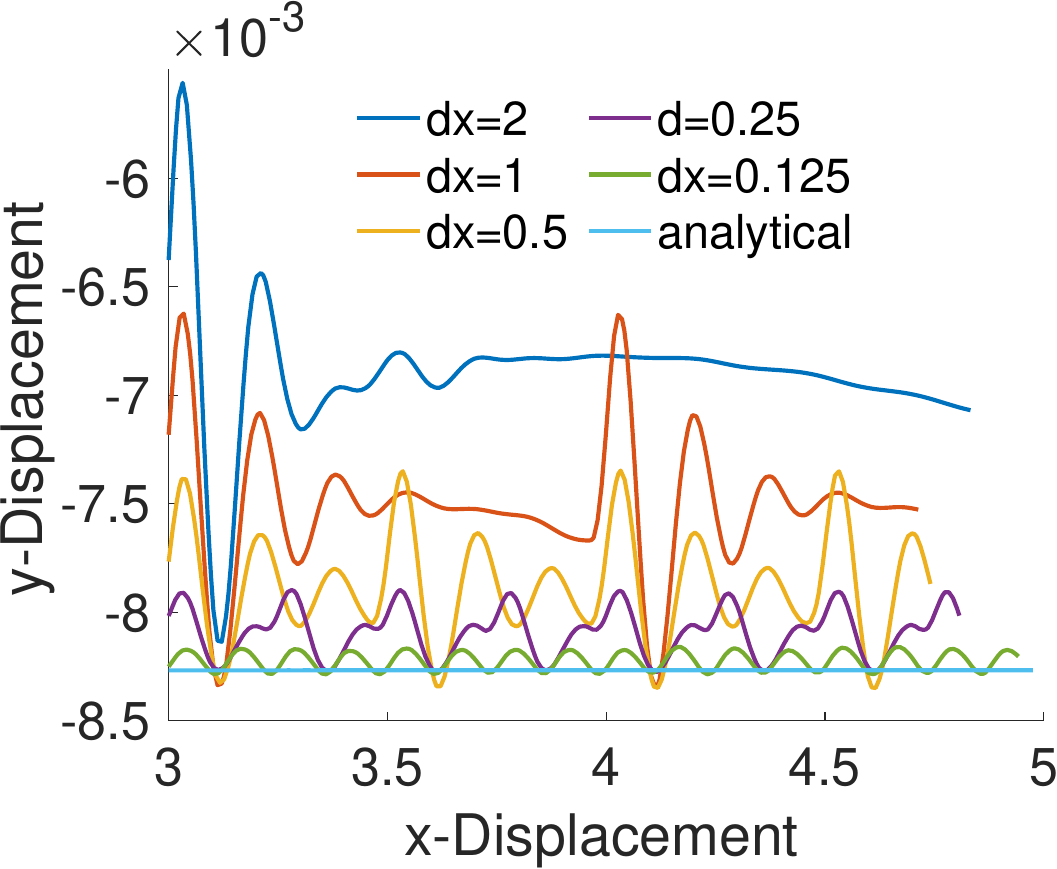}
    }
    \caption{
    \textbf{Sliding block on meshed boundary (2D).}
    Displacement plots of a block sliding on a meshed plane with the max operator approximated via direct summation~\cite{Li2020IPC} (left, nonconvergent) and our subtraction-based method (right, converging to the straight line at order = $1.0761$). Second and fourth column contains zoom-in views of the first and third column plots respectively.}
    \label{fig:2d_slip}
\end{figure}

This is because with a slipping square, the jumps also come from the activation of the ground point to square edge pairs during slipping in addition to the duplicate square point to ground edge pair when applying direct summation. As we show in \cref{fig:2d_slip} bottom, our approximation with only the first source of jumps generates results converging to a straight line under only $\Delta x$ refinement with $\hat{d}$ and time step size $h$ fixed at $0.1m$ and $0.01s$ respectively. 
But the second source of jumps does not converge as shown in \cref{fig:2d_slip} top with direct summation.
Although by refining $\hat{d}$ at fixed $\Delta x$, the portion of the duplication in direct summation vanishes, once a refinement has a fixed $\hat{d}/\Delta x$ ratio, the portion is then also fixed and thus not vanishing.

\subsubsection{3D} \label{sec:3D-convergence}

In 3D, it becomes more complicated with edge-edge stencils. We extrude the sliding experiment setup in 2D to 3D to check the convergence behavior for a cube sliding on a 3D plane with both point-triangle and edge-edge stencils.

We set $\hat{d} = 10^{-3}m$, orders-of-magnitude smaller than $\Delta x$, so that no extra duplication of the potential field from nearby edges is possible from edge-edge stencils. This enables convergence to a straight line, but not the reference solution obtained by sliding with $\Delta x = 0.25$ and $\hat{d} = 10^{-3}m$ on an analytical plane (\cref{fig:3d_sliding} left). This is because for two meshed planes touching each other, each edge would result in multiple edge-edge stencils, thus over integrating the quantity using our edge weights.
If $\hat{d}$ is also refined starting from $10^{-3}m$ linearly w.r.t. $\Delta x$, our results converge to the analytical solution $y=0$ nearly linearly (\cref{fig:3d_sliding} right).

\begin{figure}[ht]
    \centering
    \includegraphics[width=0.4\linewidth]{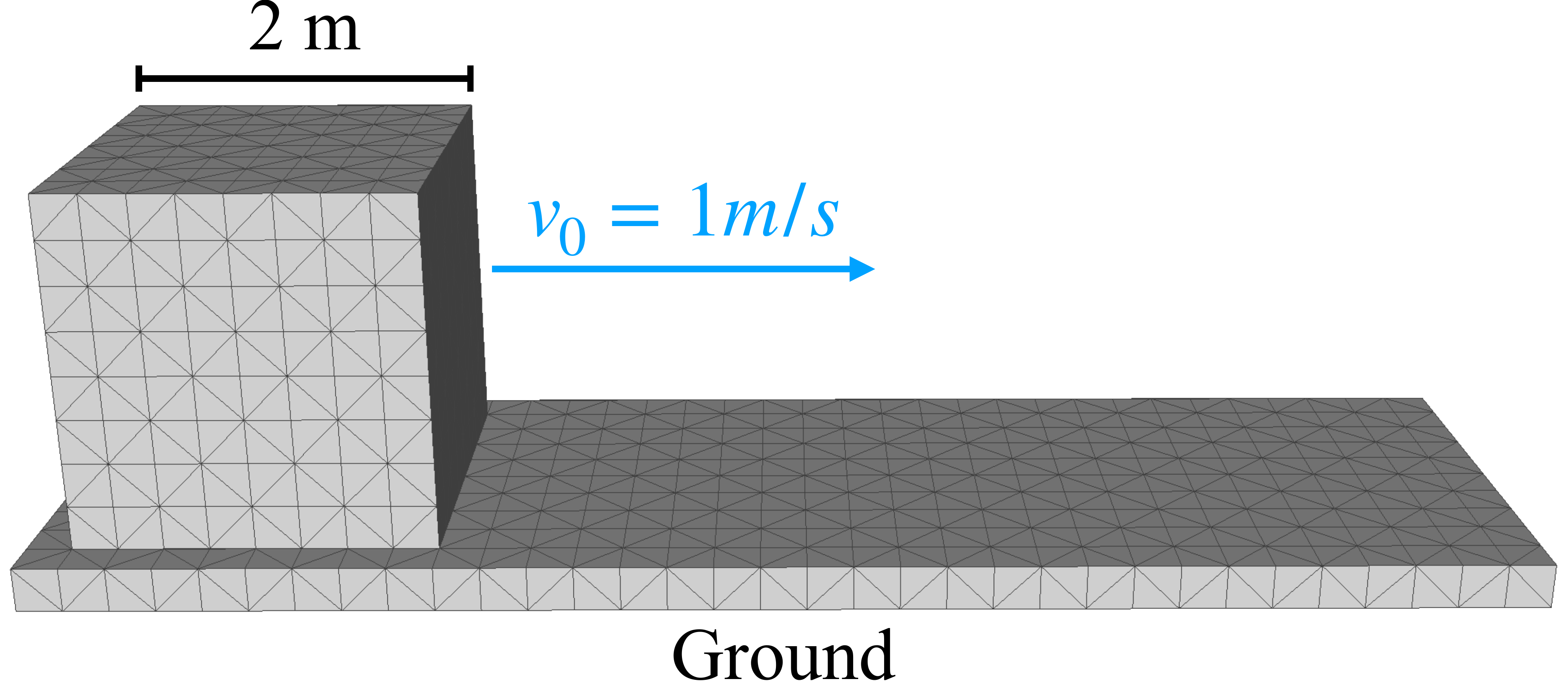}
    \includegraphics[width=0.29\linewidth]{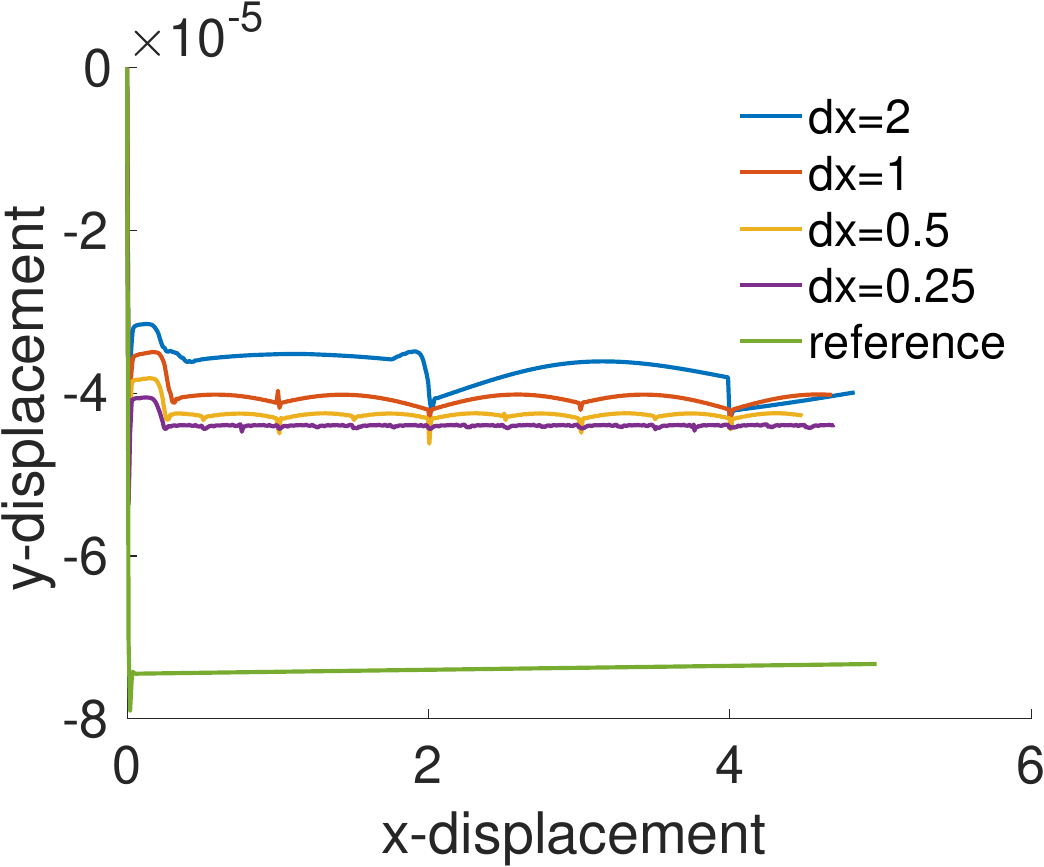}
    \includegraphics[width=0.29\linewidth]{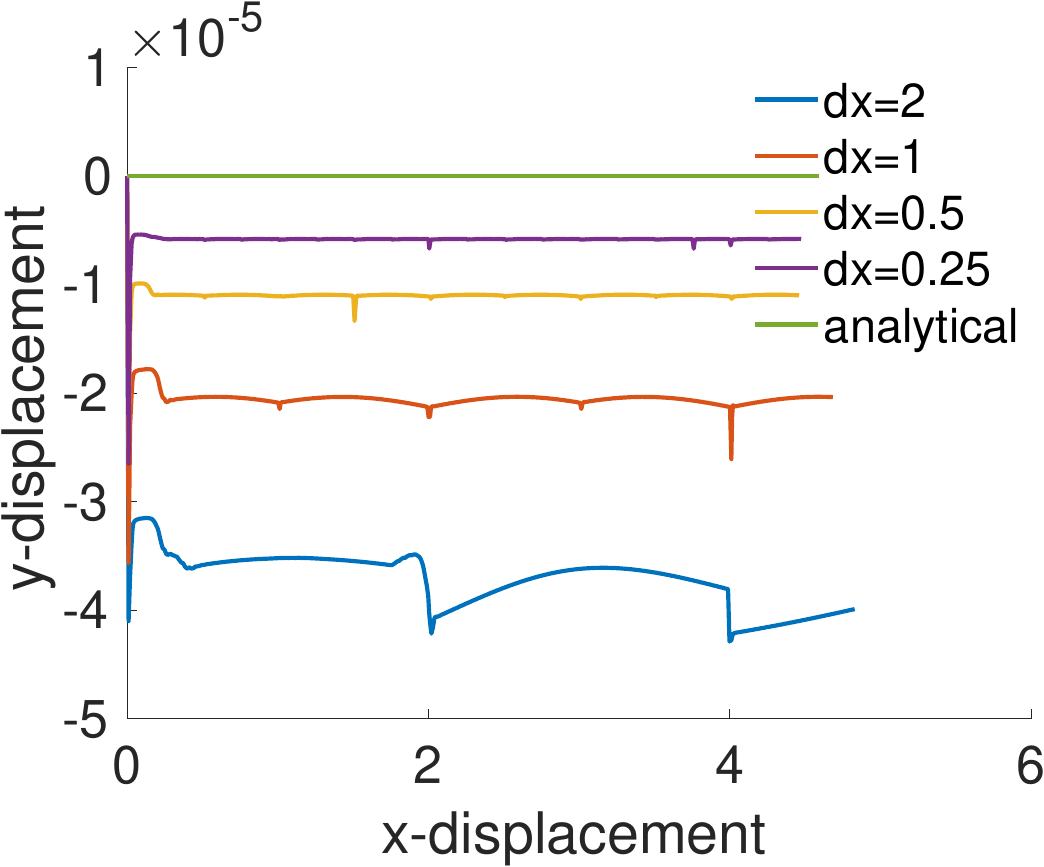}
    \caption{
    \textbf{Sliding block on meshed boundary (3D).}
    Experiment setup ($\Delta x=0.25m$, left) and displacement plots (center and right) of a block sliding on a meshed plane in 3D with the max operator approximated via our subtraction-based method. The center plot only refines $\Delta x$, and it converges to a straight line but not the reference line given by sliding on analytical plane. Both $\Delta x$ and $\hat{d}$ are refined in the right plot, which converge to the analytical solution $y=0$ at a rate of $0.8879$.}
    \label{fig:3d_sliding}
\end{figure}

\subsection{A Frictional Benchmark: Critical Angle on Slope}

To verify the accuracy of our friction model, an experiment with a stiff cube resting or sliding on a fixed analytical slope with a certain friction coefficient is created. 
When a rigid cube is placed on a slope with zero initial velocity, its acceleration has the following analytical form in the slope's tangent space:
\begin{equation}
    \mathbf{a} = \begin{bmatrix} g \min(0, \mu \cos{\theta} - \sin{\theta}) \\ 0 \end{bmatrix}.
\end{equation}
where $\mu$ is the friction coefficient between the cube and the slope, $g$ is the gravity acceleration, $\theta \in [0, \pi/4)$ is the inclined angle of the slope.

The initial configuration of this example is obtained by placing the cube $\hat{d}$ away from the slope, and then simulate under gravity ($g=\SI{5.099}{\meter/\second^2}$) with friction coefficient $\mu=0.5$ for $\SI{1}{\second}$ until the box becomes static. After obtaining the initial configuration, the slope test simulation is performed with different friction coefficients and the dynamic-static friction transition velocities $\epsilon_v$.

Here the cube is $\SI{0.02}{\meter}\times \SI{0.02}{\meter}\times \SI{0.02}{\meter}$, composed of just 8 nodes with density $\rho=\SI{1000}{\kilo\gram/\meter^3}$, Young's modulus $E=\SI{10}{\giga\pascal}$ and Poisson's ratio $\nu=0.4$. Slopes with friction coefficient $0.1$, $0.1999$, and $0.2$ have been tested (\cref{fig:3d_slope} 1st row), all with contact active distance $\hat{d}=10^{-5}m$, contact stiffness $\kappa=\SI{1}{\mega\pascal}$, static friction velocity threshold $\epsilon_v=10^{-5}m/s$, and with the lagged normal forces in friction iteratively updated until converging to a solution with fully-implicit friction. All simulations are using implicit Euler time integration with time step size $h=\SI{0.01}{\second}$, and the Newton tolerance is set to $\epsilon_d=10^{-8}m/s$.

With sliding velocity and acceleration of the cube's center of mass plotted over time (\cref{fig:3d_slope}), they have all been shown to well match analytical solutions with small absolute errors. Even for $\mu=0.1999$ ($99.95\%$ that of the critical coefficient), the sliding behavior can still be accurately captured. The dip is formed as our velocity-time curve started above the analytical one due to the mollification of static-dynamic transition, and then cross and goes below it. For $\mu=0.2$, it is also confirmed that the acceleration vanishes, and the velocity throughout the simulation is around $\epsilon_v$, the static friction velocity threshold in our approximation to provide the static friction force in the same magnitude as dynamic friction.

\begin{figure}[ht]
    \centering
    \parbox{0.48\linewidth}{
    \centering
        {\small Variable $\mu$ and fixed $\epsilon_v$}\\\vspace{5pt}
        \includegraphics[width=0.49\linewidth]{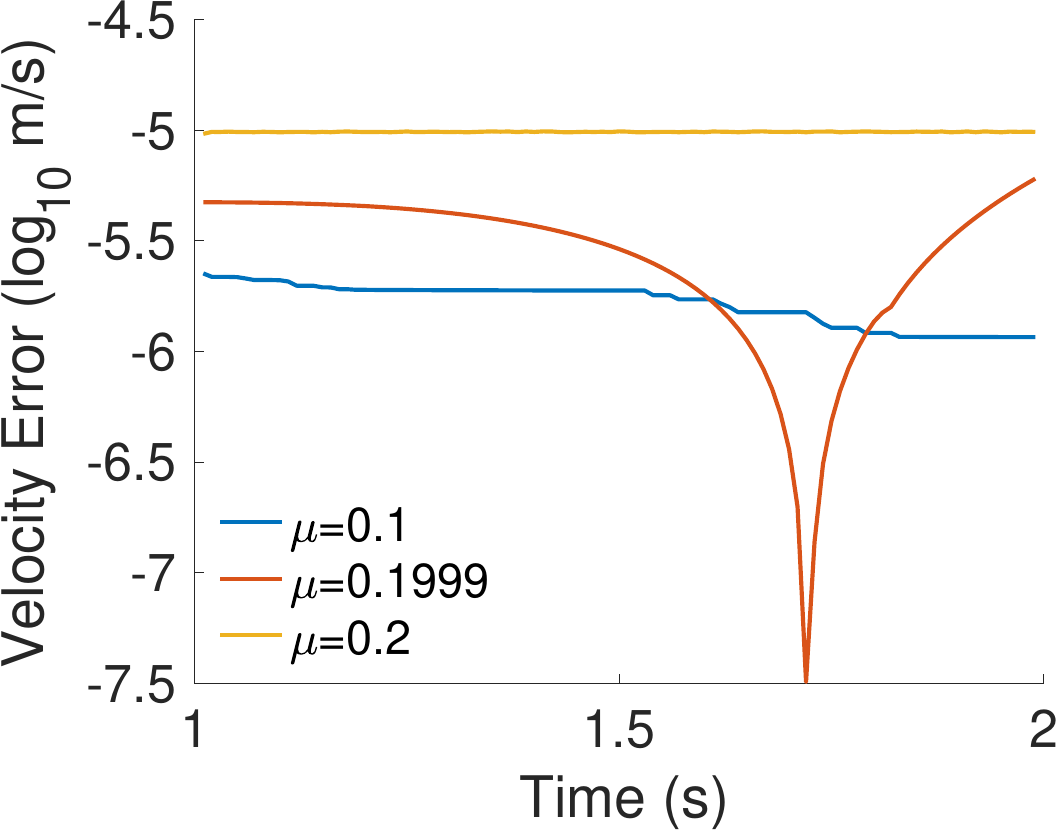}
        \includegraphics[width=0.49\linewidth]{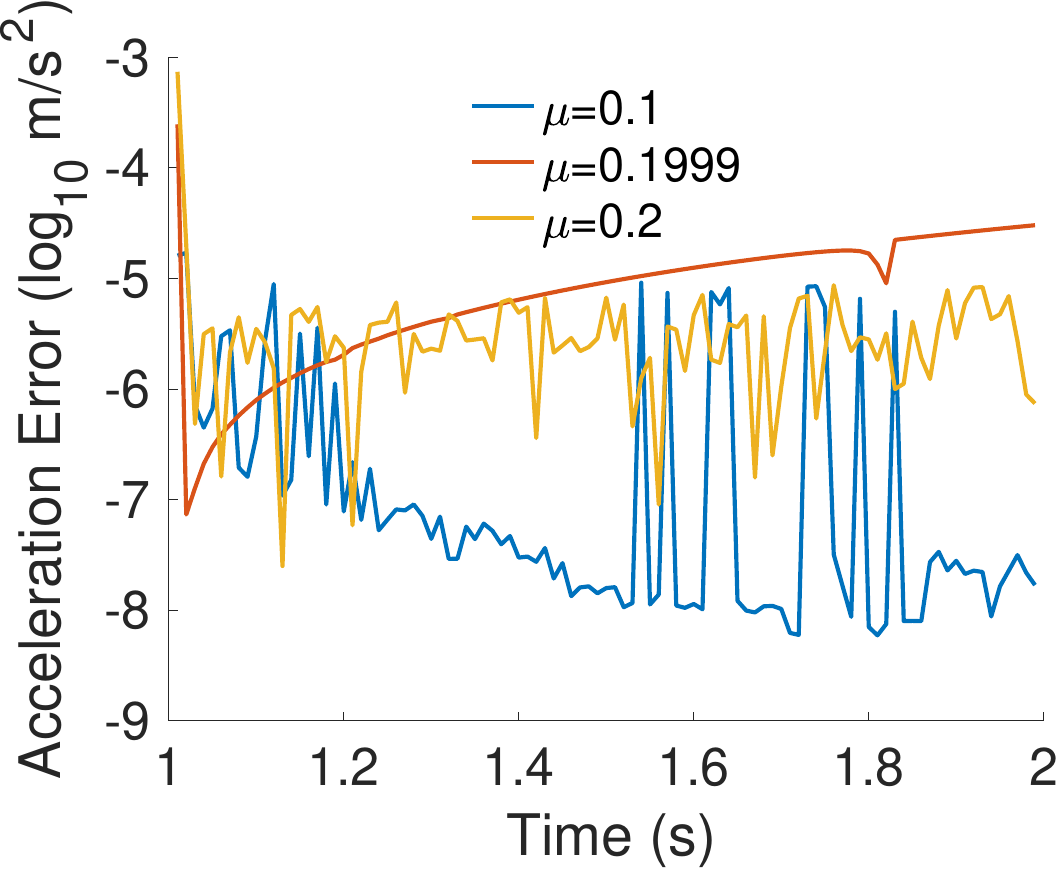}
    }
    \hspace{0.02\linewidth}
    \parbox{0.48\linewidth}{
        \centering
        {\small Fixed $\mu$ and variable $\epsilon_v$}\\\vspace{5pt}
        \includegraphics[width=0.49\linewidth]{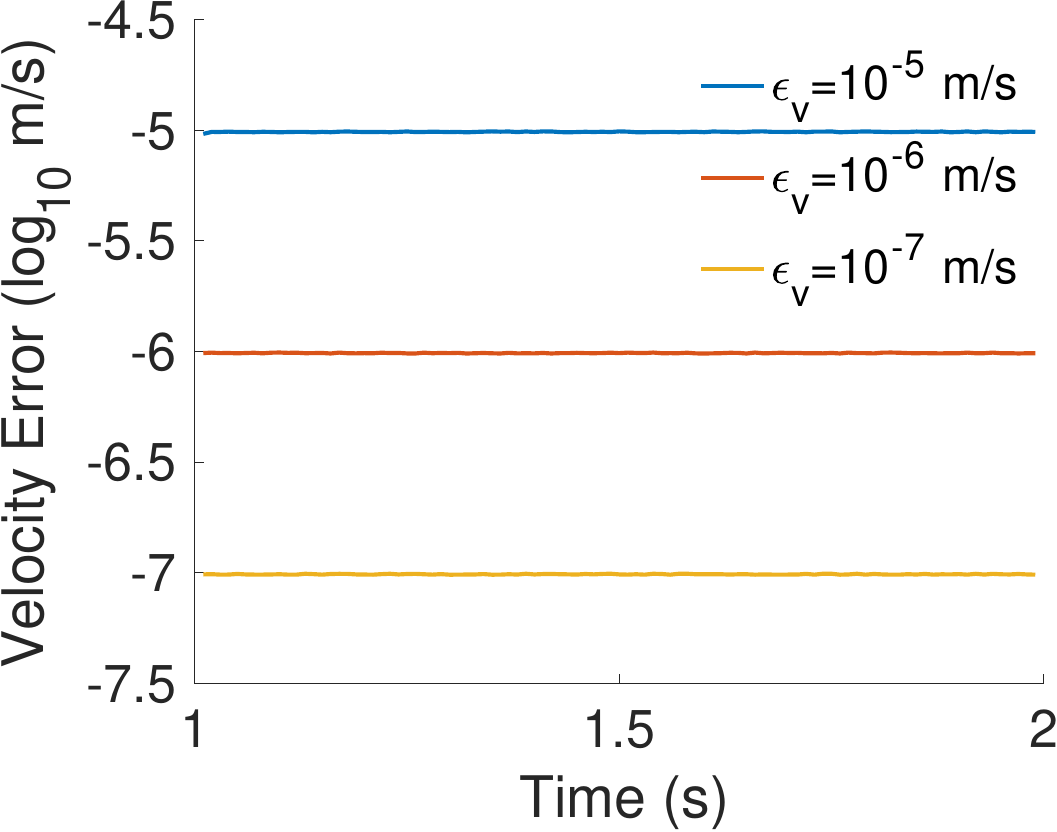}
        \includegraphics[width=0.49\linewidth]{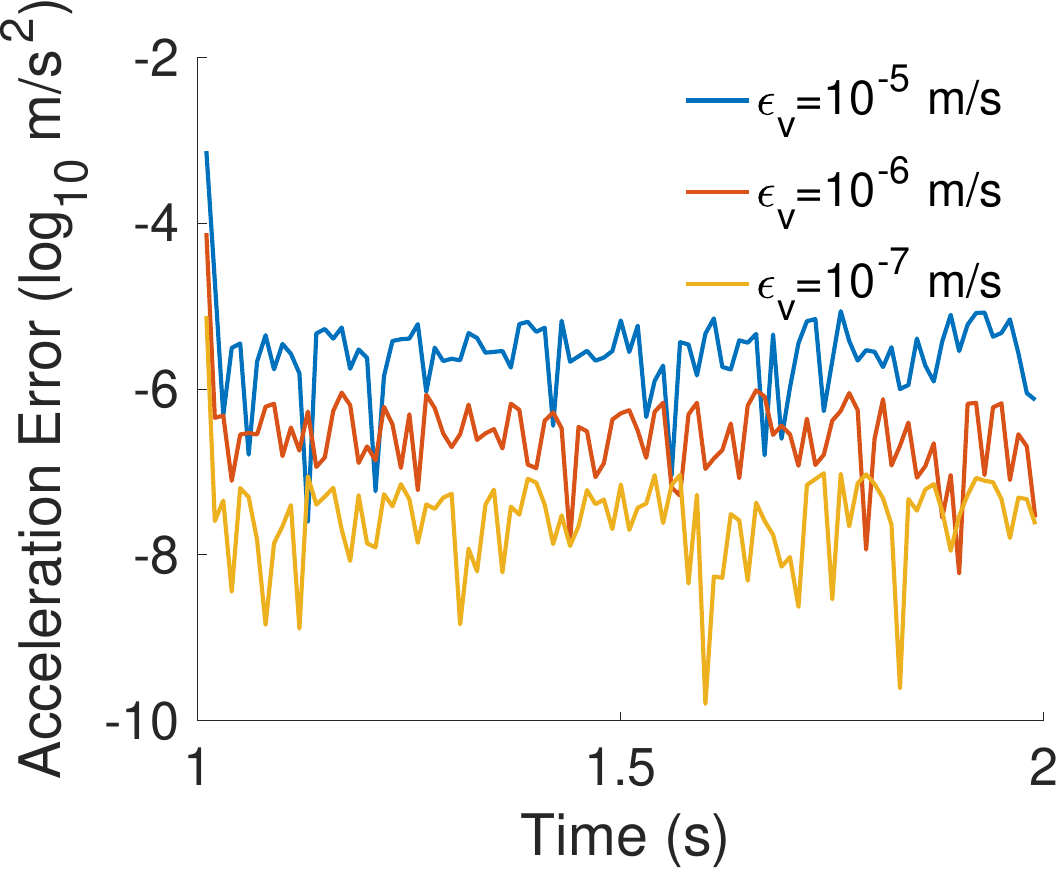}
    }
    \caption{
    \textbf{Critical angle on slope.}
    Velocity (column one and three) and acceleration (column two and four) error plots of the 3D slope test. On the left, $\mu=0.1$, $\mu=0.1999$, and $\mu=0.2$ (the critical value) are tested. On the right, we test convergence w.r.t. $\epsilon_v$ refinement with $\mu=0.2$, where both velocity and acceleration errors converge to $0$ with convergence rate $0.9959$ and $0.9926$ respectively.}
    \label{fig:3d_slope}
\end{figure}

However, notice that the error of acceleration is always much larger at the first time step than the latter steps after release. This is also an error introduced by our mollification of the dynamic-static friction transition. For $\mu=0.2$, the tangent velocity needs to increase immediately to nearly $\epsilon_v$ to obtain the static friction force for balance, which effectively ends up with a much larger acceleration than the solution ($\SI{0}{\meter/\second^2}$). By only refining $\epsilon_v$, at $\mu=0.2$ our velocity and acceleration errors (including the error in the first time step) both converges nearly linearly to the analytical solution (\cref{fig:3d_slope} 2nd row).

\subsection{Refinement in 3D with Self-Contact and Friction}

\begin{figure}
    \centering
    \includegraphics[width=0.24\linewidth]{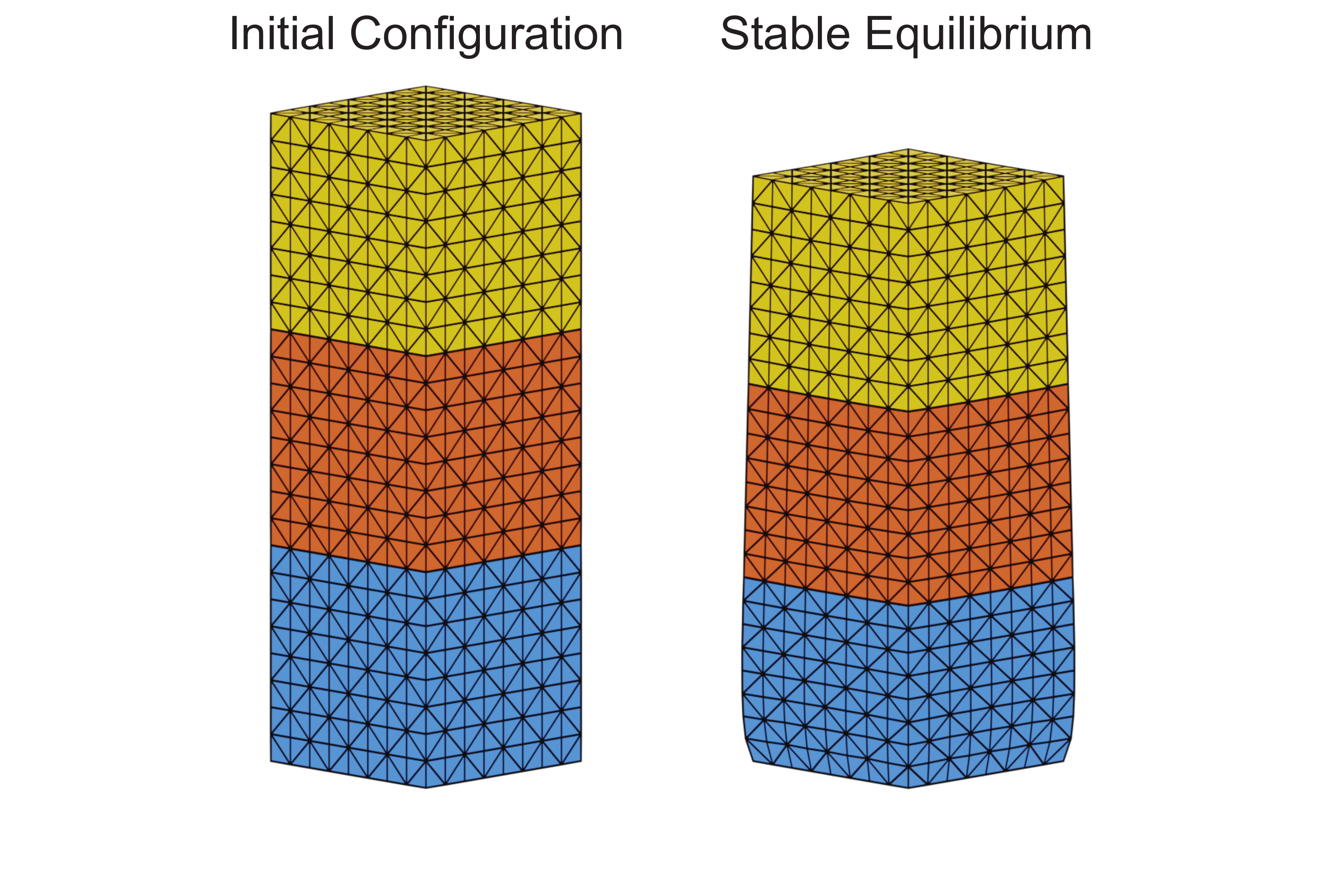}
    \includegraphics[width=0.24\linewidth]{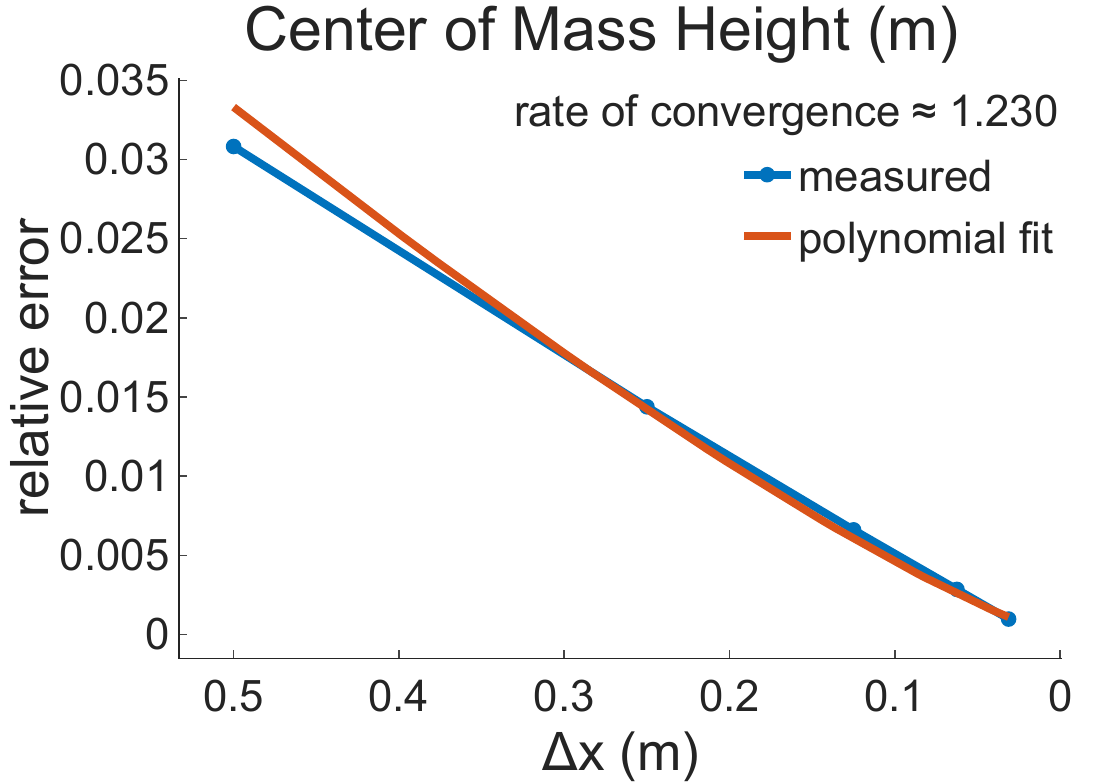}
    \includegraphics[width=0.24\linewidth]{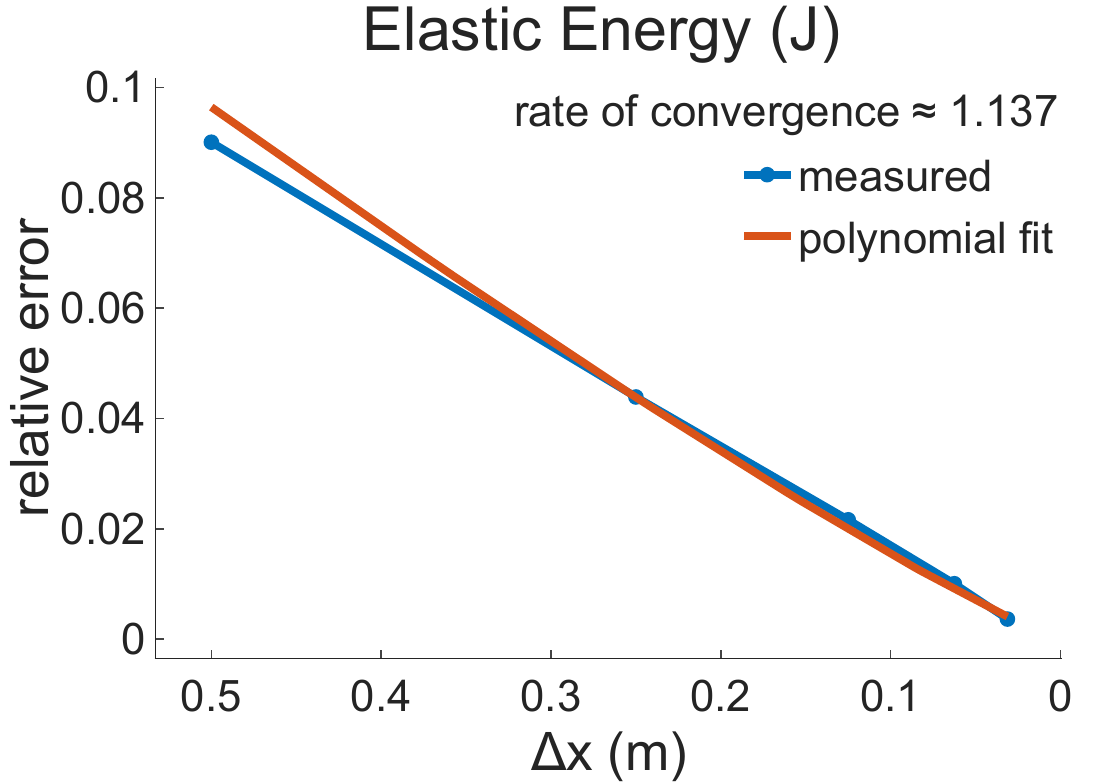}
    \includegraphics[width=0.24\linewidth]{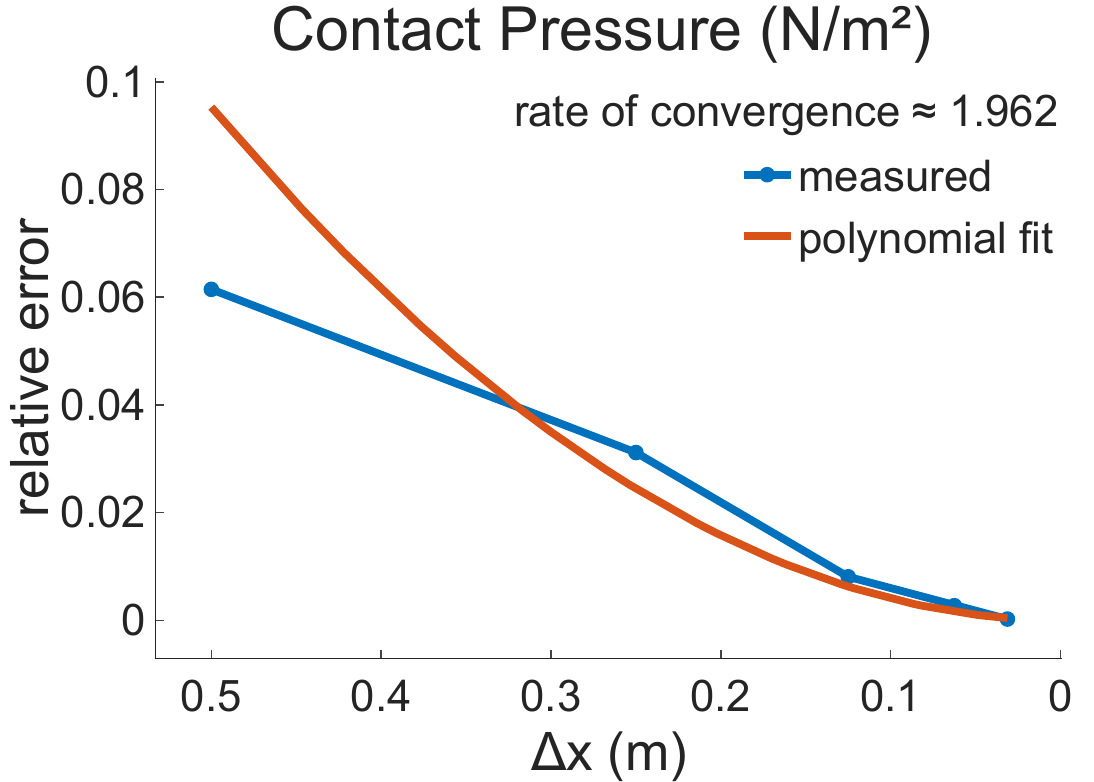}
    \caption{A stack of three cubes is simulated until equilibrium. We plot the relative error in the center of mass of the second block, elastic energy, and contact pressure at equilibrium for varying mesh resolutions. 
    We compute the relative error as 
    $\left|\frac{f(\Delta x) - f(\Delta x_{\text{ref}})}{f( \Delta x_{\text{ref}})}\right|$ 
    where 
    $\Delta x_{\text{ref}} = 2^{-7}~\si{\meter}$.
    }
    \label{fig:3D-cube-stack}
\end{figure}

We study the static equilibrium of a stack of three blocks under refinement. We choose material parameters ($E=\SI{2.5e5}{\pascal}$, $\nu=0.4$, and $\rho=\SI{1000}{\kilo\gram/\meter^3}$) which are able to reach stable equilibrium without the stack falling yet produces a visible deformation (see \cref{fig:3D-cube-stack}). Each block is $1\times1\times1~\si{\meter^3}$, and we use five levels of refinement $\Delta x=0.5,0.25,0.125,0.0625,\text{ and } 0.03125~\si{\meter}$. We set $\hat{d}=0.01\Delta x$ and create an initial gap of $0.9\hat{d}$. Additionally, we use friction with a coefficient of $\mu=0.5$ to prevent blocks from sliding off.

To solve for the static equilibrium we perform a series of incremental solves (time-stepping while zeroing out any velocity components at the start of every step). We perform this until convergence of the vertex positions (i.e. $\|x_i - x_{i-1}\|_\infty \leq \epsilon$). We choose a convergence tolerance of $\epsilon = \SI{1e-5}{\meter}$.

\Cref{fig:3D-cube-stack} shows the results of this study where we see convergence of the center of mass, elastic energy, and contact pressure.

\subsection{Dynamic Collision in 3D}

We setup a two spheres colliding experiment to test the ability of our method to resolve elasticity and kinetic energy transfer caused by high-speed collisions.
Two sphere meshes each with 29K nodes, \SI{0.04}{\meter} wide are placed \SI{0.08}{\meter} away from each other (\cref{fig:3d_2spheres_collide_sim} a and b), both with Neo-Hookean elasticity, Young's modulus $E=$\SI{1e7}{\pascal}, Poisson's ratio \SI{0.45}, and density \SI{1150}{\kilo\gram/\meter^3}, exactly the same material with the high-speed golf ball example in \citetIPC{}, except that we do not apply any damping here.
The two spheres are both with \SI{30}{\meter/\second} initial velocity towards each other. We set $\hat{d}$ to \SI{4e-5}{\meter} (0.1\% that of the sphere's diameter) and $\kappa$ to $0.1E$ as usual. For stability and accuracy, we apply BDF-2 time integration at time step size $h=$\SI{2e-5}{\second}.

\begin{figure}[ht]
    \centering
    \includegraphics[width=\linewidth]{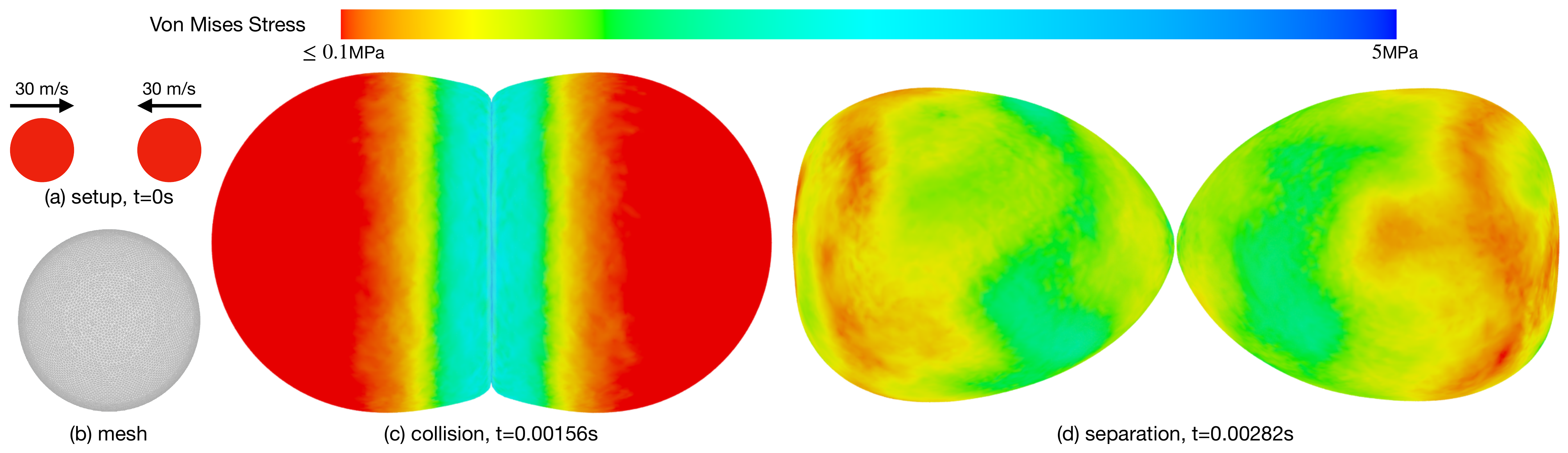}
    \caption{Two spheres colliding experiment setup and results with von Mises stress visualized.}
    \label{fig:3d_2spheres_collide_sim}
\end{figure}

During the simulation, the total x-direction momentum of the system is perfectly conserved with the the momentum of each sphere symmetrically and smoothly reversed (\cref{fig:3d_2spheres_collide_plot} right).
Since BDF-2 time integration is applied, the energy slightly dissipates around 10\% of the initial total energy (\cref{fig:3d_2spheres_collide_plot} left).
But nice symmetry and coherence on the elasticity and kinetic energy profile of the two spheres are accurately resolved. With the energy data of the left and right spheres plotted as curve and dots respectively, it is clear that the curves are well-aligned.
In \cref{fig:3d_2spheres_collide_sim} c and d, we visualize the Von Mises stress on the spheres at a collision state and right after separation.
Please see our supplemental video for the nice elastic wave propagation captured by our method.

\begin{figure}[ht]
    \centering
    \includegraphics[width=0.33\linewidth]{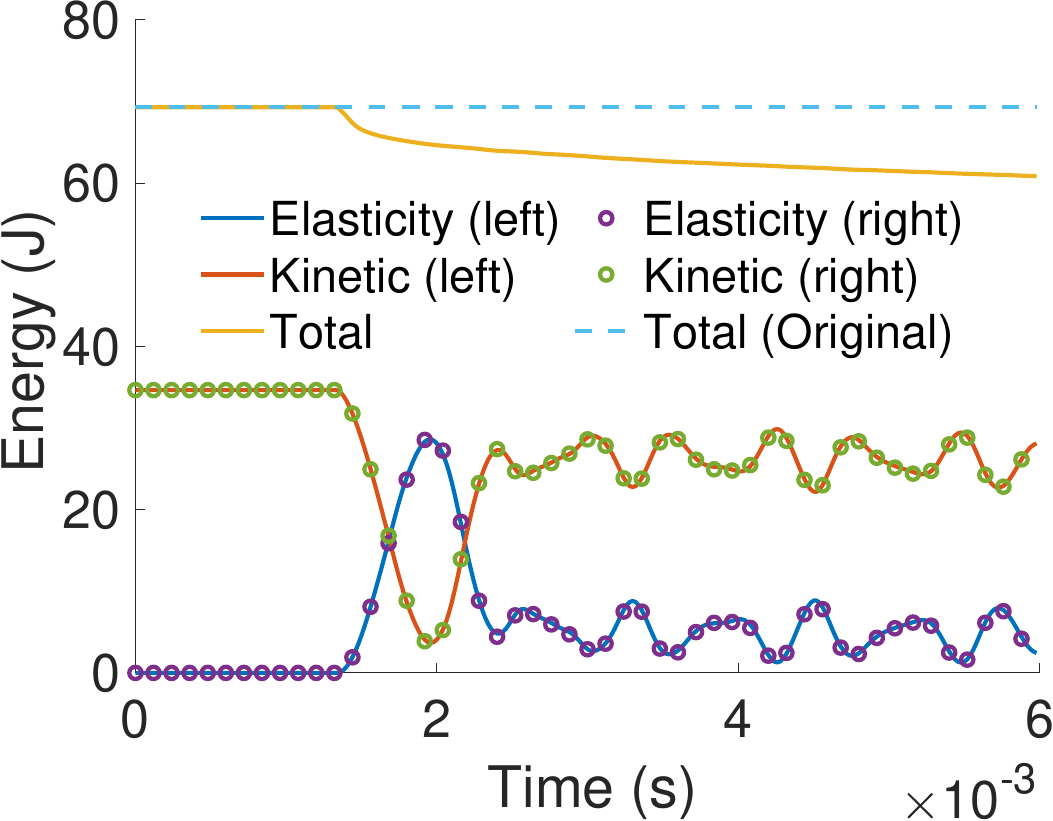} \hspace{0.2cm}
    \includegraphics[width=0.33\linewidth]{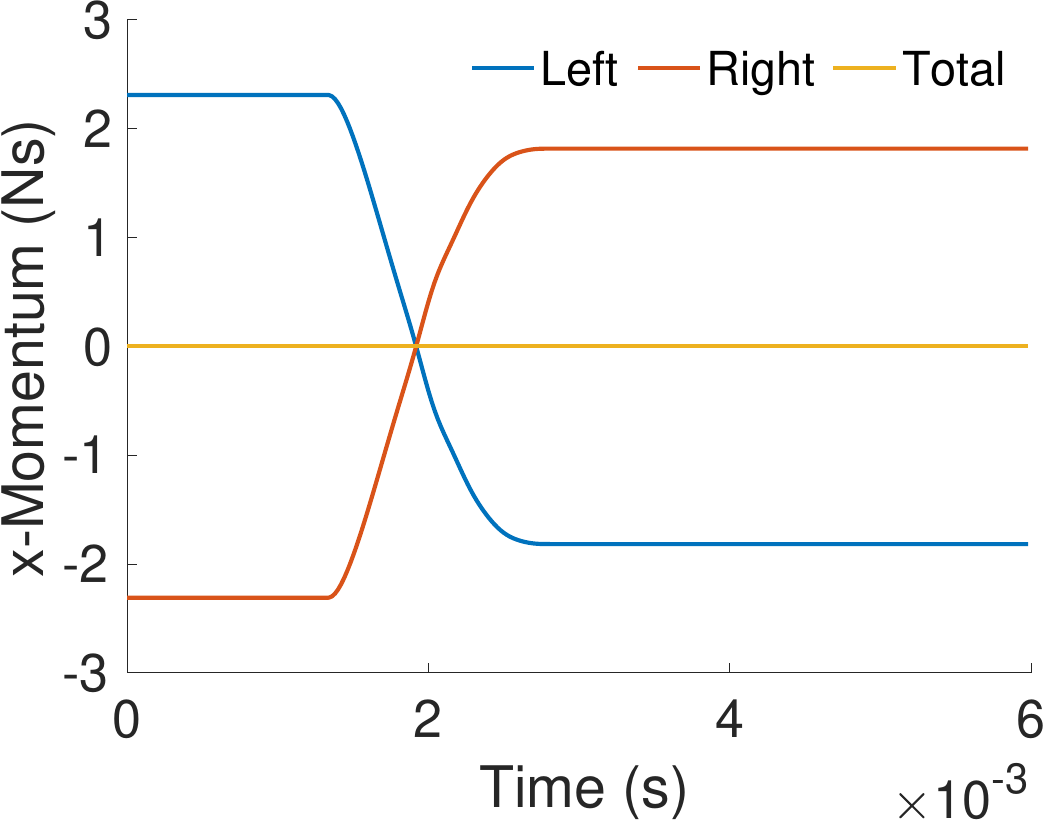}
    \caption{
    \textbf{Two spheres colliding.}
    Energy (left) and momentum (right) plots of the two spheres colliding experiment.
    The total energy is dissipated by approximately 10\% with BDF-2 time integration at $h=$\SI{2e-5}{\second}, 
    while the momentum is perfectly conserved throughout the simulation.
    The energy plot also shows nice symmetry and coherence on the elasticity and kinetic energy profile of the two spheres.}
    \label{fig:3d_2spheres_collide_plot}
\end{figure}

\subsection{3D Unit tests} 

\begin{figure}
    \centering
    \includegraphics[width=0.95\linewidth]{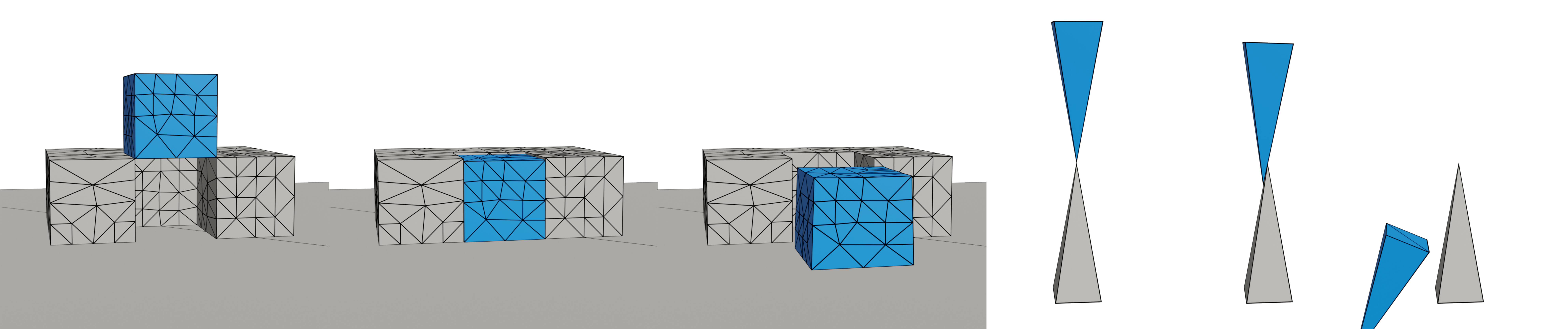}
    \caption{We demonstrate the ability of our method to handle both tight conforming contacts as well as sharp point-point contacts. Dynamic objects are colored in blue while static objects are colored in grey. Left: We drop a unit-sized cube into a static slot that is \SI{1e-5}{\meter} wider to than the cube. Right: We drop a dynamic spike onto a static spike where contact occurs between the spike points.}
    \label{fig:3D-unittests}
\end{figure}

We reproduce the unit tests presented by \citetIPC{} (\cref{fig:3D-unittests}). 

The first unit test tests the ability of our method to handle tight conforming contacts. We drop a unit cube into a C-shaped slot. The slot is only \SI{1e-5}{\meter} wider than the cube. We use a soft material ($E = \SI{1e6}{\pascal}$, $\nu = 0.4$, and $\rho = \SI{1e3}{\meter/\kilo\gram})$) and a $\hat{d}$ of \SI{1e-5}{\meter} with $\kappa=0.1 E$. We also utilize a framerate time step of \SI{0.04}{\second}. Our method passes this test without problem demonstrating the ability to handle small gaping and conforming contact.

Our second test positions two spikes such that they contact at the tips. This degenerate case is often challenging for traditional methods~\cite{erleben2018methodology}. We use the same material parameters as the first unit test and set $\hat{d} = \SI{1e-3}{\meter}$. Again we use a large $\Delta t = \SI{0.025}{\second}$. Our method has no difficulty in handling this contact, resolving the point-point contact into a downward diagonal motion.

\subsection{Application: Microstructures} 

\begin{figure}
    \centering
    \includegraphics[width=\linewidth]{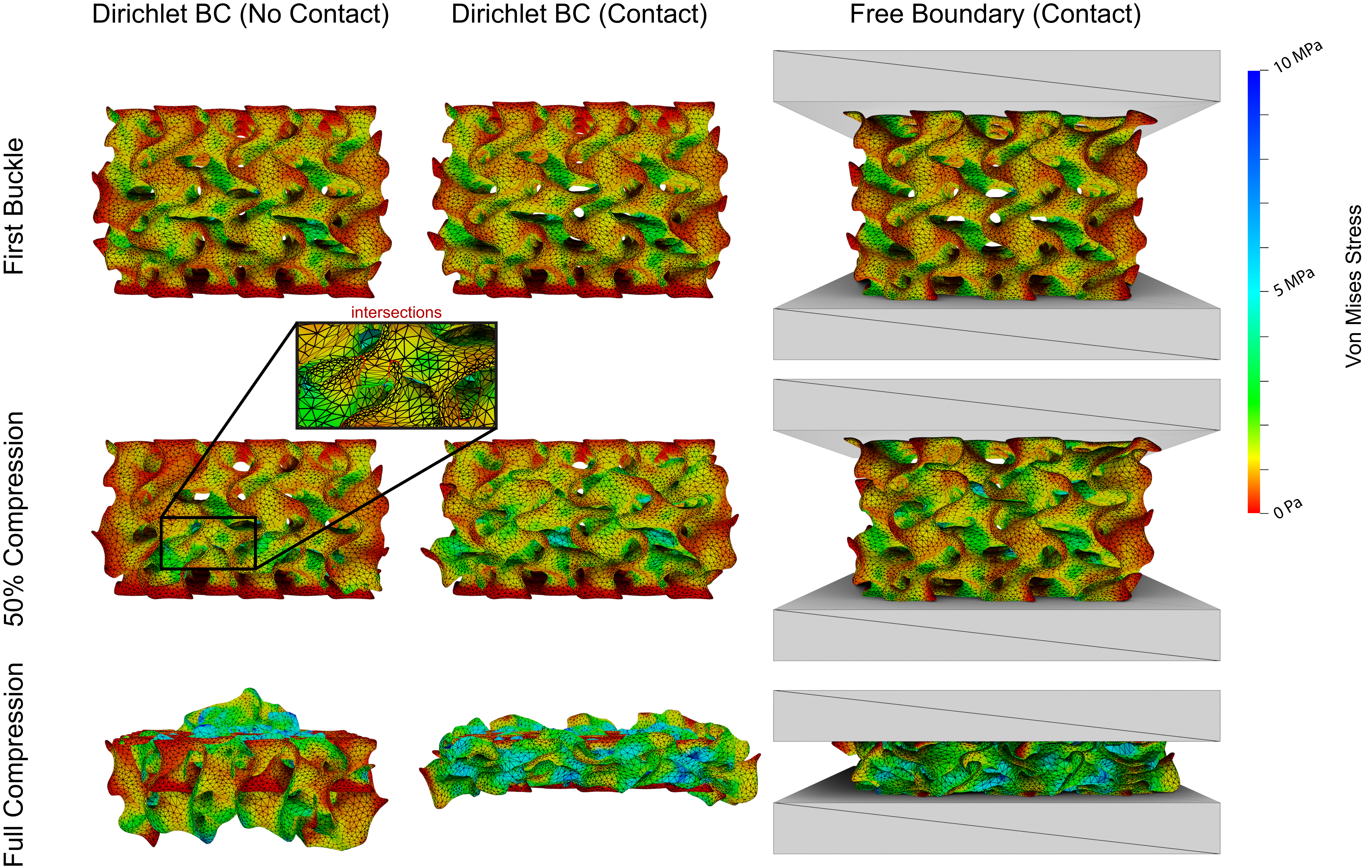}
    \caption{Compressing a gyroid microstructure with and without contact and two different boundary conditions. The von Mises stress is visualized. Without contact, large intersections are visible (highlighted in red in the inset) and the stresses are lower than when accounting for contact. The stresses are larger when using Dirichlet boundary conditions on the gyroid compared to having a free boundary. This is due to preventing the top and bottom of the gyroid from expanding in the plane orthogonal to compression.}
    \label{fig:microstructures}
\end{figure}

\begin{figure}
    \centering
    \includegraphics[width=0.75\linewidth]{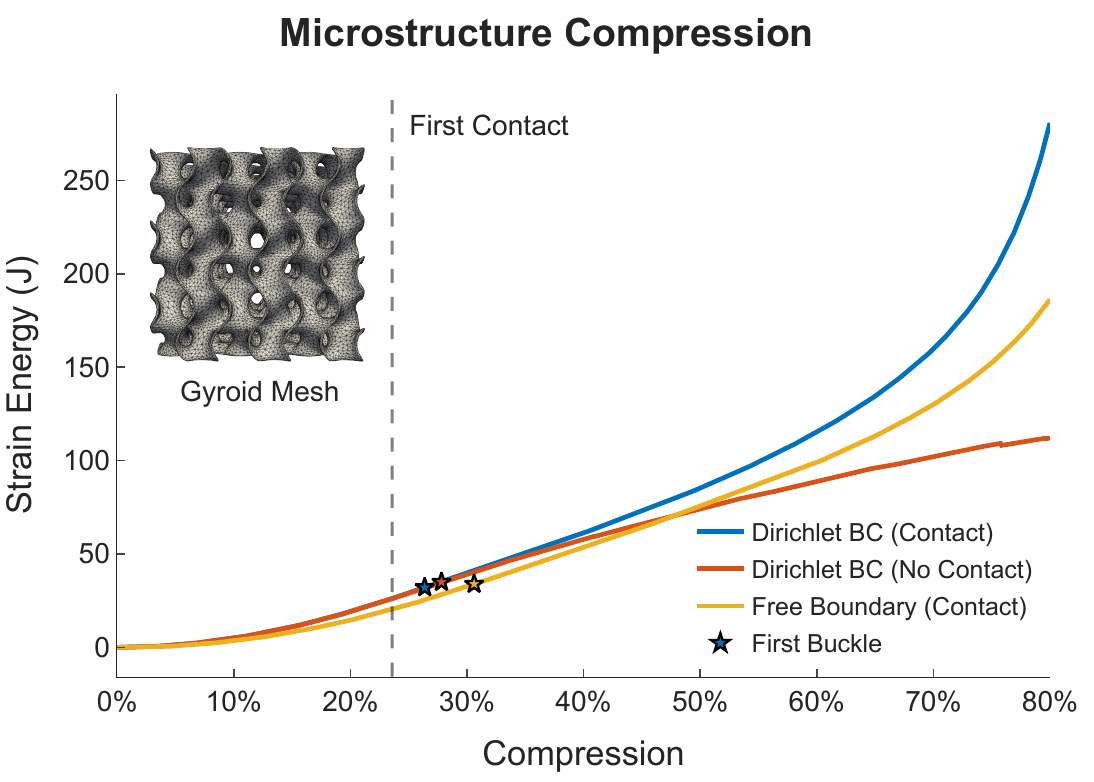}\\
    \caption{We compress a gyroid micro-structure using two different boundary conditions and one without contact. The Dirichlet boundary conditions simulations with and without contact lead to agreement in the measured strain energy up to around 26\% compression (soon after the first contact is detected). These plots quickly diverge for larger compression. When we apply the boundary conditions to a bottom and top plate and leave the gyroid's boundaries free (closely modeling a real world setup), it is clear the free boundary model's strain energy diverges with only ~15\% compression. This demonstrates the importance of modeling the deformation and contact of boundaries.}
    \label{fig:microstructures-plot}
\end{figure}

As an illustration of the importance of proper contact handling we apply our method to simulate the compression of a 3D printed gyroid micro-structure (\cref{fig:microstructures}). We perform three simulations: (1) Dirichlet boundary conditions on the gyroid to compress it \emph{without contact resolution}, (2) the same boundary conditions as (1) but now with contact modeled by our formulation, and (3) we apply the boundary conditions to rigid plates, leaving the gyroid's boundaries free (again with contact). For the material parameters we match those of an 3D printed elastomeric polyurethane ($E = \SI{9e6}{\pascal}$, $\nu = 0.48$, and $\rho = \SI{1.1e3}{\meter/\kilo\gram}$). We plot the strain energy versus compression in \cref{fig:microstructures-plot}. 

\bibliography{main}
\end{document}